\documentclass[12pt]{amsart}
\usepackage{amsmath} 
\usepackage{amssymb}
\usepackage{amsfonts}
\usepackage[T1]{fontenc}
\usepackage{ndstye}
\usepackage{mathtools}

\newcommand{\br}{\mathbf R}
\newcommand{\bz}{\mathbf Z}
\newcommand{\bn}{\mathbf N}
\newcommand{\Cal}{\mathcal}

\newcommand{\hess}{\operatorname{Hess}}

\newcommand{\id}{\operatorname{Id}}

\newcommand{\im}{\operatorname{Im}}

\newcommand{\Lim}{\operatorname{Lim}}

\newcommand{\ls}{\lesssim}

\newcommand{\mn}[1]{\Vert#1\Vert}

\newcommand{\op}{\operatorname{Op}}
\newcommand{\ol}{\overline}

\newcommand{\Ps}{\rm  ($\Psi$)\ }

\newcommand{\re}{\operatorname{Re}}
\newcommand{\restr}[1]{\big|_{#1}}

\newcommand{\set}[1]{\left\{\,#1\,\right\}}
\newcommand{\sett}{\Sigma _1}
\newcommand{\sgn}{\operatorname{sgn}}
\newcommand{\sh}{\sharp}

\newcommand{\st}{\Sigma _2}
\newcommand{\sub}{\operatorname{Sub}}
\newcommand{\supp}{\operatorname{\rm supp}}

\newcommand{\sw}[1]{\left ( #1\right ) }

\newcommand{\w}[1]{\langle #1\rangle }

\newcommand{\wf}{\operatorname{WF}}

\newcommand{\wt}{\widetilde}

\renewcommand{\sub}{\operatorname{Sub}}

\newcommand{\subr}{\operatorname{Sub_r}}
\renewcommand{\Lim}{\operatorname{Lim}}

\renewcommand{\Ps}{\rm  ($\Psi$)}

\newcommand{\wh}{\widehat}

\begin{document}

\hsize=160  mm
\vsize=250 mm

\baselineskip 20pt 
\lineskip 2pt
\lineskiplimit 2pt

\numberwithin{equation}{section}

\title[Sufficient Conditions ]{Sufficient Conditions for Solvability of  \\Operators of Subprincipal Type}
\author[Nils Dencker]{Nils Dencker\\ \\Lund University}
\address{Centre for Mathematical Sciences, University of Lund, Box 118, 221 00 Lund, Sweden}
\email{nils.dencker@gmail.com} 

\begin{abstract}
	In this paper we show that condition $ \sub_r(\Psi) $ on the subprincipal symbol is sufficient for local solvability of linear pseudodifferential operators of real subprincipal type. These are the operators having real principal symbol, which is of principal type and vanishes of second order on an involutive manifold where the subprincipal symbol is of principal type. The condition $ \sub_r(\Psi) $ is a condition on the sign changes of the imaginary part of the subprincipal symbol, which has been shown in \cite{de:lim} and \cite{de:sub} to be necessary for local solvability of linear pseudodifferential operators of  real subprincipal type. In the appendix, we study the local solvability of quasilinear second order partial differential operators of real principal type.
\end{abstract}


\maketitle

\thispagestyle{empty}

\section{Introduction}

In this paper we shall study the local solvability of classical pseudodifferential operators $ P \in \Psi^m_{cl}$, which are given by an asymptotic expansion $ p_m (x,\xi) + p_{m-1}(x,\xi)  + \dots $ of terms $ p_{m-j} (x,\xi) $ homogeneous  of degree $ m-j $ in $ \xi $ for $j \in \bn $, where $ p_m= p $ is the principal symbol. We are going to study operators which are not of principal type, i.e., when the principal symbol $ p  $ vanishes of at least order 2, in particular the sufficiency in the case when the principal symbol is real and has involutive double characteristics. 
But we will also assume that the operator is of subprincipal type, so that the subprincipal symbol of the operator is of principal type, see Definition~\ref{subprinctype}.

The definition that $ P$ is locally solvable at a compact subset of a manifold $K \subseteq X$ 
is that the equation 
\begin{equation}\label{locsolv}
Pu = v 
\end{equation}
has a local solution $u \in \Cal D'(X)$ in a neighborhood of $K$
for any $v\in C^\infty(X)$
in a set of finite codimension.  
We can also define microlocal solvability of~$P$ at any compactly based cone
$K \subset T^*X$, see Definition~\ref{microsolv}.

For classical pseudodifferential operators $ P \in \Psi^m_{cl}$ which are of principal type, local solvability is equivalent to  condition $ (\Psi) $ on the principal symbol $ p$, i.e.,
\begin{multline}\label{psicond} \text{$\im ap$
	does not change sign from $-$ to $+$}\\
\text{along the oriented
	bicharacteristics of $\re ap$ when $\re ap = 0$}
\end{multline}
for any $0 \ne a \in C^\infty(T^*X)$, see  \cite{de:nt} and  \cite{ho:nec}.
The oriented bicharacteristics are
the positive flow of the Hamilton vector field $H_{\re ap} \ne
0$ on which $\re ap =0$, these are also called
{\em semibicharacteristics} of~$p$.
Observe that if condition  $ (\Psi) $ is satisfied on a set, then it is trivially satisfied on any subset.

For operators which are not of principal type, the  invariant {\em subprincipal symbol}
\begin{equation}\label{psubdef}
p_s = p_{m-1} + \frac{i}{2} \sum_{j} \partial_{x_j}\partial_{\xi_j} p 
\end{equation}
becomes important. There are several conditions corresponding to 
condition $ (\Psi) $ on the subprincipal symbol, several necessary conditions for solvability are known, but not many sufficient conditions.

One of the earliest results are by Mendoza and
Uhlman~\cite{MU1}, who studied the case when principal symbol is equal to a product $ p = p_1p_2$ with $ p_j $ of real principal type with linearly independent differentials $ dp_1 $ and $ dp_2 $. Thus the double characteristic set $ \st = \{ p_1 = p_2 = 0\} $ is a intersection of two transversal hypersurfaces. In this case, they proved that $P$ is {\em not solvable} if the imaginary part of the subprincipal symbol $ p_s $ changes sign on the bicharacteristics of $ p_1 $ or $ p_2 $ on $ \st $.  These are the limits of the characteristics of the principal symbol at the double characteristic set $ \st$.
They proved in \cite{MU2} that $ P $ is solvable if the imaginary part of the subprincipal symbol $ p_s $ does not vanish on the double characteristics, thus there are no sign changes.

Mendoza~\cite{Men}
generalized the necessary condition to the case when the principal symbol is real, vanishes of second order on an
involutive submanifold where it has an indefinite Hessian with rank equal to the
codimension of the manifold. Then Hessian gives well-defined limit
bicharacteristis on the submanifold, and $P$ is
not solvable if the imaginary part of the subprincipal symbol changes sign on any of these
limit bicharacteristics. 

There are several other necessary condition for solvability of operators that are not of principal type corresponding to condition \Ps\  on operators of principal type. The following one generalizes Mendoza's and Uhlmann's necessary conditions for solvability. 

\begin{exe}\label{ex1} \rm 
	A necessary condition on the subprincipal symbol for solvability of operators with real principal symbol $ p $ vanishing of at least second order on an
	involutive submanifold $ \st $ is condition $\Lim(\Psi)  $:
	\begin{equation}\label{limp0}
	\im p_s \ \text{does not change sign from $-  $ to $ + $ on the limit bicharacteristics of $ p $ on $ \st$}
	\end{equation}
	which follows from the necessary condition (2.9)  of  \cite{de:lim}. 
\end{exe}

This condition is invariant under symplectic changes of variables and multiplication with nonvanishing real factors, since a negative factor changes both the direction of the limit bicharacteristic and the sign of the imaginary part. Thus, it is invariant under conjugations of the operator with Fourier integral operators with real principal symbols.
 
Observe that this is a condition on the sign changes of $ \im p_s $ at   a  (possibly empty) subset of directions on the leaves of~$ \st  $. The sufficient conditions that we are going to use will exclude any sign changes  of $ \im p_s $ on the leaves of~$ \st  $. 
But even this stronger condition is not sufficient, one also needs conditions on the imaginary part of the subprincipal symbol in the direction of the bicharacteristics if the real part of the subprincipal symbol. For that, the operator has to be of {\em subprincipal type}, which means that the subprincipal symbol is of principal type, with Hamilton vector field that is tangent to $ \st $ at the characteristics, see Definition~2.2  in \cite{de:sub} or Definition~\ref{subprinctype} in this paper.

\begin{exe}\label{ex2}  \rm
	A necessary conditions for solvability for operators of subprincipal type for involutive $ \st $ is given by Definition~2.4  in \cite{de:sub}. It is condition $ \sub(\Psi) $ which is  Sub($ \Psi $)  on  subprincipal symbol $ p_s$  on $ \st$:  
	\begin{multline}\label{subpsi0}
	\text{$ \im ap_s $ does not change sign from $ - $ to +} \\ \text{on the oriented bicharacteristics   of $ \re ap_s$ when $ \re ap_s = 0$  on $ \st$.}
	\end{multline}	
	for any $0\ne  a \in C^\infty$. It is known that this condition is invariant under symplectic changes of variables and multiplication with nonvanishing factors, since the subprincipal symbol then only get multiplied with these factors.
\end{exe}

Observe that condition \eqref{subpsi0} is empty if $ p_s $ vanishes of second order, so for this condition we need that the operator is of subprincipal type, see Definition~\ref{subprinctype}.

A stronger necessary condition for solvability involves the sign changes on the imaginary part of the subprincipal symbol on a larger set of curves on the double characteristic set, actually on  the limits of the bicharacteristics of the real part of the {\em refined principal symbol}
\begin{equation}\label{prefdef}
p_r= p + p_{s} 
\end{equation}
where $ p_s $ is the subprincipal symbol given by \eqref{psubdef},  see Theorem 18.1.33 in~\cite{ho:yellow}.
This symbol is invariant under conjugation with elliptic Fourier integral operators, and multiplication with $ a(x,D) \in \Psi^0 $ gives the refined principal symbol $ a p_r $ modulo terms in $ S^{m-1} $ vanishing of order 1 on $ \st $. In fact, the refined principal symbol of $a(x,D)P(x,D)  $ is $ a p_r + \tfrac{i}{2}H_p a $ modulo $ S^{m-2} $.

\begin{exe} \label{subk}
	A  necessary condition for solvability for operators of subprincipal type with principal symbol vanishing of second order on $ \st $ is  $\sub_2(\Psi) $. This condition
	is given by Definition~2.6  in \cite{de:isaac} and is condition \Ps \ on the  symbol 
	\begin{equation} \label{extsub}
	p_{s,2} = J^2(p) + J^0(p_{m-1})   = J^2(p) +  p_s 
	\end{equation} 
	where $J^2(p)$ equal to the 2:nd  jet of  $ p  $ at $ \st  $. 	
\end{exe}

This condition is invariant under symplectic changes of variables and multiplication with nonvanishing factors by Remark~2.3 in \cite{de:isaac}, since then $ p_{s,2} $ gets multiplied with a nonvanishing factor.

Observe that this definition gives conditions on the sign changes of  $ \im p_{s,2} $ on the limits of the bicharacteristics of $ \re p_{s,2} $ at $ \st $, which are the limits of  of the bicharacteristics of $ \re p_r $, see~\eqref{Hrepr}.
Condition $\sub_2(\Psi) $ gives \eqref{limp0} and \eqref{subpsi0}, but  the directions of the limit characteristics  depend on the sign of $ \re p_s $, see the following example. 

\begin{exe}
	If $ p = | \xi'|^2 - | \xi'' |^2$ with $ (\xi',\xi'')  \in \br^n \times \br^m$, then 
	$$ 
	H_p = 2 (\xi'\cdot \partial_{x'} - \xi''\cdot \partial_{x''}  ) = 2 |\xi | (\theta'\cdot  \partial_{x'} - \theta'' \cdot \partial_{x''}) \qquad | \theta | = 1
	$$ 
	which gives all directions in $ x $ when $ \xi \to  0 $.
	If we take the limit only  when $ p_r = 0 $, i.e., when
	$ p = -\re p_s $, then we get the limit bicharacteristics $\theta'\cdot  \partial_{x'} - \theta'' \cdot \partial_{x''}  $ with $ | \theta'| > | \theta''|$  when $ \re p_s  < 0 $, the ones with $ | \theta'| < | \theta''|$ when $\re p_s  > 0 $ and all directions in $\br^n\times \br^m  $ when $ \re p_s = 0 $.
\end{exe}

Thus, when $ \re p_s = 0$ we may obtain  that the sign of $ \im p_{s,2} = \im p_s$ is constant on the leaves of $ \st $ when  $ \re p_s = 0$, but by Example~\ref{cexample} that is not enough to get solvability. Observe that we shall assume that $ \re p_r $ is constant on the leaves by~\eqref{tangHam}.

Also observe that the necessity of the conditions in Examples~\ref{ex1}--\ref{subk} only hold under some additional conditions on the symbol, for example finite order of the sign change. For the sufficiency, it is not enough that  $\sub (\Psi) $  holds when $ p_s = 0 $, by the following example. 

\begin{exe}\label{cexample}
	Consider the PDO
	\begin{equation}
		P_\pm = (1 + t^2)\Delta_x + D_t \pm i t D_t
	\end{equation}
with symbol $(1 + t^2) | \xi |^2  + \tau \pm i t \tau = (1 + t^2) | \xi |^2  + (1  \pm i t )\tau $. This operator satisfies the condition $\Lim (\Psi) $ by~\eqref{limp0} and $ \sub ( \Psi ) $ by~\eqref{subpsi0} since $ \im p_s  = \pm t \tau = 0$  when  $ \re p_s= \tau = 0$. But multiplication by $ (1 \pm it)^{-1} = (1 \mp it)/(1 + t^2)$ gives the operator $D_t +  (1 \mp it) \Delta_x$, and conjugation with the Schr\"odinger kernel $ \exp(\pm it \Delta_x) $ gives the operator  $ Q_\mp = D_1 \mp it \Delta_x $. Here $ Q_+ $ is the Mizohata operator which is a standard example of an unsolvable operator, and $ Q_- $ is solvable, since   $ u(t,x) =i \int_0^t \exp((s^2 - t^2)\Delta_x/2) f(s,x)\, ds$ solves $ Q_-u = f \in C^\infty_0$.
\end{exe}

Observe that the condition $ \Lim(\Psi) $ in Example~\ref{ex1} does in general not imply that $ \im p_s $ has constant sign on the leaves of $ \st $, see the following example.

\begin{exe}\label{cbexample}
	If the principal symbol of $ P $ is $ D_{x_1}D_{x_2} $, then the leaves $ L $ of $\/ \st $ have dimension~$ 2 $. Divide a leaf $ L \cong \br^2 $ into a checkerboard and index the squares with $ (j,\, k) \in \bz^2 $. Denote the squares with  index $ (2j,2k) $ with $ S_+ $  and the ones with index $ (2j + 1, 2k+1) $ with $ S_- $ and the rest with $ S_0 $. 
	If  $ \im p_s  > 0$ in the interior of the squares in $ S_+ $, $ \im p_s  < 0$ in the interior of the squares in $ S_- $ and $ \im p_s  = 0$ on the squares in $ S_0 $, then $ \im p_s $ has constant sign along any $ x_1$ and $ x_2 $ lines, but not on the whole plane.

 \end{exe}

The conditions on $ P\in \Psi^m_{cl} $ in the paper will be the following. Let $ p $ be the real principal symbol, $ \Sigma = p^{-1}(0) $ be the characteristics and $ \st = \set{p = |dp| = 0}$ be the double characteristics. We assume that $ \st$ is a nonradial involutive submanifold and that $ p$ is real and vanishes of exactly second order at $\st$ so that
\begin{equation}\label{hessp}
	\text{$ \hess p $ is {\em   nondegenerate} on $ \st$}
\end{equation}
This implies that $ p$ is of real principal type on $ \sett = \Sigma \setminus \st$ in a sufficiently small conical neighborhood of  $ \st $, since  it  cannot vanish of second order on $\Sigma_1 $. Also, $ \hess p \restr{\st} $ has locally constant rank and index. That $ \st $ is nonradial means that if a function vanishes on $ \st $ then its Hamilton vector field does not have the radial direction on $ T^* X $, which is a generic condition.

\begin{rem}
	The invariant condition is that $ p$ is proportional to a real function. This means that the quotient $q =\im p/\re p $, which is defined where $ \re p \ne 0 $, can be extended to a $ C^\infty $ function with values on the extended real line $ \overline \br$, i.e., either $ q $ or $ q^{-1} $ is smooth. In fact, if \/ $ q \in C^\infty$  then $ p = (1 + iq) \re p$ and if  $ q^{-1} \in C^\infty$ then $ p = (q^{-1} + i) \im p$.
\end{rem}

In the following, we will assume that $ P$ is on the form so that the principal symbol $ p$ is real valued.
In the case when $ p $ is not proportional to a real function,  condition $ (\Psi) $ on the principal symbol has to be satisfied  on $\Sigma_1 $ since it is then necessary for solvability.

Recall that the subprincipal symbol 
\begin{equation} \label{subsymbol}
p_{s} = p_{m-1} + \frac{i}{2}\sum_{j}^{}\partial_{x_j}\partial_{{\xi}_j}p
\end{equation} 
is invariantly defined on $ \st $ under conjugation with elliptic Fourier integral operators. In fact, $ p_s $ is the value of the Weyl symbol of $ p + p_{m-1} $ at $ \st $ modulo $ S^{m-2} $, see~\cite{ho:weyl}.

\begin{rem}\label{subinv}
	When $ \st $ is involutive  we may choose symplectic coordinates so that $ \st = \set{\xi_1 = \dots = \xi_k = 0}$  and then the subprincipal symbol $p_s = p_{m-1}$  at $ \st$.
\end{rem}

In fact, since $ \partial_x\in T \st $ we find that  $ \partial_x p $ vanishes of second order on $\st $. If  $C$ is a pseudodifferential operator with principal symbol  $c = {\sigma}(C)$, then the value of the subprincipal symbol of the composition $CP$ is equal to $cp_{s} + \frac{i}{2}H_{p} c=  cp_s$ on $ \st  $.
Observe that the subprincipal symbol is complexly conjugated when taking the adjoint of the operator, see \cite[Theorem~18.1.34]{ho:yellow}. 

Since we shall assume that $ p $ is real, the real and imaginary parts of  $ p_s $  are invariant under multiplication with elliptic pseudodifferential operators and  conjugation with elliptic Fourier integral operators if these operators have  real principal symbols. In order to study the invariants, we need some symplectic concepts.

The symplectic annihilator to a linear space consists of  the vectors that are symplectically orthogonal to the space. 
Let  $T{\Sigma}_2^{\sigma}$ be the symplectic annihilator to $T
{\Sigma}_2$, which spans the symplectic leaves of $\st$. If $\st =
\set{\xi = 0}$, 
$(x,y) \in \br^{d}\times \br^{n-d}$, then the leaves are
spanned by  $\partial_{x}$.  Let
\begin{equation}
T^{\sigma}{\Sigma}_2 = T{\Sigma}_2/T{\Sigma}_2^{\sigma}
\end{equation}
which is a symplectic space over ${\Sigma}_2$. In these
coordinates it is parametrized by
\begin{equation}\label{tsigmadef}
T^{\sigma}\st =  
\set{((x_0,y_0; 0,\eta_0); (0,y;0,\eta)) \in T \st:\ (y,{\eta}) \in T^*\br^{n-d}}
\end{equation} 
Thus the fiber is isomorphic to the symplectic manifold $T^*\br^{n-d} $ with $ (x_0,y_0; 0,\eta_0) = w_0 \in \st $ as a parameter.

\begin{defn}\label{subprinctype}
	If the principal symbol is real valued, then we say that the operator $P$ is of \emph{real subprincipal type} if the following conditions hold:
	\begin{equation}\label{tangHam}
	 H_{\re p_{s}} \subset T\st  
	\end{equation}
	which means that $ d\re p_s \restr {TL} = 0 $, and
	\begin{equation}\label{subprinc}
	d \re p_{s}\restr{T^{\sigma}{\Sigma}_2} \ne 0  \qquad \text{}
	\end{equation}
	so  that $H_{\re p_{s}}$ is transversal to the leaves and we shall assume that it does not have the radial direction. 
	The bicharacteristics of $\re p_{s} $ with respect to the symplectic structure of $  \st $ are called the {\em subprincipal bicharacteristics} for any value of $ \re p_s $.
\end{defn}

This definition is invariant under symplectic changes of variables and but {\em not} by multiplication with nonvanishing real factors when $ \re p_s \ne 0 $. But it is invariant by multiplication with nonvanishing real factors that are constant on the leaves of $ \st $.
When the coordinates are given as in~\eqref{tsigmadef}, we find from \eqref{tangHam} that $ \partial_x \re p_s  = 0 $ on $ \st $ and  from \eqref{subprinc} that $  \partial_\eta \re  p_s \ne 0  $ or $ \partial_{y} \re p_s  \not \parallel\eta $ on $ \st $. Thus it follows that
\begin{equation}\label{resubprinc}
\re p_s \ \text{is of real principal type}
\end{equation}
and thus has simple zeroes. In  \cite[Definition~2.1]{de:sub} the definition was that  $ P$ is of subprincipal type if Definition~\ref{subprinctype} hold with $ \re p_s $ replaced with $ p_s $, but only when $ p_s = 0 $, which is invariant under multiplication with nonvanishing factors and symplectic changes of variables. But in that case the principal symbol may not be proportional to a real symbol ant then $ \re p_s $ is not well defined.

We shall study the microlocal solvability of the operator $ P$, which is
given by the following definition from  \cite{ho:yellow}. Recall that $H^{loc}_{(s)}(X)$ is
the set of distributions that are locally in the $L^2$ Sobolev space
$H_{(s)}(X)$.

\begin{defn}\label{microsolv}
	If  $ P\in \Psi^m_{cl} $ and $K \subset T^*X$ is a compactly based cone, then we say that $P$ is
	microlocally solvable at $K$ if there exists an integer $N$ so that
	for every $f \in H^{loc}_{(N)}(X)$ there exists $u \in \Cal D'(X)$ such
	that $K \bigcap \wf(Pu-f) = \emptyset$. 
\end{defn}

Observe that solvability at a compact set $M \subset X$ is equivalent
to solvability at $T^*X\restr M$ by~\cite[Theorem 26.4.2]{ho:yellow},
and that solvability at a set implies solvability at a subset. Also,
by Proposition 26.4.4 in~\cite{ho:yellow} the microlocal solvability is
invariant under conjugation by elliptic Fourier integral operators and
multiplication by elliptic pseudodifferential operators.

To prove solvability we shall use {\em a priori} estimates. Let $\mn{u}_{(k)}$ be the $L^2$ Sobolev norm of order $k$, $u \in C_0^\infty$.
In the following, $P^*$ will be the $L^2$ adjoint of $P$.

\begin{rem}\label{solvrem}
	Let  $P \in {\Psi}^m_{cl}(X)$ and $ K \subset T^*X$ be a compactly based cone, and
	assume that there exists  ${\nu} \in \br$
	and a pseudodifferential operator $A$ so that
	$K \bigcap \wf(A) = \emptyset $ and
	\begin{equation}\label{solvest}
	\mn {u}_{(-N)} \le C(\mn{P^*{u}}_{({\nu})} + \mn {u}_{(-N-n)} +
	\mn{Au}_{(0)}) \qquad u \in C_0^\infty(Y)
	\end{equation}
	Then $ P$ is microlocally solvable at $K$ and one can take this $N$ in Definition~\ref{microsolv}.
\end{rem}

Observe that if $ P \in \Psi^m_{cl} $ then there is a loss of  $ \nu +  m + N $ derivatives in the estimate~\eqref{solvest} compared with the elliptic case.  One can have several operators $ A_j $ in~\eqref{solvest} by taking  $ A =(A_1, \dots)$ vector valued.

\begin{defn}\label{refdef}
We say that $ P  \in \Psi^m$ satisfies condition $\subr(\Psi) $  if there exists a homogeneous $0 \ne a \in S^0$ such that  $a p_r  $  has real principal symbol that vanishes of order $ 2 $ at an involutive manifold $ \st $ with nondegenerate Hessian, $a p_r  $ is of real subprincipal type and satisfies condition $ (\Psi) $ at the limit $ \st $. This means that $ \im a p_r $ does not change sign from $ - $ to $ + $ on the limits of the bicharacteristic of \/ $ \re a p_r$ at $ \st $.
\end{defn}

The  refined principal symbol is equal to $ p_r  = p + p_s $ by~\eqref{prefdef}.
Observe that this condition gives conditions on the sign changes for {\em any} value of $ \re p_s $, and that is {\em not} the case in Example~\ref{cexample}. This condition is stronger than the conditions in Examples~\ref{ex1}--\ref{subk}.
Observe that the factor $ a $ makes this condition invariant under multiplication with with  nonvanishing factors.
 It is also invariant under symplectic changes of variables, thus the conditions is invariant 
under conjugation  with Fourier integral operators and multiplication with elliptic pseudodifferential operators having real principal symbols.

\begin{prop}\label{refcond}
  If $ P \in \Psi^m $ has real principal symbol that vanishes of order $ 2 $ at an involutive manifold $ \st $ with nondegenerate Hessian and is of real subprincipal type, then $ P $ satisfies condition $ \subr (\Psi) $ if and only if \/ $ \im p_s $ does not change sign on the leaves of $ \st $ and the sign of $ \im p_s $ on the leaves do not change from $ - $ to $ + $ on the subprincipal bicharacteristics, i.e., the bicharacteristics of $ \re p_s $ with respect to the symplectic structure of  \/ $ \st $ for {\em any value} of $ \re p_s $, see Definition~\ref{subprinctype}.
\end{prop}


Here, the sign on the leaves is $ \pm 1 $  if  $ \pm \im p_s  \ge 0$ and $ \im p_s  \not \equiv 0 $ on leaf $ L $ of $ \st $ and equal to $ 0 $ if $ \im p_s  \equiv 0 $ on $ L $, see Definition~\ref{sgndef}.
Thus, condition $ \subr (\Psi)$ implies the necessary conditions in Examples~\ref{ex1} and~\ref{ex2}, and it is not hard to show that it implies the condition $\sub_2(\Psi) $ in Example~\ref{subk}.
In fact, the limits at $ \st $ of the Hamilton vector field of the refined principal symbol only depend on the values  at $ \st $ of the Hessian of the principal symbol  and the gradient of the real part of the subprincipal symbol, see~\eqref{Hrepr}.

\begin{proof}
	By multiplying with $ \w{D}^{2-m} $ we assume that $ P  \in \Psi^2$, and we may choose symplectic coordinates so that $ (x,y) \in \br^d \times \br^{nd}$ and $ \st = \set{\xi = 0} $.
	By Taylor's formula 
	we can write  the real principal symbol as
	\begin{equation}
	p_2(x,y;\xi,\eta) = \sum_{jk} a_{jk} (x,y;\xi,\eta)\xi_j \xi_k
	\end{equation}
	where the Hessian $ \{ a_{jk} (x,y;0,\eta) \} _{jk}$ is  nondegenerate near $ w_0 $ by assumption.  We have that  $p_1 = p_s$ the subprincipal  symbol on $ \st $  and by extending it we may assume it is constant in $  \xi $.
	
	Now the Poisson  parentheses $ \set {\xi, \re p_s}\equiv \partial_x \re p_s  \equiv 0 $ on $ \st $ by Definition~\ref{subprinctype} and $ \set {x, \re p_s}   \equiv -\partial_\xi \re p_s \equiv 0$. Since $ \re p_s$ is homogeneous and of principal type, we can thus complete $x $, $ \xi $ and $ \tau = \re p_s\restr{\st}$ to a symplectic homogeneous coordinate system $(x,t,y; \xi, \tau, \eta) $ microlocally near~$ w_0 \in \st$. 
	We then obtain that the refined principal symbol of $ P$ is equal to
	\begin{equation}\label{normal}
	\sum_{jk} a_{jk} (x,t, y; \xi, \tau,\eta)\xi_j \xi_k + \tau + i f(x,t,y; \tau, \eta) 
	\end{equation}
	modulo terms in $ S^{-1}$,
	where $ f $ is real and  homogeneous of degree 1. 

   The Hamilton vector field of the real part of the refined principal symbol is given by 
	\begin{equation}  \label{Hrepr}
	H_{\re p_r}  \cong 2\sum_{jk} a_{jk} (x,t, y; 0, \tau,\eta)\xi_j \partial_{x_k} + \partial_t 
	\end{equation}
	modulo terms with coefficients that are $ \Cal O( |\xi|^2 /\Lambda)$ with $ \Lambda = \sqrt{\tau^2 + | \eta |^2}$ in the base of homogeneous vector fields $ V= (\partial_t, \partial_x,  \partial_y, \Lambda \partial_\tau, \Lambda \partial_\xi,  \Lambda \partial_\eta) $. We get the limit at $ \st $ when  $ |\xi|/\Lambda \to 0 $. Observe that the terms in $ p_1 - p_s $ vanish on $ \st $ so their Hamilton vector fiels are tangent to the leaves of $ \st $. 
	
    One limit is  when $ \xi \to 0 $ for fixed $ \Lambda $ and then the limit of  \eqref{Hrepr}  is equal to $ \partial_t $ which gives the subprincipal bicharacteristics.
	We can also take  $ \tau =  \Lambda \tau_0$,  $ \eta = \Lambda \eta_0 $, $ y = y_0 $ and $ \xi = \Lambda^{1/2}\theta $ with $ | \tau_0 |^2 + |\eta_0|^2 = | \theta   | = 1 $.  Since the coefficients $ a_{jk} $ and $ C $ are homogeneous,    one can write \eqref{Hrepr}  as
	\begin{equation} \label{limchar}
	2 \Lambda^{1/2} \sum_{jk} a_{jk} (x,t, y_0;0, \tau_0,\eta_0)\theta_j \partial_{x_k}  + \partial_t 
	\end{equation}  
	modulo homogeneous vector fields in $ V $ with coefficients that are $ \Cal O(\Lambda^{-1/2}) $. 
	By dividing by $\lambda^{1/2} $ and  taking the limit as $ \lambda  \to \infty $ we obtain the limit  Hamilton vector field
	\begin{equation} \label{limham}
	H_{p_2} = 2\sum_{jk} a_{jk} (x,t, y_0; 0, \tau_0,\eta_0)\theta_j \partial_{x_k} 
	\end{equation}
	This vector field gives for fixed  $ \theta_0 =(\theta_1, \dots, \theta_d)$ a foliation of  the leaves of $  \st $. 
	Let $ \Gamma(\xi) $ with $\xi = \varrho \theta $ be the flow-out of the Hamilton vector field $  \varrho H_{p_2}$ with $ H_{p_2} $ given by~\eqref{limham}. Then the Hessian of $ \Gamma(\xi) $ at $ \xi = 0 $,  $\hess  \Gamma(0) $, is non-degenerate so the mapping $ \xi \mapsto  \Gamma( \xi)$ is a diffeomorphism from $\{\xi :  | \xi | < c  \} $ to a neighborhood of $ w_0  = (x,t_0,y_0;0,\tau_0, \eta_0) $ in the leaf through $ w_0$. Thus we obtain from condition $\subr(\Psi) $ that there can be no sign changes of $ f $ on the leaves of~$ \st $ near $ w_0 $.

	We can also take the limit $ (\tau, \eta) = \Lambda(\tau_0, \eta_0) $ and $ \xi = \varrho\cdot \theta $ with $ | \tau_0 |^2 + |\eta_0|^2 = | \theta   | = 1 $ and $0 \le  \varrho \ll \sqrt{ \Lambda} \to \infty$.  As before, we obtain the limit Hamilton vector field
	\begin{equation} \label{Hrepr0}
	 \partial_t + 2 \varrho \sum_{jk} a_{jk} (x,t, y_0; 0, \tau_0,\eta_0)\theta_j \partial_{x_k}  
	\end{equation}
	For fixed $ \varrho $ and $ \theta $ we find that the orbits of~\eqref{Hrepr0} gives a foliation of $ \st $. When $ \varrho = 0 $ the orbit~$ \gamma_0 $ of~\eqref{Hrepr0} through $ w_0 = (x_0 ,t_0, y_0; 0, \tau_0,\eta_0) $ is a subprincipal bicharacteristic. 
	When $ \xi = \varrho \theta  \ne  0  $  then the orbits $ \Gamma(t, \xi) $ of~\eqref{Hrepr0} through $ w_0 $  with  $ t > 0 $ form a proper cone in $ \st $ with the subprincipal bicharacteristic in its interior. 
	We obtain as before that the mapping $ (t,\xi) \mapsto \Gamma(t,\xi)$ is a diffeomorfism from $\{(t,\xi) :  | (t,\xi) | < c  \} $ to a neighborhood of $w_0 = (x_0,t_0,y_0;0,\tau_0 ,\eta_0) $.
	Thus, we find from condition $ \subr  (\Psi)$  that  if $ f  < 0  $ on the leaf of $ w_0 $ in a neighborhood of $ w_0 $ then $ f  \le 0 $ on the leaves of $w_t = (x_0,t,y_0;0,\tau_0 ,\eta_0) $  in a neighborhood of $ w_t $ for $ 0  \le t - t_0 \ll 1$, which proves the proposition.
\end{proof}

\begin{rem}
	The requirement that condition ($ \Psi $) shall hold on the refined principal symbol for all values of the real part is only needed when $f =  \im p_s$ depends on $ \tau = \re p_s $. In fact, if $ f $ does not depend on $  \tau $, then by choosing suitable $ \tau $ so that  $ \re p_r = 0 $ we can get the same limits at $ \st $ of the Hamilton vector field of $\re p_r  $ as in the proof  of Proposition~\ref{refcond}.
	One may of course eliminate the $ \tau $ dependence of $ f $ by the Malgrange preparation theorem, but that would change the imaginary part of the principal symbol $ p $, see for example Example~\ref{cexample}.
\end{rem}

The following is the main result of the paper.

\begin{thm}\label{mainthm}
	Assume that $P \in {\Psi}^m_{cl}(X)$  satisfies  $\sub_r (\Psi) $ microlocally 
	near $w_0 \in \st$, then $P$ is  microolocally solvable  near~$w_0$ with a loss of $ 5/2 $ derivatives by the {\sl a priori }estimate~\eqref{solvest}.
\end{thm}

Thus, by Example~\ref{subk} condition   $\subr (\Psi) $  is both necessary (under additional conditions) and sufficient for local solvability for operators  of subprincipal type with  principal symbol that is real and
vanishes of exactly second order at a nonradial involutive manifold~$ \st $.

The solvability with a loss of $ 5/2 $ derivatives can be compared with the loss of $ 2 $ derivatives when the antisymmetric part of P is bounded. In fact, by using the normal form~\eqref{normal0} with  $ f \equiv 0 $
gives a Schr\"odinger type operator  that is symmetric modulo bounded operators.  By using a multiplier as in Lemma~\ref{multlem} one can obtain $ L^2 $ estimates with arbitrarily small constants. For small enough constant, these estimates may be perturbed by any bounded term. 
This is similar to the case of operators of principal part, where the loss of derivatives is $ 1/2 $  more when the antisymmetric part if the operator is unbounded and condition $( \Psi )$ is  satisfied.

This paper treats subprincipal type operators with involutive characteristics having nondegenerate second order vanishing of the principal symbol. For the noninvolutive or degenerate cases, see~\cite{Parm}, \cite{Serparm} and the references there.
 
To prove Theorem~\ref{mainthm} we shall use suitable {\em a priori} estimate and Remark~\ref{solvrem}. The proof will occupy most of the remaining paper.

\section{The Preparation}\label{prep}

As in the proof of Proposition~\ref{refcond} we may assume that the operator $ P \in {\Psi}^2_{cl}(X)$ is of second order with real  principal symbol, $ X = \br^n$ and the coordinates are chosen so that $ \st = \set {\xi = 0}$, $(x,y;\xi,\eta) \in T^*(\br ^d \times \br ^{d-n}) $ microlocally near $w_0 \in \st $, where $ x \mapsto (x, y_0;0, \eta_0) $  spans the leaves of the symplectic foliation of $ \st $ and  $d$ is the codimension of $ \st$. We may multiply $ P $ with an elliptic operator of order zero so that the refined principal symbol satifies condition $ \subr(\Psi) $, see Definition~\ref{refdef}.
Thus, we have the operator on the normal form given by~\eqref{normal} 
\begin{equation}\label{normal0}
\sum_{jk} a_{jk} (x,t, y; \xi, \tau,\eta)\xi_j \xi_k + p_1(x,t,y; \xi,\tau, \eta)   + p_0(x,t,y;\xi, \tau, \eta)
\end{equation}
modulo terms in $ S^{-1}$, where  $ \set{a_{jk}}_{jk}$ is nondegenerate on $ \st $,
\begin{equation}\label{normalsub}
p_1(x,t,y; \xi,\tau,  \eta)  =  \tau + i f(x,t,y; \tau, \eta) + C(x,t,y;  \xi,\tau,  \eta)\cdot \xi
\end{equation}
with $ f  $ homogeneous of degree 1, and $ C $ and $ p_0 $ homogeneous of degree 0. 
By  condition $ \subr  (\Psi)$ and Proposition~\ref{refcond} we find that $ f $ does not change sign on the leaves of $ \st $ and  the sign on the leaves do not change from $ - $ to $ + $ as $ t $ increases by Definition~\ref{sgndef}.

We shall compute the symbol modulo the error terms 
\begin{equation}\label{error}
R = \{ \w{C_2\xi, \xi} + C_1 \cdot \xi + C_0:\ C_j \in S^{-1} \} 
\end{equation}
These are sums of terms that are either  in $ S^1 $ vanishing  of second order  on $ \st $,  in  $ S^0 $  vanishing  on $ \st $  or  in $ S^{-1} $.
Observe that homogeneous vector fields that are tangent to $ \st $ maps $ R $ into itself.

We shall use the Weyl quantization, which has the property that symmetric operators have real symbols.
The Weyl quantization of symbols $a \in \Cal
S'(T^*\br^n)$ is defined by: 
\begin{equation} \label{weylop}
\sw{a^wu,v} = (2{\pi})^{-n}\iiint
\exp{(i\w{x-y,{\xi}})}a\!\left(\tfrac{x+y}{2},{\xi}\right)u(x)\ol
{v(y)}\,dxdyd{\xi}
\qquad u, v \in C_0^\infty
\end{equation} 
Observe that $\re a^w = (\re a)^w$ is the symmetric
part and $i\im a^w = (i\im a)^w$ the antisymmetric part of the
operator $a^w$. Also, if $a \in S^m_{1,0}$ then $a(x,D_x)
= b^w(x,D_x)$ modulo ${\Psi}^{m-2}_{1,0}$ where
\begin{equation}\label{kntoweyl}
b(x,\xi)  = a(x,\xi) + \frac{i}{2}\sum_{j}\partial_{x_j}\partial_{\xi_j} a(x,\xi) 
 \end{equation}
which gives the subprincipal symbol by~\cite[Theorem~18.5.10]{ho:yellow}. 
The equality~\eqref{kntoweyl} shows that  $ a \in R $ if and only if  $b \in R $.
I
 
Now by conjugating with $ e^{\phi} \in S^0$ having phase $ \phi(x,t,y;\tau, \eta)$ that is  real and homogeneous of degree 0, we may obtain that $ \im p_0 = 0 $ at $ \st $, i.e., $  \im p_0 \in R$. In fact, we obtain this by solving the equation 
 \begin{equation} 
\partial_t  \phi(x,t,y;\tau, \eta) + \re C(x,t,y;\tau, 0,\eta)\cdot \partial_x  \phi(x,t,y;\tau, \eta) = \im p_0(x,t,y;\tau, 0,\eta) 
  \end{equation}
but this may of course change  the values of $\im  C $ and $\re  p_0 $.

Next, we want to reduce to the case $ \im C = 0 $ on $ \st $. i.e, $ \im C \in R $, but that can in general not be done by conjugation. Instead we shall use symplectic changes of variables given microlocally by Fourier integral operators.
In the following, we shall for simplicity include the variable $ t $ in the  $ y $ variables, and the variable $ \tau $ in the $ \eta $ variables. The variables $ (y,\eta) $ will then parametrize the leaves of $ \st $.

Let $ \br^d \ni x \mapsto \chi(x, y,\eta)\in \br^d $, where $\chi \in C^\infty $,  homogeneous in $ \eta $ and $ | \partial_{x}\chi| \ne 0 $ and let
\begin{equation} 
F u(x,y) = (2\pi)^{-n}\iiiint e^{i (\w{\chi(x,y,\eta) - z, \xi} + \w{y-w, \eta})} u(z,w)\, dzdwd\xi d\eta \qquad u \in C^\infty_0
 \end{equation}
 which is an elliptic Fourier integral operator.
This correspond to the homogeneous symplectic transformation 
$$ 
(x,y; \partial_{x}\chi(x,y,\eta)  \cdot \xi , \eta + \partial_{y}\chi(x,y,\eta) \cdot \xi  ) 
 \mapsto ( \chi(x,y,\eta), y + \partial_\eta \chi \cdot \xi ;\xi,\eta)
$$
which preserves $ \st $, thus $ |\xi | $ and $ | \eta | $ are preserved modulo   multiplicative constants.
In fact, when $ \xi = 0 $ we get the mapping $ (x,y; 0,\eta)  \mapsto ( \chi(x,y,\eta), y  ; 0,\eta) $ which gives a homogeneous change of $ x $ variables.
We put  the amplitude of $ F $ equal to 1 to simplify the notation, actually the amplitude only has to equal to 1 near the wave front set of the kernel of $ F $. 

By applying the operator $ P $ we find
\begin{equation} 
PFu(x,y) =  (2\pi)^{-n}\iiiint e^{i (\w{\chi(x,y,\eta) - z, \xi} + \w{y-w, \eta})} Q(x,y;\xi, \eta) u(z,w)\, dzdwd\xi d\eta
 \end{equation}
so $ PFu = FQu $, where 
\begin{equation} \label{conjsymb}
Q(x,y;\xi, \eta) = \sum_{\alpha, \, \beta \in \bn} \partial_{{\xi}}^\alpha \partial_{\eta}^{\beta} P(x,y; \partial_x \chi \cdot \xi, \eta + \partial_y \chi \cdot \xi) {\Cal M}^{\chi}_{\alpha, \beta } (x,y;\xi, \eta) /\alpha ! \beta !
 \end{equation}
 with
 \begin{equation} 
 {\Cal M}^{\chi}_{\alpha, \beta } (x,y;\xi, \eta)  = D_z^\alpha D_w^\beta e ^{i\chi_2(x,y,z,w,\eta)\cdot \xi}\restr{\substack{z=x\\ w=y}}
  \end{equation}
where the phase function
\begin{multline} 
\chi_2(x,y,z,w,\eta) = \chi(z,w,\eta)  - \chi(x,y,\eta)  \\ + \w{x-z, \partial_x \chi (x,y,\eta)}
+ \w{y-w, \partial_y \chi (x,y,\eta)}
 \end{multline} 
vanishes of second order at $ z= x $ and $ w = y $, see \cite[Th.\ 18.1.17]{ho:yellow} or \cite[Chapter 7, Theorem 3.1]{T2}.
Thus there are no terms with $ | \alpha| + | \beta | = 1 $ in the expansion of~\eqref{conjsymb}. Since we only need the symbols modulo terms in $ R  $ it suffices to compute the first two terms of the expansion~\eqref{conjsymb}.

We obtain from~\eqref{normal0} by a straighforward computation that
\begin{equation} \label{qsymbol}
Q(x,y;\xi, \eta) \cong P(x,y; \partial_x \chi \cdot \xi, \eta + \partial_y \chi \cdot \xi) \\ + \frac{1}{2i} \sum_{jk} a_{jk}(x,y; 0, \eta) \partial_{x_{j}}\partial_{x_{k}}\chi (x,y,\eta)\cdot \xi
 \end{equation} 
modulo terms that are in $ R $. In fact, $ \partial_{{\xi}} a_{jk}\in S^{-1} $,  $  \partial_{{\xi}}^\alpha \partial_{\eta}^{\beta} p_2 = \Cal O (|\eta |^{-1}|\xi | ) $ if $ \beta \ne 0 $ and  $| \alpha +  \beta|  \ge 2 $, and $  \partial_{{\xi}}^\alpha \partial_{\eta}^{\beta} p _1 \in  S^{-1} $ if  $| \alpha +  \beta|  \ge 2 $. 

We shall first simplify by making a change of variables to diagonalize $A = \{a_{jk}\}_{jk} = \hess p_2 $. Since $ A $ is nondegenerate, we can use the spectral projections to obtain either that
$ A = \begin{pmatrix} A_+ & 0 \\ 0 & A_- \end{pmatrix} $ 
or $ A = A_\pm $, where $ \pm A_\pm  $ is positive definite. Then we can use $ (\pm A_\pm )^{-1/2} $ to construct real valued $ \chi (x,y, \eta)$ so that $ (\partial_x \chi\cdot \xi)^t A\/ \partial_x \chi \cdot \xi = |\xi'|^2 - |\xi''|^2 = L(\xi)$ has constant coefficients near $ w_0 \in \st $. Here  $ (\xi', \xi'') = \xi $,  and $   L(\xi)$ is the real quadratic form with the polarized bilinear form $   L(\xi_0,\xi)$. 
Of  course,  this may also change the values of $ p_j $ for $ j < 2 $.
But observe that the term  homogeneous of order 0 in  the expansion~\eqref{qsymbol} of $ Q $ is equal to
$$  
q_0(x,y;\xi,\eta) = p_0 (x,y;\partial_{x} \chi_k (x,y,\eta)\cdot \xi,\eta  + \partial_{y} \chi (x,y,\eta) \cdot \xi) 
$$  
modulo terms vanishing at $ \st $, so we find that $ \im q_0(x,y;0,\eta)  =  \im p_0 (x,y;0,\eta )  = 0$.

Thus, we may assume that $ A = L $ is constant in the following. Next we shall do another change of  symplectic variables to make $ \im C = 0$ at $ \st $. Assume that $ \chi = (\chi_1, \dots, \chi_d) $ where $ \chi_j $ is parallel to $ e_j $  for the standard base  $ e_1, \dots , e_d $ of $ \br^d $, so that we can write $ \chi = (\chi_1e_1, \dots, \chi_d e_d) $ with scalar $ \chi_j $. As before, we get the expansion~\eqref{qsymbol} with $ \set{a_{k}}_{jk} = A = L $.

In order to get the symmetric part of the operator we shall compute the Weyl symbol of $ Q $ which is given by $ \wt Q \cong Q  +  \frac{i}{2}\sum_{j}\partial_{x_j}\partial_{\xi_j} q_2$ modulo  $S^0 $ where 
$$ q_2(\xi) = L(\partial_x \chi \cdot \xi)=  | \partial_{x'} \chi\cdot \xi|^2 -  | \partial_{x''} \chi\cdot \xi|^2  = \sum_{k= 1}^{\ell}(\partial_{x_k}\chi \cdot  \xi)^2 -  \sum_{k= \ell + 1}^{m}(\partial_{x_k}\chi \cdot \xi)^2
$$ 
is the principal symbol of $ \wt Q $, so~\eqref{qsymbol} gives 
\begin{equation}\label{weylcorr}
\wt Q \cong q_2  + i \sum_{k= 1}^m \left (2L(\partial_x \chi_k, \partial_x \partial_{x_k} \chi_k)  - \frac{1}{2}  L(\partial_{x}) \chi_k  \right)  \xi_k +  p_1(x,y;\partial_x \chi \cdot \xi , \eta  + \partial_y \chi \cdot \xi)
\end{equation}
 modulo $ S^0 $.

Let $ q_1 $ be the terms homogeneous of order 1 in  the expansion~\eqref{weylcorr} of $\wt  Q $, then we have  $ q_1(x,y;0,\eta)  =  p_1 (x,y;0,\eta ) $ on $ \st $. We find from~\eqref{weylcorr} that the $ \xi_k $ derivative of $ q_1 $ at $ \st $ is equal to
\begin{multline} 
2i L(\partial_x \chi_k (x,y,\eta), \partial_x \partial_{k} \chi_k (x,y,\eta))  - \frac{i}{2}  L(\partial_{x}) \chi_k (x,y,\eta) \\
+ \partial_\xi p_1  (x,y;0,\eta)\cdot \partial_{x} \chi_k  (x,y,\eta) + \partial_\eta p_1  (x,y;0,\eta) \cdot \partial_{y} \chi_k  (x,y,\eta) 
 \end{multline} 
 Here $ \im \partial_\eta p_1  =  \partial_\eta f  $ and $ \im \partial_\xi p_1 =  \im C $ on $ \st $.
 Observe that the terms in  $ q_1 $ vanishing of second order  at $ \st $ are in $ R $.

 By taking the imaginary part and ignoring the term  $ \partial_\eta f \cdot \partial_{y} \chi_k  $ for now, we find $ \partial_{{\xi}} \im q_1  (x,y;0,\eta)  = 0$ if for any $ k $ we have
\begin{multline} \label{changesys}
L(\partial_{x}) \chi_k (x,y,\eta)  - 4L(\partial_x \chi_k (x,y,\eta), \partial_x \partial_{k} \chi_k (x,y,\eta)) \\ - 2 \im  C (x,y;\eta)  \cdot \partial_{x} \chi_k (x,y,\eta)  = 0 
\end{multline}
which is a  quasilinear second order system of  PDE on the leaves of $ \st $ having real coefficients.
 Observe that this system is completely decoupled with one equation for each $ \chi_k $. 
As before, we we find that $ \im q_0(x,y;0,\eta)  =  \im p_0 (x,y;0,\eta )  = 0$. 
In order to get a suitable change of coordinates we shall solve system~\eqref{changesys}  in a neighborhood of $ w_0 = (x_0,y_0,\eta_0)$ with initial data  $ \chi = 0 $  and  $|\partial_{x} \chi| \ne 0 $ at $ w_0  $.

\begin{prop}\label{approp}
	For any $ v_k \in \br^d $, $ 1 \le k \le d $, the equation \eqref{changesys} 
	with data  $ \chi_k = 0$ and $ \partial_{x} \chi_k = v_k $ at $ w_0 =   (x_0,y_0,\eta_0) $  has a solution $ \chi_k (x,y,\eta) \in C^\infty $ in a neighborhood of ~$w_0$.
\end{prop}

Proposition~\ref{approp} follows from  Theorem~\ref{appthm} 
in Appendix~\ref{app}, and shows that this initial value problem   has a $ C^\infty $ solution $\chi(x,y,\eta)  = (\chi_1(x,y,\eta), \chi_2(x,y,\eta) \cdots   )$ near $ w_0$ such that $ \chi(x_0,y_0,\eta_0) = 0  $ and $  \partial_x\chi(x_0,y_0,\eta_0) =  \id  $. By restricting $ \chi(x,y,\eta)  $ to the set $ |\eta| = 1 $ and extending it by homogeneity in $ \eta $, we obtain a homogeneous change of coordinates so that $ \partial_{{\xi}} \im p_1 = \im C = 0$  and  $ \im p_0 = 0 $ at $ \st $ near $ w_0 $.  Thus we have proved the following result.
 
\begin{prop}\label{prepprop}
Assume that $ P \in \Psi_{cl}^m $(X) satisfies the conditions in Theorem \ref{mainthm} microlocally near $ w \in \st $. 
By	conjugation with elliptic Fourier integral operators and  multiplication with symmetric elliptic pseudodifferential operators, we may assume that $ X = T^*\br^n $, the coordinates are $ (x,t,y; \xi, \tau. \eta) $ so that $  \st = \set{\xi = 0} $ and $ P $ is on the form
\begin{equation}\label{preppropform}
	P \cong D_t + A^w +  i f_1^w 
\end{equation} 	
 microlocally near $ w \in \st $ modulo terms with symbols in $ R $ given by~\eqref{error}.
Here $ A = A_2 + A_1 + A_0 \in S^2_{cl} $, $ A_j \in S^j $, is real valued, 
with principal symbol $ A_2 $ vanishing of second order at $ \st $ with $ \hess A_2$ nondegenerate on the normal bundle $ N\st $, and  $ A_1 = 0 $  at $ \st $. Also, $ f _1 = f + f_0$  is  real and homogeneous of degree $ 1 $ where $ f $ does not depend on $ \xi $ and $f_0 =   \partial_{\eta}f \cdot r \cdot \xi $ with  $ r \in S^0 $.  By condition $ \subr(\Psi) $ we have that $ f  = f_1\restr \st$
does not change sign on the leaves of $ \st $ and the sign of $ f$ on the leaves do not change from $ - $ to $ + $ as $ t $ increases. Here the sign of $ f $ is given by Definition~\ref{sgndef}.
\end{prop}

Observe that $ P $ in \eqref{preppropform} is an evolution operator and it is of subprincipal type.

\begin{rem}\label{preprem}
	The normal form of the $ L^2 $ adjoint $ P^* $ is~\eqref{preppropform} with $ f_1 $  replaced by $ -f_1 $. Then $ P^* $ satisfies condition $ \subr(\ol \Psi)$ which gives the opposite  conditions on sign changes as $t $ increases given by~\eqref{pcond0}.
\end{rem}

\section{The microlocal estimate}\label{reduct}

Next, we shall microlocalize and reduce the proof of
Theorem~\ref{mainthm} to the semiclassical multiplier estimate of
Proposition~\ref{mainprop} for a microlocal normal form of the adjoint
operator.  We shall consider operators given by Proposition~\ref{prepprop}
\begin{equation}\label{pdef1}
 P^* \cong  D_t +A^w +  i f_1^w(x,t,y;D_t, D_y)
\end{equation}
modulo terms with symbols in $ R $ given by~\eqref{error}.
Here $ f _1 = f + f_0$  is  real and homogeneous of degree $ 1 $ where $f_0 =   \partial_{\eta}f\cdot  r \cdot \xi $ with  $ r \in S^0 $, $ f $ does not depend on $ \xi $ and 
\begin{equation}\label{adef}
A = \sum_{jk} a_{jk}\xi_j\xi_k + \sum_j a_j\xi_j + a_0
\end{equation}
where $ a_{jk} $ and $ a_j  \in S^0_{1,0} $ are real and homogeneous of degree 0, and $ \{ a_{jk}\}_{jk} $ is symmetric  and nondegenerate.

In the following, we shall
assume that ~~$P^* $ satisfies  condition~$ \subr(\ol \Psi) $, 
so that the sign of $ f(t,x,y;\tau,\eta) $ is constant in $ x $ and
\begin{equation} \label{pcond0}
f(t,x_0,y_0;\tau_0,\eta_0) > 0  \quad\text{and $s > t$} \implies  f(s,x, y_0;\tau_0,\eta_0 ) \ge 0 \ \ \forall\, x
\end{equation}
so that the sign on the leaves cannot change from $ + $ to $ - $ as $ t $ increases, see Definition~\ref{sgndef} for the definition of the sign.
Observe that if ${\chi} \ge 0$ then ${\chi}f$ also satisfies  condition~ $ \subr(\ol\Psi) $,  so this condition can be microlocalized.

In order to prove Theorem~\ref{mainthm} we shall make a
second microlocalization using the specialized symbol classes of
the Weyl calculus. 
We shall therefore recall the definitions of the Weyl calculus: let $g_{w}$ be a
Riemannean metric on~ $T^*\br^n$, $w = (x,{\xi})$, then we say that
$g$ is slowly varying if there exists $c>0$ so that $g_{w_0}(w-w_0) <
c $ implies $g_{w} \cong g_{w_0}$, i.e., $1/C \le g_{w}/g_{w_0}
\le C.  $ Let ${\sigma}$ be the standard symplectic form on $T^*\br^n$, and assume
$ g^{\sigma}(w) \ge g(w)$ where $ g^{\sigma}$ is the dual metric of $w \mapsto
g({\sigma}(w))$.  We say that $g$ is ${\sigma}$~ temperate if it
is slowly varying and there exists $ C > 0 $ and $ N \in \bn $ so that
\begin{equation*}
g_{w} \le C g_{w_0}( 1 +
g^{\sigma}_{w}(w-w_0))^N  \qquad
\forall\ w,\ w_0 \in T^*\br^n
\end{equation*}
 Actually, ${\sigma}$~ temperate metrics with $ g \le g^\sigma $ are called H\" ormander metrics.
A positive real valued function $m(w)$ on $T^*\br^n$ is 
$g$~continuous if there exists $c>0$ so that
$g_{w_0}(w-w_0) < c $ implies $m(w) \cong
m(w_0)$.
We say that $m$ is 
${\sigma}$, $g$~ temperate if it is $g$ ~continuous and there exists $ C > 0 $ and $ N \in \bn $ so that
\begin{equation*}
m(w) \le C m(w_0)( 1 +
g^{\sigma}_{w}(w-w_0))^N  \qquad \forall\ w,\ w_0 \in T^*\br^n
\end{equation*}
If $m$ is ${\sigma}$, $g$~ temperate, 
then $m$ is a weight for ~$g$ and we can define the symbol classes:
$a \in S(m,g)$ if $a \in C^\infty(T^*\br^n)$ and
\begin{equation}\label{symbolest} 
|a|^g_j(w) = \sup_{T_i\ne 0}
\frac{|a^{(j)}(w,T_1,\dots,T_j)|}{\prod_1^j
	g_{w}(T_i)^{1/2}}\le C_j
m(w)\qquad \forall \ w \in T^*\br^n \qquad\text{for $j \ge 0$}
\end{equation}
which gives the seminorms of $S(m,g)$. If $a \in
S(m,g)$ then we say that the corresponding Weyl operator $a^w \in \op
S(m,g)$.  
For more on the Weyl calculus, see \cite[Section
18.5]{ho:yellow}.

\begin{defn}\label{s+def}
	Let $m$ be a weight for the metric $g$.  We say that $a \in S^+(m,g)$
	if $a \in C^\infty(T^*\br^n)$ and $|a|_j^g \le C_jm$ for $j \ge
	1$.
\end{defn}

Observe that by the mean value theorem we find that 
\begin{equation}\label{asymbol}
|a(w) - a(w_0)| \le C_1 \sup_{{\theta} \in [0,1]}
g_{w_{\theta}}(w-w_0)^{1/2}m(w_{\theta}) \\ \le C' m(w_0)(1 +
g_{w_0}^{{\sigma}}(w-w_0))^{(3N+1)/{2}}  
\end{equation}
where $w_{\theta} = {\theta}w + (1-{\theta})w_0$, since $ w_\theta - w_0 = \theta(w - w_0) $ for some $  0  < \theta < 1 $ and  
\[ g_{w_{\theta}}(w-w_0) \lesssim g^\sigma_{w_{\theta}}(w-w_0) \lesssim  
g^\sigma_{w_0}(w-w_0) (1 + g^\sigma_{w_0}(w-w_0) )^N\]
Thus $m + |a|$ is a weight for ~$g$ and
 $a \in S(m + |a|, g)$, so the operator ~$a^w$ is well-defined.

\begin{lem}\label{calcrem}
	Assume that $m_j$ is a weight for $g_j = h_j g^\sh \le g^\sh =
	(g^\sh)^{\sigma} \le g_j^\sigma \le h_j^{-1}g^\sh$ and $a_j \in
	S^+(m_j, g_j)$, $j=1$, $2$.  Let $g = g_1 + g_2$ and $h^2 = \sup
	g_1/g_2^{\sigma} = \sup g_2/g_1^{\sigma} = h_1h_2$, then
	\begin{equation}\label{2.3}
	a_1^wa_2^w -(a_1a_2)^w \in \op S(m_1m_2h,g)
	\end{equation}
	with the usual expansion of~\eqref{2.3}  in terms in $
	S(m_1m_2h^k,g)$, $k\ge 1$.
	We also have that
	\begin{equation}\label{2.3a}
	\re a_1^wa_2^w -(a_1a_2)^w \in \op S(m_1m_2h^2,g)
	\end{equation}
	if $a_j\in C^\infty$ is real and $| a_j|^{g_j}_k \le C_k 
	m_j$, $k \ge 2$, for  $j=1$, $2$. In that case we have
	$a_j \in S(m_j + |a_j| + |a_j|^{g_j}_1, g_j)$.
\end{lem}

\begin{proof}
	As shown after Definition~\ref{s+def} we have that $m_j + |a_j|$ is a
	weight for ~$g_j$ and $a_j \in S(m_j + |a_j|, g_j)$, $j = 1$, 2.
	Thus 
	$a_1^w a_2^w \in \op S((m_1 + |a_1|)(m_2 + |a_2|),g)$
	is given by  Proposition~18.5.5 in~\cite{ho:yellow}.
	We find that $a_1^wa_2^w -(a_1a_2)^w = a^w$ with
	\begin{equation}\label{nfcalc}
	a(w) = E(\tfrac{i}{2}{\sigma}(D_{w_1},D_{w_2}))
	\tfrac{i}{2}{\sigma}(D_{w_1},D_{w_2})a_1(w_1)a_2(w_2)\restr{w_1= w_2=w}  
	\end{equation}
	where $E(z) = (e^z -1)/z = \int_0^1 e^{{\theta}z}\,d{\theta}$. 
	Here
	${\sigma}(D_{w_1},D_{w_2})a_1(w_1)a_2(w_2) \in S (M H, G)$ where
	$
	M(w_1,w_2) = m_1(w_1)m_2(w_2)
	$,
	$G_{w_1,w_2}(z_1,z_2) = g_{1,w_1}(z_1) + g_{2,w_2}(z_2)$ and $H^2(w_1,w_2) = h_1(w_1)h_2(w_2) =
	\sup\, G_{w_1,w_2}/G_{w_1,w_2}^{\sigma}$ so that $H(w,w) = h(w)$. The
	proof of Theorem~18.5.5 in ~\cite{ho:yellow} works when
	${\sigma}(D_{w_1},D_{w_2})$ is replaced by
	${\theta}{\sigma}(D_{w_1},D_{w_2})$, uniformly in $0 \le  {\theta} \le
	1$ (when $ \theta = 0 $ we just get the Poisson parenthesis $ \tfrac{i}{2} \{a_1, a_2  \} $). By integrating over ${\theta} \in [0,1]$ we obtain that $a(w)$ has
	an asymptotic expansion in $S(m_1m_2 h^k,g)$, which proves ~\eqref{2.3}. 

	If  $| a_j|^{g_j}_k \le C_k 
	m_j$, $k \ge 2$, then we have by Taylor's formula as in
	~\eqref{asymbol} that
	\begin{multline*} 
	|a_j(w) - a_j(w_0)| \le g_{w_0}(w-w_0)^{1/2}|a_j|_1^g(w_0) + C_1 \sup_{{\theta} \in [0,1]}
	g_{w_{\theta}}(w-w_0)m(w_{\theta}) \\ \le C' (|a_j|_1^g(w_0) + m(w_0))(1 +
	g_{w_0}^{{\sigma}}(w-w_0))^{2N+ 1}  
	\end{multline*}
	\begin{multline*}
	|\w{T,\partial_wa_j(w)} - \w{T,\partial_wa_j(w_0)}| \le C_2
	\sup_{{\theta} \in [0,1]}
	g_{w_{\theta}}(T)^{1/2}g_{w_{\theta}}(w-w_0)^{1/2}m(w_{\theta}) \\
	\le C_3g_{w_0}(T)^{1/2} m(w_0)(1 + 
	g_{w_0}^{{\sigma}}(w-w_0))^{(4N+1)/{2}} 
	\end{multline*}
	thus $m_j + |a_j| + |a_j|^{g_j}_1$ is a weight for $g_j$ and clearly
	$a_j \in S(m_j + |a_j| + |a_j|^{g_j}_1, g_j)$. 

	Now  if $a_1$ and $a_2$ are real, then $\re a_1^wa_2^w -(a_1a_2)^w = a^w$ with
	\begin{equation*}
	a(w) = \re E(\tfrac{i}{2}{\sigma}(D_{w_1},D_{w_2}))
	(\tfrac{i}{2}{\sigma}(D_{w_1},D_{w_2}))^2a_1(w_1)a_2(w_2)/2\restr{w_1= w_2=w}  
	\end{equation*}
	where ${\sigma}(D_{w_1},D_{w_2})^2a_1(w_1)a_2(w_2) \in S(MH^2,G)$,
	with the same~$E$, $M$, ~$G$ and $H$ as before.
	The proof of ~\eqref{2.3a} then follows in the same way as the proof
	of ~\eqref{2.3}. 
\end{proof}

\begin{rem}\label{vvcalc}
	The conclusions of Lemma~\ref{calcrem} also hold if $a_1
	$ has values in $\Cal L(B_1, B_2)$ and $a_2$ in~ $B_1$ 
	where $B_1$ and $B_2$ are Banach spaces, then $a_1^wa_2^w$ has values
	in $B_2$.
\end{rem}

For example, if $\set{a_j}_j \in S(m_1,g_1)$ with values in $\ell^2$,
and $b_j \in S(m_2,g_2)$ uniformly in ~$j$, then $\set{a_j^wb_j^w}_j \in
\op (m_1m_2,g)$ with values in ~$\ell^2$. 

\begin{rem}\label{nfcalcrem}
For pseudodifferential operators with the Kohn-Nirenberg quantization, we have by Theorem 4.5 and (4.13)  in \cite{ho:weyl} that $ a_1(x,D) a_2(x,D) = a(x,D) $ with
\begin{equation}\label{nfcalc0}
	a(x,\xi)   = e^{i\w{D_\xi, D_y}}  a_1(x,\xi) a_1(y,\eta) \restr{\substack{y=x \\ \eta = \xi}}
\end{equation}
As in the proof of  Lemma~\ref{calcrem} we find that  $ a_1(x,D) a_2(x,D) - a(x,D) = r(x,D)$ with
\begin{equation}\label{nfremain}
r(x,\xi)   = E({i\w{D_\xi, D_y}}) \partial_{{\xi}} a_1(x,\xi) D_ya_1(y,\eta) \restr{\substack{y=x \\ \eta = \xi}}
\end{equation}
where $E(z) = (e^z -1)/z = \int_0^1 e^{{\theta}z}\,d{\theta}$.
\end{rem}

To prove Theorem~\ref{mainthm} we shall  prove an estimate for the microlocal normal form of the adjoint operator. 
Since the proof is rather long, we will take it in two steps, and the first is microlocal estimate.
For that, we shall use the symbol classes $ S^m_{1,0} $ and $ S^m_{1/2,1/2} $.

\begin{prop}\label{maincor}
	Assume that $P$ is as in Proposition~\ref{prepprop} microlocally near $ w_0 \in \st $ and  that  $ t = t_0 $ and $ x = x_0 $ at $ w_0 $.
	Then there exist $T_0 > 0$ and a real valued symbol
	$b_T \in S^1_{1/2,1/2}$ with homogeneous gradient $\nabla b_T = (\partial_{z} b_T, |{{\zeta}}| \partial_{\zeta} b_T)
	\in  S^{1}_{1/2,1/2}$ uniformly for $0 < T \le
	T_0$, $ (z,\zeta) \in T^*\br^n$, such that for every $ N > 0 $ there exists $ C_N > 0 $ so that
	\begin{equation}\label{corest}
	\mn{b_T^wu}^2_{(-1/2)}+ \mn{D_x u}^2 + \mn u^2\le  C_N\left(T \im\sw{P^*u,b_T^wu} + 
	\mn{u}_{(-N)}^2 \right) + \mn{\psi^w u}^2
	\end{equation} 
	for $ u \in C_0^\infty$ having support where $|t - t_0| \le T$ and  $|x -x_0| \le T$.
	Here $ \psi \in S^{2} $, $ w_0 \notin \wf \psi ^w$ and
	the constants $T_0$, $C_N$ and the
	seminorms of\/ $b_T$ only depend on the seminorms of'
	the symbols in $ P $.
\end{prop}

\begin{proof} [Proof that  Proposition~\ref{maincor} gives
	Theorem~\ref{mainthm}] We shall prove that there exists ${\phi} \in S^0_{1,0}$ such that ${\phi} \ge  0 $ and $ \phi = 1 $ in a conical
	neighborhood of $w_0 \in \st$, and $R \in S^{3/2}_{1,0}$ with $w_0 \notin \wf R^w $ so that
	for any $ N > 0 $ there exists $ C_N > 0 $ such that
	\begin{equation}\label{solvest1}
	\mn{{\phi}^w u} \le C_N\left(\mn{{\phi}^w P^*u}_{(1/2)} + \mn{R^w u} + \mn{u}_{(-N)}\right)
	\qquad u \in C_0^\infty
	\end{equation}
	Here $\mn{u}_{(s)}$ is the usual $ L^2 $ Sobolev norm, so by Remark~\ref{solvrem} we obtain that $ P $ is solvable with a loss of $ 5/2 $ derivatives in a conical neighborhood of  $ w_0 $ since $ w_0 \notin \wf (1 - \phi)^w $ and $ m = 2 $. 
	
	We may assume that $ m = 2 $ and $ P \in \Psi^2_{1,0}$ is on the form
	in Proposition~\ref{prepprop}  in a conical neighborhood of $ w_0 $. Let $ \phi \ge 0 $  have support in a smaller conical neighborhood such that $ \max{(|t|, |x|)}\le T \le T_0 $ in $ \supp {\phi} $, 
	$ \psi $ in~\eqref{corest} vanishes on $ \supp \phi $ and 
	$ \phi = 1 $ in a conical
	neighborhood of  $ w_0 $. Then by applying the estimate~\eqref{corest} on ${\phi}^wu$ we obtain for any $ N > 0 $ 
	\begin{multline}\label{fest}
	\mn{b_T^w{\phi}^wu}^2_{(-1/2)} + \mn{{\phi}^w u}^2   + \mn{D_x{\phi}^w u}^2  \le C_N
	\left(T \im\sw{P^*{\phi}^w u, b_T^w{\phi}^wu} + \mn {  u}^2_{(-N)} \right)
	\end{multline}
	where $  C_N  > 0 $ and $b_T^w \in {\Psi}^{1}_{1/2,1/2}$ is
	symmetric with homogeneous gradient $\nabla b_T \in
	S^{1}_{1/2,1/2}$. 
	
	By Cauchy-Schwarz,
	\begin{equation}
	|\sw{P^*{\phi}^wu, b_T^w{\phi}^wu}| \lesssim  \mn {P^*{\phi}^wu}^2_{(1/2)} + \mn{ b_T^w{\phi}^wu}^2_{(-1/2)}
	\end{equation}
	and $  \mn {P^*{\phi}^wu}^2_{(1/2)}  \le  \mn {{\phi}^wP^*u}^2_{(1/2)}   +  \mn {[P^*, {\phi}^w] u}^2_{(1/2)} $
	where the commutator $[P^*, {\phi}^w] \in \Psi^1$ with $ w_0 \notin \wf [P^*, {\phi}^w] $.  
	
	Thus, for $ T$ small enough, we obtain for any $ N > 1$ the estimate
	\begin{multline}\label{fest1}
	\mn{{\phi}^w u}^2   \le \mn{b_T^w{\phi}^wu}^2_{(-1/2)} + \mn{{\phi}^w u}^2   + \mn{D_x{\phi}^w u}^2 \\ \le C_N \left(
	T  \mn {{\phi}^wP^* u}^2_{(1/2)} + T\mn{[P^*, \phi^w]u}_{(1/2)}^2 + \mn { u}^2_{(-N)}  \right)
	\end{multline}
	This gives the estimate~\eqref{solvest1} with $ R = T \w{D}^{1/2}[P^*, {\phi}^w]  \in \Psi^{3/2} $ which completes the proof of Theorem~\ref{mainthm}.
\end{proof}

Next we shall derive a semiclassical estimate for the proof of Proposition~\ref{maincor}. We shall assume that the coordinates are chosen as in Proposition~\ref{prepprop} so that $ \st = \set{\xi  = 0} $. The proof involves a second microlocalization 
near~$(t_0, x_0,y_0; \tau_0, 0, {\eta}_0) = (z_0; \zeta_0)\in \st $ using the homogeneous metric $ g = g_{1,0} $. 
Then $ \sup g/g^\sigma = h^2 \le 1$  are constant and $ |\xi | \ls h^{-1}   \cong \w{(\tau, {\eta})} $.
We have $ g/h =  g^\sh  \cong g_{1/2,1/2} $ where $ g^\sh =(g^\sh)^\sigma $ is constant, $S^k_{1,0}= S(h^{-k}, g)$ and $S^k_{1/2,1/2} = S(h^{-k}, g^\sh)$ for $k \in \br $.
Observe that now the the symbols of the error terms $ R $ can be written $ \w{R_{2}\xi, \xi} + R_{1}\cdot \xi + R_{0} $  where $ R_{j}\in S(h , g ) $.

\begin{prop}\label{mainprop} 
	Assume that $P^* \cong D_t + A^w + i f_1^w $ modulo operators with symbols in $ R $. Here $f_1 = f + f_0$ is real,  $f  \in
	S(h^{-1}, g)$ is independent of $ \xi  $ satisfying  condition $ \subr(\ol \Psi) $ in~\eqref{pcond0}, and $ f_0 = \partial_\eta f \cdot r\cdot \xi \in S(h^{-1}, g) $ with $ r \in S(1,g) $. We also have 
	\begin{equation}\label{adef1}
	A = \sum_{jk} a_{jk}\xi_j\xi_k + \sum_j a_j\xi_j + a_0
	\end{equation}
	where $ a_{jk} $ and $ a_j  \in S(1,g)$ are real and $ \{ a_{jk} \}_{jk} $ is symmetric and  nondegenerate. Here $ g = g^\sharp /h $ with $0 < h \le 1$ and $g^\sh = (g^\sh)^{\sigma}$ constant. Then there
	exist \/ $T_0 > 0$ and real valued symbols $b_T(t,x,{\xi}) \in
	S(h^{-1/2}, g^\sh) \bigcap S^+(1, g^\sh)  + S(h^{-1/2}, g)$ uniformly for any 
	$0 < T \le T_0$  and $ |x| \le T $, so that
	\begin{equation}\label{propest}
	h^{1/2}\left(\mn{b_T^wu}^2 + \mn {D_xu}^2+ \mn u^2\right) \le C_0 T \im\sw{P^* u,b_T^wu}  + \mn{\Psi^w u}^2
	\end{equation} 
	when $u \in C_0^\infty$ has support where
	$|t| \le T$ and $ |x| \le T $. Here 
	$ \Psi \in S^{2} $, $ \st \bigcap \supp \Psi = \emptyset $ and  
	$C_0$, $T_0$ and the seminorms of~$b_T$ only depend on the seminorms of $f$ in $S(h^{-1}, g)$.
\end{prop}

Proposition~\ref{mainprop} will be proved at the end of Section
~\ref{lower}. 

\begin{proof}[Proof of that Proposition~\ref{mainprop} gives Proposition~\ref{maincor}]
	First note that in the estimate \eqref{corest}, $ P^* $ can be perturbed by 
	operators with symbols in $ R $. In fact, if
	$\wt R =  \w{R_2D_x, D_x} + \w{R_1, D_x} + R_0 $ with $ R_j \in  \Psi^{-1}_{1,0}$, then $\re b_T^w \wt R = \w{S_2D_x, D_x} + \w{S_1, D_x} + S_0  $ where $ S_j \in  \Psi^{0}_{1/2,1/2}$ is continuous on $ L^2 $. Thus, this term can be estimated by the last two terms in the left hand side of~\eqref{corest} for small enough $  T $.
	
	 As before, we shall include $ \tau $ in the variables $ \eta $ and use  the coordinates $ (z,\zeta) \in T^*\br^n$. 
	For the localization, we shall take ${\phi} \in S^0_{1,0}$ such that $0 \le {\phi} \le  1 $, $ \phi = 1 $ in a conical
	neighborhood of $w_0  \in \st$ such that $ \phi $ is supported  where $ | \xi | \ls  h^{-1} $, $ \Psi = 0 $ and $ P $ is on the normal form~\eqref{preppropform}.  
	By taking $ \phi(z,\zeta) = \phi_0(z)\phi_1(\zeta) $ as a product, we obtain that $ \partial_{t,x}\varphi $ has support where $\max( | t | , | x| )| \ge   T_0$ for some $ T_0 > 0 $. 
	
	Next, we shall  microlocalize in $\zeta=  (\tau, \xi,  \eta) $ with respect to the homogeneous metric $  g= g_{1,0} $ with  a partition of unity  $\set{{\varphi}_j (\zeta) }_{j} \in S^0_{1,0} = S(1,g)$ independent of $ z $ with values in $ \ell^2 $ such that $  \sum_j \varphi_j^2 = 1$,  $0 \le {\varphi}_ \le  1 $ and $ \varphi_j $ is supported where  $\w{{\zeta}} \cong h_j^{-1} $. 
	Then we can get a partition of unity in a conical neighborhood of $ w_0 $ by putting $ \phi_j = \phi \varphi_j $, 
	so that $ \phi_j $ is supported where $ | \xi | \ls h_j^{-1}$, 
	$ \sum_j \phi_j^2 = \phi^2 $  and  $ \partial_{t,x}\phi_j $ has support where $\max( | t | , | x| )| \ge   T_0$.
	
	Since the functions $ \phi_j  $ are real, we find from the calculus and symmetry that $ \sum_j\phi_j^w \phi_j^w  = \phi^w\phi^w + r^w $ where $ r \in S^{-2} $ is real valued,  which gives
	$	\mn {\phi^w v}^2 \le \sum_j \mn {\phi_j^w v}^2 + C\mn{ v}_{(-2)}^2$ for  $ v  \in C^\infty_0 $ 
	and by continuity we have $ \sum_j \mn {\phi_j^w v}^2  \ls  \mn {v}^2$.	 	
	By cutting off, we find that 
	\begin{equation}\label{cutoff}
		 \mn {v }^2_{(-2)} \ls \sum_j  \mn{h_j^{2}\phi_j^w v}^2 + \mn{\w{D}^{-2}(1- \phi)^w v}^2
	\end{equation}
    Since the cut-off functions have values in $ \ell^2 $ the calculus gives that the operators that we obtain from these will have values in  $ \ell^2 $ (or scalar values after summation) by Remark~\ref{vvcalc}.

	By possibly shrinking $ T_0 $ we can also choose real symbols $\set{{\psi}_j }_{j} \in
	S^0_{1,0} $ with values in ~$\ell^2$, such that  $0 \le {\psi}_j \le 1$ has support in a $ g $ neighborhood of $ w_j $  of radius $ 2T_0 $ so that ${\psi}_j{\phi}_j = {\phi}_j$. If $ T_0 $ is small enough, we may assume that $ P $ is on the normal form~\eqref{preppropform} and $ g_{1,0}  \cong g = hg^{\sh}$ is constant in $ \supp \psi_{j} $,  and that there is a fixed bound on number of overlapping supports of $ \psi_j $, see \cite[Section
	18.5]{ho:yellow}. Then we obtain
	that $S^m_{1,0} = S(h_j^{-m}, g_j)$ and $S^m_{1/2,1/2}
	= S(h_j^{-m}, g^\sh)$ in $\supp {\psi}_j$ for $m \in \br$, where
	$h_j  \le 1$, $g_j= h_j g^\sh $.
	
	The microlocalization of $ P $ is
    $
	P_{j} = D_t	+ A_j^w + if_j^w 
	$ 
	where $ A_j =  \psi_jA  + (1 - \psi_j)A_{0,j}\in S(\w{\xi}^2, g_j ) $  with $A_{0,j}(t,x,y; \eta, \xi) = \sum_{k\ell} a_{k\ell}(w_j) \xi_k \xi_\ell$,
	$f_{1j} = {\psi}_j f_1 \in S(h_j^{-1}, g_j ) $ uniformly in~$j$
	satisfying condition $ \subr(\Psi) $. If the support of~$ \psi_j $ is small enough, then the Hessian $ \partial_{{\xi}}^2 A_j $  is nondegenerate at~$ \st $.    
    	
  Then, by using Proposition~\ref{mainprop} with $P_{j}$ and substituting $  {\phi}_j^w u$ in~\eqref{propest}, we obtain real $b_{j,T} \in S(h_j^{-1/2}, g_j^\sh) \bigcap	S^+(1, g_j^\sh) + S(h_j^{1/2}\w{\xi}, g_j)$ uniformly so that 
	\begin{multline}\label{mmest}
	\mn {b_{j,T}^w{\phi}_j^w u}^2 + \mn {{\phi}_j^w u}^2 + \mn { D_x{\phi}_j^wu}^2 \\  \le C_0T h_j^{-1/2}\im\sw{P_{j}^*{\phi}_j^wu,
		b_{j,T}^w {\phi}_j^wu} + C_N \mn{{\phi}_j^wu}^2_{(-N)}
	\end{multline} 
	for $u\in C_0^\infty$ having support where
	$ \max (|t|,  |x| ) \le T \le T_0 $.
	We have $ P_j^* \phi_j^w = \phi_j^w P_j^* + Q_j ^w$, where 
	\begin{equation}
	Q_j^w = [D_t,  \phi_j^w]  + [ A_j^w,  \phi_j^w ] - i [f_{1j}^w, {\phi}_j^w ] \in \op S(h_j^{-1}, g_j)
	\end{equation}
	Since the commutator of symmetric operators is antisymmetric, the calculus gives that $\im  Q_j  \in S(h_j\w{\xi}^2, g_j)$ when $\max( | t | , | x| )| \le   T_0$ since then $ \partial_{t,x} \phi_j = 0  $ which gives $[ A_j^w,  \phi_j^w ] =  \sum_{k\ell}[ a_{k\ell}^w,  \phi_j^w ] D_{x_k} D_{x_\ell}$ and $  [D_t,  \phi_j^w] = 0 $. This also gives $ D_x{\phi}_j^w u = {\phi}_j^w D_x u $ when  u is supported  where  $\max( | t | , | x| )| \le   T_0$.
	We also have  $\re  Q_j  \in S(1, g_j)$ since $\w{\xi} \ls h_j^{-1}  $ in $ \supp \phi_j $. 
	
	By using the calculus we obtain that  $
	{\phi}_j^wP = {\phi}_j^w P_{j}$ modulo $\op S(h_j^N, g_j ),\ \forall N$, since ${\psi}_j{\phi}_j = {\phi}_j$.	We obtain for any $ N $ that
	\begin{multline}\label{mmest1}
	\mn {b_{j,T}^w{\phi}_j^w u}^2 + \mn {{\phi}_j^w u}^2 + \mn {{\phi}_j^w D_x u}^2\\ \le C_0T \left(\im\sw{P^*u, B_{j,T}^w u} +h_j^{-1/2}\im \sw {Q_j^w u, b_{j,T}^w {\phi}_j^wu}\right) + C_N \mn{h_j^N{\phi}_j^w u}^2 \qquad \forall\,j
	\end{multline} 
	if $u\in C_0^\infty$ supported where
	$\max (|t|,  |x| )\le T  \le T_0$. Here 
	$$ B_{j,T}^w =  h_j^{-1/2}{\phi}_j^w	b_{j,T}^w {\phi}_j^w  \in  \op S(h_j^{-1}, g_j^\sh) \bigcap \op  S^+(h_j^{-1/2}, g_j^\sh) + \op S(\w{\xi}, g_j)
	$$ 
	uniformly and by symmetry $  B_{j,T} $ is real. Since $ {\phi}_j^w b_{j,T}^w  \cong (b_{j,T}{\phi}_j)^w $ modulo $ \op S(h_j^{1/2}, g_j^\sh) $ and $\re  Q_j  \in S(1, g_j) $ we find that $\{ h_j^{-1/2}\im  {\phi}_j^wb_{j,T}^wQ_j^w\}_j \in  \op S(\w{\xi}^2, g^\sh)$  with values in~$ \ell^2 $ when $\max( | t | , | x| )| \le   T_0$. Thus we may find $ \psi \in S^1 $ with support outside a conical neighborhood of $ w_0 $ so that
	$$
	\sum_j h_j^{-1/2}\im \sw {Q_j^w u, b_{j,T}^w {\phi}_j^wu}\ls \mn{u}^2 + \mn {D_xu}^2 +  \mn{\psi^w u}^2 
	$$
	if $u\in C_0^\infty$ supported where
	$\max (|t|,  |x| )\le T_0$.	

	Let $b_{T}^w = \sum_{j}B_{j,T}^w$, then 
	by the finite bound on the overlap of the supports we find that 
	\begin{equation}
    \mn{b_T^wu}_{(-1/2)}^2 \ls \left\| \sum_{j}h_j^{1/2}B_{j,T}^w u \right\|^2 =   \left\| \sum_{j}{\phi}_j^w	b_{j,T}^w {\phi}_j^w u \right\|^2 \\
    \lesssim \sum_j\mn{b_{j,T}^w{\phi}_j^w u}^2 + \mn{u}_{(-N)}^2   
	\end{equation}
	since $ \w{B_{j,T}^w u, B_{k,T}^w u} = 0$ if $ | j - k | \gg 1 $.
 Thus, by summing up we obtain
	\begin{multline}\label{2.12}
	\mn{b_T^wu}^2_{(-1/2)} + \mn {u}^2 + \mn {D_xu}^2 \\  \le C_1 \left(T(\im \sw{P^*u, b_{T}^wu} + \mn {u}^2 + \mn {D_xu}^2 + \mn{\psi^w u}^2) + \mn{u}_{(-N)}^2  + \mn{(1-\phi)^w u}_{(1)}^2 \right)
	\end{multline}
	for $u \in  C_0^\infty$ having support where
	$ \max (|t|,  |x| ) \le T  \le T_0 $. Here we find that $ w_0 \notin \wf \psi^w \ \bigcup  \wf (1-\phi)^w $  which gives~\eqref{corest} for small enough ~$T$.
	We also have that $b_{T}^w =
	\sum_{j}^{}h_j^{-1/2}{\phi}_j^wb_{j,T}^w{\phi}_j^w \in \Psi^{1}_{1/2,1/2}$ 
	since  ${\phi}_j \in S(1, g_j)$ is supported where $\w{{\zeta}} \simeq h_j^{-1}$ and   $b_{j,T} \in S(h_j^{-1/2},g^\sh)$.
	The homogeneous gradient $\nabla b_{T} \in S^{1}_{1/2,1/2}$ since 
	 $b_{j,T} \in S^+(1,g^\sh) $ and the homogeneous gradient is equal to $ h^{-1/2} $ times the gradient in coordinates which are  $ g^\sh $ orthonormal.  This finishes the proof of  Proposition~\ref{maincor}.
\end{proof}

It remains to prove Proposition~\ref{mainprop}, which will be done at
the end of Section ~\ref{lower}. The proof involves the construction of a
multiplier $b^w_T$, and it will occupy most of
the remaining part of the paper.

\section{The symbol classes}\label{symb}

In this section we shall define the symbol classes we shall use. We shall follow Section~3 in \cite{de:X} with some changes due to the different conditions and normal forms. We shall study the subprincipal symbol $ p_s = \tau + i f_1$ with $ f_1 = f + f_0 $ where $f(t,x,y, \tau, \eta) \in S(h^{-1}, g)   $ and $ f_0 = \partial_\eta f \cdot r\cdot \xi $  where  $\partial_\eta f \cdot r \in S(1, g) $.  
The metric is localized and assumed to be constant, but the result holds in general for $ \sigma $ temperate metrics $ g \le h^2 g^\sigma $.

Since we are going to study the adjoint, we shall also assume that $f(t,x,\tau, w) \in S(h^{-1}, g)$ is independent of $ \xi $ and satisfies condition $ \subr(\ol  \Psi)$ in~\eqref{pcond0}. Here $ g = g_{1,0} $ is the usual homogeneous metric, $ x \in \br^{m} $, $ (t,\tau) \in T^*\br $ and $ w =(y,\eta) \in T^*\br^{n-m-1} $ as in Section~\ref{prep}.  We have that $ g = h g^\sh $ where $ g^\sh \le (g^{\sh})^\sigma $,  in the case of the homogeneous metric we have
\begin{equation}\label{gdef}
g^\sh = (dt^2 + |dx|^2 + |dy|^2)/h + h(d\tau^2 + |d\xi |^2 + |d\eta|^2)
\end{equation}
We shall construct a metric, weight and multiplier adapted to $ f $, so the symbols in this section will be independent of $ \xi $ except for $ f_0 $, which will be handled as an error term in the estimates, see Remark~\ref{f0rem}. 
We shall suppress the  $ \xi $ variables and assume that we have choosen $g^\sh$ orthonormal coordinates so that $g^\sh $ is the euclidean metric so that $g^\sh (t,x, \tau, w) =  | (t,x, \tau, w)   |^2 $. Then we have $ |f'| = | f|^{ g^\sh}_1\ls h^{-1/2}$,  $ | f''| = | f|^{ g^\sh}_2 \ls 1 $  and  $|f^{(k)}| = | f|^{ g^\sh}_k \ls h^{-1 + k/2} \ls h^{1/2}$ for $ k > 2$.
By decreasing ~$h$ we may obtain that  $| f' | \le h^{-1/2}$
which we assume in what follows. Observe that after the change of coordinates $ | \partial_\eta f | =  h^{1/2} | \partial_\eta f |^{ g^\sh} \le h^{1/2}|f  ' | \le 1$ and $ | \partial_t f | =  h^{-1/2} | \partial_tf |^{ g^\sh}\le h^{-1/2}| f ' | \le h^{-1}$.
The results in this section are uniform in the sense that they depend only on
the seminorms of $f$ in~$S(h^{-1}, g)$. 

Since we assume that  \/ $f = \im p_r $ does not change sign on the leaves of  $ \st $,  we may have the following definition of the sign of  $ f$.

\begin{defn}\label{sgndef}
	If  \/ $ f $ does not change sign on the leaves $ L $ of  $ \st $, then we define the sign function
	\begin{equation}
	\sgn (f) = 
	\left\{
	\begin{aligned}
	\pm 1 \quad &\text{ if } \/\pm f \ge 0  \text{ and } f \not\equiv 0 \quad  \text{on $ L $} \\
	0  \quad &\text{ if }  f \equiv 0 \quad \text{on $ L $}
	\end{aligned}
	\right.
	\end{equation}
	which is then constant on the leaves of $ \st $ such that\/ $ \sgn(f)f \ge 0 $.
\end{defn} 

 Let
\begin{align}\label{fsign}
&X_+ = \set{ (t,\tau,w ) : \exists\,s\le t,\ \max_x f(s,x,\tau,w) >0}\\
&X_- = \set{(t,\tau,w ) : \exists\,s\ge  t,\ \min_x f(s,x,\tau,w) < 0}
\end{align}
Observe that by the definition, $ X_\pm $ is open  in $\st$ and is a union of leaves of~$ \st $. 
By condition $ \subr(\Psi) $
we find that $\pm f(t,x, \tau, w) \ge 0$ when  $(t, \tau, w) \in X_\pm$ and that $X_-\bigcap X_+ = \emptyset$.
Let $X_0= \st \setminus \left(X_+\bigcup
X_-\right)$ which is a union of leaves and  is relatively closed in
$\st$.

\begin{defn} \label{d0deforig} 
	Let
	\begin{equation}\label{d0def}
	d(t,\tau, w) = \inf\set{| (t,\tau, w) - (s,\sigma, z) |: (s,\sigma, z) \in X_0}
	\end{equation}
	be the $g^\sh$ distance to $ X_0$,
	it is constant in $ x $ and equal to $+\infty$ in the case when 
	$X_0= \emptyset$.
	We define the signed  distance function ${\delta}(t,w)$ by
	\begin{equation}\label{delta0def}
	{\delta} =  \sgn(f) \min(d,h^{-1/2}),
	\end{equation}
	where $d$ is given by ~\eqref{d0def} and $\sgn(f)  $ by Definition~\ref{sgndef}.
\end{defn}

	We say that $a(t,x,\tau,w)$  is Lipschitz continuous 
	if it is Lipschitz with respect to the metric ~$g^\sh$. 

\begin{prop}
	The signed distance function $(t,\tau, w) \mapsto {\delta}(t,\tau, w)$ given by
	Definition~\ref{d0deforig} is Lipschitz continuous  with Lipschitz constant equal to 1. We
	also find that 
	$t \mapsto {\delta}(t,\tau, w)$ is nondecreasing, $ \delta $ is constant in $ x $, $0 \le {\delta}f$,
	$|{\delta}| \le h^{-1/2}$ and $|{\delta}| = d$ 
	when $|{\delta}| < h^{-1/2}$.
\end{prop}

\begin{proof} 
	Clearly ${\delta}f \ge 0$, and by the definition we have that
	$|{\delta}| =\min(d, h^{-1/2}) \le h^{-1/2}$ so $|{\delta}| =
	d$ when $|{\delta}| < h^{-1/2}$.  
	Now, it suffices to show the Lipschitz continuity of $(t,\tau, w)  \mapsto
	{\delta}(t,\tau,w)$ locally, and thus locally on ~$\complement{X_0}$ when $d < \infty$.  Then $ d(t,\tau,w) $ is the distance function to $ X_0 $ so it is Lipschitz continuous with constant 1.
	
	Next we show that $ t \mapsto {\delta}(t,\tau,w)$ is nondecreasing. In fact, when $ t $ increases we can only go from $ X_- $ to $ X_0 $ and from $ X_0 $ to $ X_+ $. 
	If $(t,\tau, w) \in  X_+ $ then $ c = \delta(t,\tau, w)  > 0  $  is the distance to $ \complement X_+ $. If there exists $ \varepsilon > 0$ so that $\delta(t+\varepsilon,\tau, w) < c  $ then there would exists $   (s,\sigma, z)  \notin X_+$ so that $ | (t + \varepsilon,\tau, w) - (s,\sigma, z) |  < c$. But then $ | (t,\tau, w) - (s - \varepsilon,\sigma, z) |  < c $ where $(s - \varepsilon,\sigma, z) \notin X_+  $ which gives a contradiction. By switching $ t $ to $ -t $, $ \delta $ to $ -\delta $ and $ X_+  $ to $ X_- $ we similarly find that $ \delta < 0$ is nondecreasing on $ X_- $ and $ \delta $ is of course equal to $ 0 $ on $ X_0 $. 
\end{proof}

Next, we are going to define the metric that we are going to use for $ f $.

\begin{defn}\label{g1def}
	Let
	\begin{equation}\label{h2def}
	H^{-1/2} = 1 + |{\delta}| + \frac{|f'|}{|{f''}| + h^{1/4}|f'|^{1/2} +
		h^{1/2}}
	\end{equation}
	and $G = H g^\sh$.
\end{defn}

\begin{rem} \label{H1prop}
	We have that 
	\begin{equation} \label{H1hest}
	1 \le H^{-1/2}\le 1 + |{\delta}| + h^{-1/4}  |f'|^{1/2}\le 3 h^{-1/2}
	\end{equation}
	since $|f'| \le h^{-1/2}$ and $|{\delta}| \le
	h^{-1/2}$. Moreover, $|f'| \le H^{-1/2}(|{f''}| + h^{1/4}|f'|^{1/2}
	+ h^{1/2})$ so by the Cauchy-Schwarz inequality we obtain
	\begin{equation}\label{dfest0}
	|f'| \le 2 |{f''}| H^{-1/2} + 3h^{1/2}H^{-1} \le C_2H^{-1/2} 
	\end{equation}
which gives that $ f \in S^+(H^{-1}, G) $, see Definition~\ref{s+def}.
\end{rem}

Since the metric $ G $ does not depend on  the values of $ f  $, we shall need a weight to define the symbol class of $ f $.

\begin{defn} \label{Mdef}
	Let
	\begin{equation}\label{Mdef0}
	M = |f| +|f'| H^{-1/2} + |{f''}| H^{-1} + h^{1/2}H^{-3/2}  
	=|f| +|f'|^G_1 + |{f''}|^G_2 + h^{1/2}H^{-3/2}
	\end{equation}
	then we have that $h^{1/2} \le M \le C_3 h^{-1}$ by~\eqref{H1hest}.
\end{defn}

In the following, we shall simplify the notation and include the variables  $ x $, $ t $ and $ \tau $ in the $ w $ variables.

\begin{prop}\label{g1prop}
	We find that $H^{-1/2}$ is Lipschitz continuous,
	$G$ is ${\sigma}$ ~temperate  such that $G =
	H^2G^{\sigma}$ and
	\begin{equation}\label{tempest}
	H(w) \le C_0 H(w_0)(1 + G_w(w-w_0))
	\end{equation}
	We have that $M$ is a weight for $G$ such that $f \in S(M,G)$
	and
	\begin{equation}\label{mtemp}
	M(w) \le C_1 M(w_0)(1 + G_{w_0}(w-w_0))^{3/2}
	\end{equation}
	In the case when $1 + |{\delta}(w_0)| \le
	H^{-1/2}(w_0)/2$  we have $ |f'(w_0)| \ge  h^{1/2}$,
	\begin{equation} \label{dfsymbolest}
	|f^{(k)}(w_0)| \le C_k  |f'(w_0)|H^{\frac{k-1} {2}}(w_0)\qquad\text{$k
		\ge 1$}
	\end{equation} 
	and $1/C \le   | f'(w)| /   |f'(w_0| \le C$ when $|w-w_0| \le
	c H^{-1/2}(w_0)$ for some $c>0$.
\end{prop}

\begin{rem}\label{f0rem}
     The term $ f_0 = \partial_\eta f\cdot r\cdot \xi \in S(h^{-1}, g)$  in Proposition~\ref{mainprop} can be written $f_0 = r_0\cdot \xi $  with $r_0 \in S(MH^{1/2} h^{1/2},G) \subset S(1,G) $  since $ |\partial_\eta f | \le h^{1/2}| f ' | $. Now, $ \w{\xi} $ is not a weight for $ g^\sh $ near $ \st $ but $ f_0 \in S(MH^{1/2} h^{1/2}\w{\xi},G_0 ) $  where $ G_0 $ is given by~\eqref{G0def}.
\end{rem}

Since $G \le g^\sh \le
G^{\sigma}$ we find that the conditions ~\eqref{tempest}
and~\eqref{mtemp} are stronger than the property of being
${\sigma}$~temperate (in fact, it is strongly ${\sigma}$ ~temperate in the
sense of \cite[Definition~7.1]{bc:sob}), and imply that $ G $ is slowly varying and $ M $ is $ G $ continuous.
When $1 + |{\delta}| < H^{-1/2}/2$ we find from~\eqref{dfsymbolest} that $ |f' | > 0$ is a weight for~$ G $, $f' \in S(|f'|,G)$  and $f^{-1}(0)$ is a $C^\infty$ hypersurface. Since that surface does not depend on $ x $ we find that $ \min_x H^{1/2}$
gives an upper bound on the curvature of~$f^{-1}(0)$ by
~\eqref{dfsymbolest}. Proposition~\ref{mproplem} shows that
~\eqref{dfsymbolest} also holds for $k = 0$ when $1 + |{\delta}| \ll
H^{-1/2}$.

\begin{proof}
	First we note that if $ H^{-1/2} $  is Lipschitz continuous, then 
\begin{equation}\label{tempest0}
	H^{-1/2} (w_0) \ls H^{-1/2} (w) + C_1|w-w_0|
\end{equation}
     which gives \eqref{tempest} since $ G = Hg^\sh $.
	Next, we shall show that $H_1^{-1/2}$ is Lipschitz
	continuous. Since the first terms of ~\eqref{h2def} are Lipschitz continuous, we
	only have to prove that 
	\begin{equation*}
	{|f'|}/(|{f''}| + h^{1/4}|f'|^{1/2} +
	h^{1/2}) = E/D
	\end{equation*}
	is Lipschitz continuous. Since this is a local property, it suffices to prove 
	this when $|{\Delta}w| = |w-w_0| \le 1$. Then we have that 
	$D(w) \cong D(w_0)$, in fact $D^2 \cong h + h^{1/2}|f'| + |f''|^2$ so
	\begin{equation*}
	D^2(w) \le C(D^2(w_0) +  |f''(w_0)| h^{1/2} + h) \le C'D^2(w_0)
	\end{equation*}
	when $|{\Delta}w| \le 1$. We find that
	\begin{equation*}
	\left|{\Delta}\frac{E}{D} \right| = \left| \frac{E(w)}{D(w)} - \frac{E(w_0)}{D(w_0)} \right| \le
	\frac{|{\Delta}E|}{D(w)} + \frac{E(w_0)|{\Delta}D|}{D(w)D(w_0)} 
	\end{equation*}
	Taylor's formula gives that
	\begin{equation}\label{estDE}
	|{\Delta}E| \le (|f''(w)| + Ch^{1/2})|{\Delta}w| \le CD(w)
	\end{equation}
	when $|{\Delta}w| \le 1$. It remains to show that $E(w_0)|{\Delta}D| \le C
	D(w)D(w_0)|{\Delta}w|$, which is trivial if $E(w_0) = 0$.
	Else, we have
	\begin{equation*}
	|{\Delta}|f''|| \le Ch^{1/2}|{\Delta}w|  \le C D^2(w_0)|{\Delta}w| /E(w_0)
	\le C' D(w_0) D(w)||{\Delta}w|/E(w_0)
	\end{equation*}
	when $|{\Delta}w| \le 1$ since $h^{1/2} \le D^2/E$ and $D(w_0) \le CD(w)$. Finally, we have 
	\begin{multline*}
	h^{1/4}|{\Delta}|f'|^{1/2}| \le  h^{1/4}|{\Delta}E|/(|f'(w_0)|^{1/2} + |f'(w)|^{1/2}) \\
	\le C h^{1/4}|f'(w_0)|^{1/2}D(w) |{\Delta}w|/|f'(w_0)|   \le 
	C D(w_0) D(w)|{\Delta}w| /E(w_0)
	\end{multline*}
	when $|{\Delta}w| \le 1$ by ~\eqref{estDE}, which proves the Lipschitz continuity.

	Next, we study  the case when $1 + |{\delta}(w_0)| \le H^{-1/2}(w_0)/2$, then 
	$ H^{1/2}(w_0) \le 1/2$.
     Then we find from~\eqref{h2def} that
	\begin{equation} \label{h1dfest}
	|{f''(w_0)}| +h^{1/4}|f'(w_0)|^{1/2} + h^{1/2} 
	\le 2 H^{1/2}(w_0)|f'(w_0)| \le |f'(w_0)|
	\end{equation}
	which gives $|f'(w_0)| \ge h^{1/2}$,  $|f'(w_0)| \ge |{f''(w_0)}|$ and that $ h^{1/2} \le 4 H(w_0) 
	| f'(w_0) | $.  When $ | w - w_0 | \le cH^{-1/2}(w_0) $  we find from~\eqref{h1dfest} by using Taylors formula that
	$$
	|f'(w) - f'(w_0) | \le | f''(w_0) | cH^{-1/2}(w_0) + C_3h^{1/2}c^2 H^{-1}(w_0) \le (c + 4C_3c^2)	| f'(w_0) |
	$$
	which gives  $1/C \le   | f'(w)| /   |f'(w_0| \le C$ for small enough $ c > 0 $.
	Now~\eqref{dfsymbolest} follows from
	~\eqref{h1dfest} for $k = 2$. When $k \ge 3$ we have
	$$
	|f^{(k)}(w_0)| \le C_kh^{\frac {k-2}2} \le 4C_k 3^{k-3} |f'(w_0)| H^{\frac{k-1}2}
	$$
	since $h^{1/2} \le 4 H |f'(w_0)| $ by~\eqref{h1dfest} and $h^{(k-3)/2}
	\le 3^{k-3}H^{(k-3)/2}$ by~\eqref{H1hest}.
	
	Finally, we shall prove that ~$M$ is a weight for ~$G$.
	By Taylor's formula we have
	\begin{equation}\label{taylorest0}
	|{f^{(k)}(w)}| \le C_4\sum_{j=0}^{2-k}
	|{f^{(k+j)}(w_0)}| | w-w_0|^{j} + 
	C_4 h^{1/2}   | w-w_0|^{(3-k)} \qquad 0 \le k \le 2	
	\end{equation}
	thus we obtain that
	\begin{equation*}
	M(w) \le C_5 \sum_{k=0}^{2}|{f^{(k)}(w_0)}|( | w-w_0| + H^{-1/2}(w))^k
	+ C_5h^{1/2} (  | w-w_0| + H^{-1/2}(w))^3
	\end{equation*}
	By switching $w$ and $w_0$ in~\eqref{tempest0}  we find
	$H^{-1/2}(w) +  | w-w_0| \le C_0 (H^{-1/2}(w_0) +  | w-w_0|)$. Thus we
	obtain that
	\begin{multline*}
	M(w) \le C_6  \sum_{k=0}^{2}|{f^{(k)}(w_0)}| H^{-k/2}(w_0) (1 +
	H^{1/2}(w_0)  | w-w_0|)^k  \\ + 
	C_6h^{1/2}H^{-3/2}(w_0) (1 + H^{1/2}(w_0)  | w-w_0|)^3 \le C_6M(w_0) (1 +
	G_{w_0}(w-w_0))^{3/2} 
	\end{multline*}
	which gives ~\eqref{mtemp}.
	It is clear from the definition
	of $M$ that  
	$ 
	|{f^{(k)}}| \le MH^{k/2}$ when $k \le 2 
	$, 
	and when $k \ge 3$ we have
	$
	|{f^{(k)}}| \le C_kh^{\frac{k-2}{2}} \le C_k3^{k-3} M H^{\frac{k}{2}}
	$
	since  $ h^{1/2} \le MH^{3/2}$ and $h^{(k-3)/2} \le 3^{k-3}H^{(k-3)/2}$ when $k \ge 3$. 
\end{proof}

\begin{prop}\label{mproplem}
	We have that 
	\begin{equation}\label{Mcomp}
	1/C \le M /  (|{f''}| H^{-1} + h^{1/2}H^{-3/2}) \le C
	\end{equation}
	When  $|{\delta}| \le {\kappa}_0H^{-1/2}$ and $H^{1/2} \le {\kappa}_0$ for\/ $0 < {\kappa}_0  \le 1/4$ we find that
	\begin{equation} \label{Mcomp0}
	1/C_1 \le M/ |f'|H^{-1/2} \le C_1
	\end{equation}
	which implies that $f \in  S(H^{-1}, G)$ by~\eqref{dfest0}.
\end{prop}

\begin{proof}
	First note that  when $|{\delta}| \cong h^{-1/2}$ we have $H^{-1/2}
	\cong h^{-1/2}$, which gives $M \cong h^{-1}$ and proves~\eqref{Mcomp}
	in this case.  If $|{\delta}(w_0)| < h^{-1/2}$, then as before there
	exists $w \in f^{-1}(0)$ such that $|w-w_0| = |{\delta}(w_0)| \le
	H^{-1/2}(w_0)$. Since $f(w)=0$, Taylor's formula gives that 
	\begin{equation}\label{tayloref}
	|f(w_0)| \le |f'(w_0||{\delta}(w_0)|  + |{f''(w_0)}|
	|{\delta}(w_0)|^2/2 + Ch^{1/2}|{\delta}(w_0)|^3
	\end{equation}
	We obtain from ~\eqref{tayloref}
	and~\eqref{dfest0} that
	$$ M \le C  \big(|f'|H^{-1/2} + |{f''}| H^{-1} + h^{1/2}H^{-3/2}
	\big) \le C'  \left(|{f''}| H^{-1} + h^{1/2}H^{-3/2}\right) 
	\qquad\text{at ~$w_0$}
	$$
	which gives~\eqref{Mcomp} since the opposite estimate is trivial.
	
	If $|{\delta}| \le {\kappa}_0H^{-1/2} < h^{1/2}$ and $H^{1/2} \le {\kappa}_0$ with $ {\kappa}_0 \le 1/4 $
	then  $ \w{\delta} \le H^{-1/2}/2 $ so we we obtain by 
	 \eqref{dfsymbolest},  \eqref{h1dfest}  and \eqref{tayloref} that 
	$$ M
	\le C  \big(|f'|H^{-1/2} + |{f''}| H^{-1} + h^{1/2}H^{-3/2} \big)
	\le C'   | f ' |H^{-1/2} \qquad\text{at ~$w_0$}
	$$ 
	since we have as before $ h^{1/2} \le 4 H  | f ' | $ by  \eqref{h2def}. 
	This gives~\eqref{Mcomp0} since the opposite estimate is trivial, which completes the proof of the proposition.
\end{proof}

\begin{prop}\label{ffactprop}
	Let $H^{-1/2}$ be given by Definition~\ref{g1def} for $f \in
	S(h^{-1},g)$. There exists positive ${\kappa_1} $ and $ c  $ so that if\/
	$\w{{\delta}} = 1 + |{\delta}| \le {\kappa}_1H^{-1/2}$ at $ w_0 $ then
	\begin{equation}\label{ffactor}
	f = {\alpha}{\delta} \qquad \text{when $ | w - w_0 | \le c H^{-1/2}(w_0) $}
	\end{equation}
	where $0 \le {\alpha} \in S(MH^{1/2}, G)$, such that $ \alpha \ge {\kappa}_1MH^{1/2}  $,
	which implies that 
	${\delta} = f/{\alpha} \in S(H^{-1/2}, G)$.
\end{prop}

\begin{proof}
	Let ${\kappa}_0 > 0$ be
	given by Proposition~\ref{mproplem}. If
	${\kappa_1} \le {\kappa}_0$  and $  \w{\delta}\le \kappa_1 H^{-1/2} $ at $ w_0 $ then we find that
	$  |f'(w_0)| \cong
	M(w_0)H^{1/2}(w_0)$ by \eqref{Mcomp0}. We may change coordinates so that  $ w_0 = 0 $. Let $H^{1/2} = H^{1/2}(0)$ and $M = M(0)$, $w = H^{-1/2}z$ and
	$$
	F(z) = H^{1/2}f(H^{-1/2}z)/|f'(0)| \cong f(H^{-1/2}z)/M \in C^\infty
	$$ 
	Now ${\delta}_1(z) = H^{1/2}{\delta}(H^{-1/2}z)$ is the
	signed distance to $F^{-1}(0)$ in the $z$~coordinates which is constant in $ x $.  
	
	We have	$|F(0)| \le C_0$, $|F'(0)| = 1$, $|F''(0)|\le C_2$ and $|F^{(3)}(z)|
	\le C_3$, $\forall\, z$, by \eqref{dfsymbolest} in  Proposition~\ref{g1prop}. It is no restriction to assume that the coordinates $ z = (z_1, z') $ are chosen so that $\partial_{z'}F(0) = 0$, and then $|\partial_{z_1}F(z)| \ge c > 0$ in
	a fixed neighborhood of the origin. If $|{\delta}_1(0)| =
	|{\delta}(0)H^{1/2} |\le {\kappa}_1 \ll 1$ then $F^{-1}(0)$ is a
	$C^\infty$~manifold in this neighborhood, ${\delta}_1(z)$ is uniformly
	$C^\infty$ and $\partial_{z_1}{\delta}_1(z) \ge c_0 > 0$ in a fixed
	neighborhood of the origin. 
	
	By choosing $(\delta_1(z),z')$ as local
	coordinates and using Taylor's formula we find that $F(z) =
	{\alpha}_1(z){\delta}_1(z)$ for any $ x $ since $ F = 0 $ when $ \delta_1 = 0 $. Here $0  \le {\alpha}_1 \in C^\infty$  and $ \alpha_1 \ge C > 0 $ in a	fixed neighborhood of the origin.  
	Thus
	$$
	 f (w)= |f '(0)| H^{-1/2}\alpha_1(H^{1/2}w)\delta_1(H^{1/2}w) =  |f '(0)| \alpha_1(H^{1/2}w)\delta(w)
	$$
	when $ |w | \le c H^{-1/2} $. Now  $ \alpha(w) =  |f '(0)| \alpha_1(H^{1/2}w) \in S(MH^{1/2}, G) $ with $   \alpha  \cong MH^{1/2}$ which gives the proposition.
\end{proof}

\section{Properties of the symbol}\label{local}

In this section we shall study the properties of the symbol near
the sign changes. 
We shall follow Section~ in \cite{de:X} with some minor changes,
and we shall start with a one-dimensional result.

\begin{lem} \label{lem1}
	Assume that $f(t) \in C^3(\br)$ such that
	$\mn{f^{(3)}}_\infty = \sup_t |{f^{(3)}(t)}|$ is bounded. If
	\begin{equation}\label{fsgncond0}
	\sgn(t)f(t) \ge 0\qquad\text{when ${\varrho}_0 \le |t|
		\le {\varrho}_1$}
	\end{equation} 
	for ${\varrho}_1 \ge 3{\varrho}_0 >0$, then we find
	\begin{align}
	&|f(0)| \le \frac{3}{2}\left({\varrho}_0 f'(0) + 
	{\varrho}_0^3\mn{f^{(3)}}_\infty /2\right) \label{f0est}\\
	&|f''(0)| \le f'(0)/{\varrho}_0 + 7
	{\varrho}_0 \mn{f^{(3)}}_\infty/6 \label{d2fest}
	\end{align}
\end{lem}

\begin{proof}
	By Taylor's formula we have
	\begin{equation*}
	0 \le \sgn({t})f({t}) = |{t}| f'(0) +
	\sgn({t})(f(0) +f''(0){t}^2/2) + R({t}) \qquad
	{\varrho}_0 \le |{t}| \le {\varrho}_1
	\end{equation*} 
	where $|R({t})| \le  \mn{f^{(3)}}_\infty |{t}|^3/6$. We obtain that
	\begin{equation}\label{5.2a}
	\left|f(0)+ t^2 f''(0)/2\right| \le
	f'(0)|{t}| + \mn{f^{(3)}}_\infty |{t}|^3/6
	\end{equation} 
	for any $|{t}| \in [{\varrho}_0, {\varrho}_1]$. 
	By choosing
	$|{t}| = {\varrho}_0$ and $|{t}| = 3{\varrho}_0$, we obtain that
	$$4{\varrho}_0^2 |f''(0)| \le 
	4 f'(0){\varrho}_0 + 28 \mn{f^{(3)}}_\infty {\varrho}_0^3/6
	$$ 
	which gives~\eqref{d2fest}. By letting $|t| ={\varrho}_0$  in
	~\eqref{5.2a} and substituting ~\eqref{d2fest}, we obtain~\eqref{f0est}. 
\end{proof}

\begin{prop}\label{newestprop}
	Let $f (w) \in C^\infty(\br^n)$ such that
	$\mn{f^{(3)}}_\infty < \infty$. Assume that there exists
	$0 <{\varepsilon}\le r/5 $ such that 
	\begin{equation}\label{newlemass}
	\sgn(w_1)f(w) \ge 0\quad\text{when $ |w_1| \ge {\varepsilon} +
		|w'|^2/r$ and $|w| \le r$}
	\end{equation}
	where $w= (w_1,w')$. Then we obtain
	\begin{equation}
	\label{ddfest}
	|{f''(0)}| \le 33(|\partial_{w_1}f(0)|/{\varrho} + {\varrho} \mn{f^{(3)}}_\infty )
	\end{equation}
	for any ${\varepsilon} \le {\varrho} \le r/\sqrt{10}$.
\end{prop}

\begin{proof}
	We shall consider the function $t \mapsto f(t,w')$ which satisfies
	~\eqref{fsgncond0} for fixed $w'$ with 
	$${\varepsilon} + |w'|^2/r = {\varrho}_0(w') \le |t| \le
	{\varrho}_1 \equiv 3r/\sqrt{10}
	$$ and $|w'| \le r/\sqrt{10}$ which we assume in what follows. In
	fact, then $t^2 + |w'|^2 \le r^2$ and $3{\varrho}_0(w') \le 9r/10 \le
	3r/\sqrt{10} = {\varrho}_1$. We obtain from \eqref{f0est} and~\eqref{d2fest}
	that
	\begin{align}
	\label{f0est0}
	&|f(0,w')| \le   \frac{3}{2}\partial_{w_1}f(0,w'){\varrho} +
	3{\varrho}^3 \mn{f^{(3)}}_\infty/4\\
	\label{d2fest0}
	&|\partial_{w_1}^2f(0,w')| \le \partial_{w_1}f(0,w')/{\varrho} +
	7{\varrho} \mn{f^{(3)}}_\infty/6\qquad \text{}
	\end{align}
	for ${\varepsilon} + |w'|^2/r \le
	{\varrho} \le  r/\sqrt{10} $ and $|w'| \le r/\sqrt{10}$.
	By letting $w' =0$
	in~\eqref{d2fest0} we find that
	\begin{equation}\label{d2w1fest}
	|\partial_{w_1}^2f(0)| \le \partial_{w_1}f(0)/{\varrho}  +
	7{\varrho} \mn{f^{(3)}}_\infty/6
	\end{equation} 
	for ${\varepsilon} \le {\varrho} \le r/\sqrt{10}$.
	By letting ${\varrho} = {\varrho}_0(w')$ in~\eqref{f0est0} and dividing
	with $3{\varrho}_0(w')/2$, we obtain that 
	\begin{equation}\label{nyest1}
	0 \le \partial_{w_1}f(0,w') + 2\mn{f^{(3)}}_\infty |w'|^2   
	\end{equation}
	when ${\varepsilon} \le |w'| \le r/\sqrt{10}$ since then ${\varrho}_0(w') \le
	{\varepsilon} + |w'| \le 2|w'|$.  By using Taylor's formula for $w'
	\mapsto \partial_{w_1}f(0,w')$ in ~\eqref{nyest1}, we find that 
	\begin{equation*}
	0 \le \partial_{w_1}f(0) + \w{w',\partial_{w'}(\partial_{w_1}f)(0)} +
	\frac{5}{2} \mn{f^{(3)}}_\infty |w'|^2 \qquad\text{} 
	\end{equation*}
	when ${\varepsilon} \le |w'| \le r/\sqrt{10}$. Thus, by optimizing over fixed
	~$|w'|$, we obtain that
	\begin{equation}\label{dd1fest}
	|w'||\partial_{w'}(\partial_{w_1}f)(0)| \le  \partial_{w_1}f(0) +
	\frac{5}{2} \mn{f^{(3)}}_\infty |w'|^2  \qquad\text{when ${\varepsilon} \le |w'| \le r/\sqrt{10}$} 
	\end{equation}
	By again putting ${\varrho} = {\varrho}_0(w')$ in~\eqref{f0est0},
	using Taylor's formula for $w' \mapsto \partial_{w_1}f(0,w')$ 
	but this time substituting ~\eqref{dd1fest}, we obtain 
	\begin{equation}\label{nyest2}
	|f(0,w')| \le 6\partial_{w_1}f(0)|w'| + 15
	\mn{f^{(3)}}_\infty |w'|^3 \qquad\text{when ${\varepsilon} \le |w'|
		\le r/\sqrt{10}$}
	\end{equation}
	We may also estimate the even terms in Taylor's formula by
	~\eqref{nyest2}:
	\begin{equation*}
	\begin{split}
	|f(0) + \w{\partial_{w'}^2f(0)w',w'}/2| \le \frac{1}{2}|f(0,w') +
	f(0,-w')| + \mn{f^{(3)}}_\infty |w'|^3/6 \\ \le
	6\partial_{w_1}f(0)|w'| + \frac{91}{6} \mn{f^{(3)}}_\infty |w'|^3
	\end{split}
	\end{equation*}
	when ${\varepsilon} \le |w'| \le r/\sqrt{10}$. Thus, by
	using~\eqref{f0est0} with ${\varrho} = {\varepsilon}$ and $w' = 0$ to
	estimate ~$|f(0)|$ and optimizing over fixed ~$|w'|$, we obtain that
	\begin{equation}\label{dfddfest}
	|{\partial_{w'}^2f(0)}||w'|^2/2 \le \frac{15}{2}
	|\partial_{w_1}f(0)||w'| + 16 \mn{f^{(3)}}_\infty
	|w'|^3
	\end{equation}
	when ${\varepsilon} \le |w'| \le r/\sqrt{10}$. Thus we obtain ~\eqref{ddfest}
	by taking ${\varepsilon} \le |w'| = {\varrho} \le r/\sqrt{10}$ in
	~\eqref{d2w1fest}--\eqref{dfddfest}.
\end{proof}

As before, if $f \in C^\infty(\br^n)$ then we define the {\em signed
	distance function} of $f$ as ${\delta} = \sgn(f)d$ where $d$ is the
Euclidean distance to $f^{-1}(0)$.

\begin{prop}\label{betalemma}
	Let $f_j(w) \in C^\infty(\br^n)$ and \/ ${\delta}_j(w)$ be the signed distance
	functions of $f_j(w)$, for $j = 1$, $2$. Assume that $ f_1(w) > 0
	\implies f_2(w) \ge 0$. There exists positive $c_0 $ and $c_1$, such
	that if\/ $|{\delta}_j(w_0)| \le c_0$, $|f'_j(w_0)| \ge c_1$,  for $j =
	1$, $2$, and
	\begin{equation} 
	\label{d0var}
	|{\delta}_1(w_0) -{\delta}_2(w_0)| = {\varepsilon} 
	\end{equation} 
	then there exist $g^\sh$~~ orthonormal coordinates $w= (w_1,w')$ so that
	$w_0 = (x_1,0)$ with $x_1 = {\delta}_1(w_0)$ and
	\begin{align} 
	&\sgn(w_1)f_j(w) > 0\quad\text{when $|w_1| \ge ({\varepsilon}+ |w'|^2)/c_0
		$ and $|w| \le c_0$}\qquad j= 1,\ 2 \label{betapmdef}\\
	&|{\delta}_2(w) - {\delta}_1(w)| \le ({\varepsilon} + |w-w_0|^2)/c_0\qquad
	\text{when $|w| \le c_0$} \label{Dest}
	\end{align}
	The constant $c_0$ only depends on the seminorms of $f_1$ and $f_2$ in
	a fixed neighborhood of ~$w_0$.
\end{prop}

\begin{proof}
	Observe that the conditions get stronger and the conclusions weaker
	when ~$c_0$ decreases. Assume that $f_1$ and ~$f_2$ are uniformly bounded in ~$C^\infty$ 
	near ~$w_0$.
	For any  positive $ c_0 $ and $ c_1 $  there exists $ c_2 > 0 $ so that if $|f'_j(w_0)| \ge c_1$ and $|{\delta}_j(w_0)|
	\le c_0 $, $j = 1$, $2$ then 
	$|f'_j(w)| >0$ for $|w- w_0| \le c_2$, thus $f_j^{-1}(0)$ is a
	$C^\infty$ hypersurface in $|w- w_0| \le c_2$. By decreasing~$c_0$ we obtain as in the proof of
	Proposition~\ref{ffactprop} that there exists $c_3 > 0$ so that $w
	\mapsto {\delta}_j(w) \in C^\infty(\br^{n})$ uniformly in $|w- w_0|
	\le c_3$, $j = 1$,~2.  We may also choose $z_0 \in f_1^{-1}(0)$ so
	that $|{\delta}_1(w_0)| = |w_0 -z_0|$, and then choose $g^\sh $
	~~orthonormal coordinates so that $z_0= 0$, $w_0 =
	({\delta}_1(w_0),0)$ and $\partial_{w'}{\delta}_1(0)
	= \partial_{w'}{\delta}_1(w_0) = 0$, $w = 
	(w_1,w')$.  If $c_0 \le c_3/3$ we find that ${\delta}_j
	\in C^\infty$ in $|w| \le c_4 = 2c_3/3$. Since $\sgn(f_1(w_0)) = \sgn
	({\delta}_1(w_0))$ we find that $\partial_{w_1}f_1(0) > 0$.
	
	We have that $|{\partial_w^2{\delta}_j(w)}| \le C_0$ for $|w| \le
	c_4$, $j = 1$, 2, and ${\Delta}(w) = {\delta}_2(w) - {\delta}_1(w) \ge
	0$ by the sign condition. 
	By ~\cite[ Lemma~7.7.2]{ho:yellow} we obtain that $| \partial_w
	{\Delta}(w)|^2 \le  C_1{\Delta}(w) \le C_1{\varepsilon}$ when $w=w_0$
	by~\eqref{d0var}. This gives
	\begin{multline}\label{Dest0}
	|{\Delta}(w)| \le |{\Delta}(w_0)| + |\partial_{w}{\Delta}(w_0)||w-w_0| +
	C_2|w-w_0|^2\\ \le C_3({\varepsilon} + |w-w_0|^2)\quad\text{for $|w| \le c_4$}
	\end{multline}
	which proves~\eqref{Dest}.  Since $|\partial_{w'}{\delta}_1(w_0)| =0$ we
	find that $|\partial_{w'}{\delta}_2(w)| \le C_4(\sqrt{\varepsilon} +
	|w-w_0|) \ll 1$ when $|w-w_0| \ll 1$ and ${\varepsilon} \le 2c_0 \ll 1$.  Now $f_2(\ol
	w) =0$ for some $|\ol w| \le {\varepsilon} $. Thus for $c_0 \ll 1$ we obtain
	$|\partial_{w'}{\delta}_2(\ol w)| \ll 1$, which gives that
	$|\partial_{w_1}f_2(\ol w)| \ge c_5 |\partial_{w}f_2(\ol w)| \ge c_5^2
	> 0$ for some $c_5 > 0$. Since $\sgn(f_2(w_1,0)) = 1$ when $w_1 > 0$, 
	we obtain that $\partial_{w_1}f_2(w) \ge
	c_6|\partial_{w}f_2(w)| \ge c_6^2$ 
	when $|w| \le c_6$ for some $c_6 > 0$.
	
	By using the implicit function theorem, we obtain 
	$b_j(w') \in C^\infty $, so that that $\pm f_j(w)
	> 0$ if and only if $\pm(w_1 - b_j(w')) \ge 0$ when $|w| \le c_7 > 0$, $j
	= 1$, $2$. Since $f_1(0) = |\partial_{w'}f_1(0)| = 0$ we obtain that
	$b_1(0) = |b_1'(0)| = 0$. This gives that $ |b_1(w')| \le C_5|w'|^2 $
	and proves the positive part of ~\eqref{betapmdef} by the sign
	condition.  Observe that the sign condition gives that  $b_1(w') \ge b_2(w')$. 
	Now $|\delta_2(w_0)| \le |\delta_1(w_0)| + \varepsilon$, thus we find
	$-\varepsilon \le b_2(\ol w') \le b_1(\ol w')$ for some $|\ol w'| \le
	C\sqrt \varepsilon \le C\sqrt{2c_0} \le c_7$ for $c_0 \ll 1$. 
	This gives $b_2(\ol w') \le C_5C^2 \varepsilon$ and $|b_1'(\ol
	w')| \le C_6\sqrt \varepsilon$, and 
	we obtain as before that
	$ |b_1'(\ol w') - b_2'(\ol w')| \le C_7 \sqrt{{\varepsilon}}$. 
	As in ~\eqref{Dest0}, we find
	$$ 
	|b_2(w')| \le C_8({\varepsilon} + |w'- \ol w'|^2) \le
	C_9({\varepsilon} + |w'|^2) \qquad \text{for $|w| \le c_7$}
	$$ 
	which proves  the negative part of ~\eqref{betapmdef} and the proposition.
\end{proof}

\section{The Weight function}\label{weight}

In this section, we shall define the weight ~$m$ we shall use. 
 We shall follow Section 5 in \cite{de:X} with some necessary changes because of the different conditions and normal forms.
We shall use the same notation as in Section~\ref{symb}, and let
${\delta}(t, \tau, w)$, $H^{-1/2}(t,x,\tau,w)$ and $ M(t,x,\tau,w) $  be given by Definitions
~\ref{d0deforig}, \ref{g1def} and \ref{Mdef}  for $f(t,x,\tau, w) \in  S(h^{-1},
hg^\sh)$ satisfying condition~$ \subr(\ol \Psi) $ given by~\eqref{pcond0} such that $|f'| \le h^{-1/2}$. As before, we shall include $ \tau $ in $ w $ but $ t $ will be a parameter.

The weight~$m$ will essentially measure how much $t \mapsto
{\delta}(t,w)$ changes between the minima of $t \mapsto
H^{1/2}(t,x,w)\w{{\delta}(t,w)}^2$, which will give restrictions on
the sign changes of the symbol.  As before, we assume that we have
choosen $g^\sh$ orthonormal coordinates so that $g^\sharp(w) $ is the euclidean metric,
and the results will only depend on the
seminorms of~ $f$. 
The following definition uses that $ t \mapsto \delta(t,w) $ is nondecreasing and $ \delta $ is constant in~$ x $, and assumes that $ H $ is  only defined in $ |x | \le C $.

\begin{defn}\label{h0def}
	Let $ H_1(t,w) = \min_x H(t,x,w) $ and
	\begin{multline}\label{mphodef}
	m(t,w) = \inf_{t_1 \le t \le
		t_2} \big\{ {\delta}(t_2,w) - {\delta}(t_1,w) \\ +
	\max\big(H_1^{1/2}(t_1,w)\w{ {\delta}(t_1,w)}^2,
	H_1^{1/2}(t_2,w)\w{ {\delta}(t_2,w)}^2\big)/2 \,\big\}
	\end{multline}
	where $\w{{\delta}} = 1 + |{\delta}| \le H_1^{-1/2} =  \max_x H^{-1/2}$.  Thus $ m $ is constant on the leaves of $ \st $.
\end{defn}

This is actually Definition~5.1 in \cite{de:X} with $ H $ replaced by $ H_1 $. It will be important in the proof that this weight is constant on the leaves of $ \st $.

\begin{rem} 
	When $t \mapsto {\delta}(t,w)$ is constant for fixed ~$w$, we find
	that $t \mapsto m(t,w)$ is equal to the largest quasi-convex
	minorant of $t \mapsto H_1^{1/2}(t,w)\w{{\delta}(t,w)}^2/2$, i.e., $\sup_I
	m = \sup_{\partial I} m $ for compact intervals $I \subset \br$,
	see~\cite[Definition~1.6.3]{ho:conv}.
\end{rem}

\begin{rem}\label{mdefrem}
	One can also make a local definition of $ m $ by taking the infimum over $ -T \le t_1 \le t \le
	t_2 \le T $ in~\eqref{mphodef} for some $0 <  T \ll 1$. Then the results of this  section will hold when $ |t| \le T $. By making a translation in $ t $ we can of course define $ m $ in a neighborhood of any point.
\end{rem}

\begin{prop}\label{weightprop}
	We have that $m \in L^\infty $, such
	that $w \mapsto m(t,w)$ is uniformly Lipschitz continous, $\forall\, t$,
	and
	\begin{equation} \label{hhhest} 
	h^{1/2}\w{ {\delta}}^2/6 \le
	m \le H_1^{1/2}\w{ {\delta}}^2/2 \le H^{1/2}\w{ {\delta}}^2/2 \le \w{ {\delta}}/2 
	\end{equation}  
	We may for any $ (t_0,w_0) $ choose $t_1 \le t_0 \le t_2$
	so that 
	\begin{equation}\label{6.11a}
	\max_{j= 0,1,2}\w{{\delta}(t_j,w_0)} \le  2\min_{j=
		0,1,2}\w{{\delta}(t_j,w_0)} 
	\end{equation}
	and
	\begin{equation} \label{defH0}
	H_0^{1/2} = \max(H_1^{1/2}(t_1,w_0),\, H_1^{1/2}(t_2,w_0))
	\end{equation}
	satisfies
	\begin{equation}\label{6.6}
	H_0^{1/2} <
	16 m(t_0,w_0)/\w{{\delta}(t_j,w_0)}^2
	\qquad \text{for $j=0$, $1$, $2$} 
	\end{equation} 
	If $m(t_0,w_0)\le {\varrho}\w{{\delta}(t_0,w_0)}$ for ${\varrho} \ll 1$,
	then we may
	choose $g^\sh$ ~orthonormal coordinates so
	that $w_0 = (x_1,0)$, $|x_1| < 2\w{{\delta}(t_0,w_0)} <
	32 {\varrho}H_0^{-1/2}$, and
	\begin{align}
	\sgn(w_1)f(t_0,w) \ge 0\quad\text{when $|w_1| \ge (m(t_0,w_0) +
		H_0^{1/2}|w'|^2)/c_0$} \label{signref}\\
	|{\delta}(t_1,w) -{\delta}(t_2,w) | \le 
	(m(t_0,w_0) + H_0^{1/2}|w-w_0|^2)/c_0 \label{d0est}
	\end{align}
	when $|w| \le c_0 H_0^{-1/2}$.
	The constant $c_0$ only depends on the seminorms of $f$.
\end{prop}

Observe that condition~\eqref{signref} is not empty when
$m(t_0,w_0)\le {\varrho}\w{{\delta}(t_0,w_0)}$ for ~${\varrho}$ sufficiently small since then $ H_0^{-1/2} \gtrsim \w{\delta}^2/m \gg m  $ by \eqref{6.6} .

\begin{proof}
	If we let
	$$F(s,t,w) =
	|{\delta}(s,w) -{\delta}(t,w)| + \max(H_1^{1/2}(s,w)
	\w{{\delta}(s,w)}^2,H_1^{1/2}(t,w)
	\w{{\delta}(t,w)}^2)/2$$ 
	then we find that $w \mapsto
	F(s,t,w)$ is uniformly Lipschitz continuous. Now, it
	suffices to show this when $|{\Delta}w| = |w - w_0| \ll 1$, and we know that
	$\w{{\delta}}$ and $H^{-1/2}$  are uniformly Lipschitz continuous  by
	Proposition~\ref{g1prop} which gives that $H_1^{-1/2} =  \max_x H^{-1/2}$ is uniformly Lipschitz continuous.
	The first term $|{\delta}(s,w) -{\delta}(t,w)|$ is obviously
	uniformly Lipschitz continuous. We have 
	for fixed $t$ that
	\begin{equation*}
	\left|{\Delta}(H_1^{1/2} \w{{\delta}}^2) \right| \le
	C(\w{{\delta}}^2|{\Delta}H_1^{1/2}| + H_1^{1/2}
	\w{{\delta}}|{\Delta}{\delta}|) 
	\end{equation*} 
	where $H_1^{1/2} \w{{\delta}} \le 1$, $|{\Delta}{\delta}| \le |{\Delta}w |  $ and
	$|{\Delta}H_1^{1/2}| \le CH_1|{\Delta}H_1^{-1/2}| \le C'H_1|{\Delta}w|$, 
	which gives the  uniform Lipschitz continuity of ~$F(s,t,w)$.
	By taking the infimum, we obtain ~\eqref{hhhest} and the uniform 
	Lipschitz continuity of ~$m$. 
	In fact, $h^{1/2}/3 \le
	H_1^{1/2}$ by~\eqref{H1hest} and since $t \mapsto \delta(t,w)$ is monotone, we find that 
	$t \mapsto \w{\delta(t,w)}$ is quasi-convex. Thus
	$h^{1/2}\w{{\delta}(t_0,w_0)}/6 \le F(s,t,w_0)$ when $s \le t_0 \le t$.
	
	By approximating the infimum, we may choose $t_1 \le t_0 \le t_2$ so
	that $F(t_1,t_2,w_0) < m(t_0,w_0) + h^{1/2}/6$. Since $h^{1/2}/6 \le
	m \le H_1^{1/2}\w{{\delta}}^2/2$ by ~\eqref{hhhest}, we find that
	\begin{align}\label{6.6a}
	&|{\delta}(t_1,w_0) -{\delta}(t_2,w_0)| <
	m(t_0,w_0) \le \w{{\delta}(t_0,w_0)}/2
	\qquad\text{and}\\ 
	&H_1^{1/2}(t_j,w_0)\w{{\delta}(t_j,w_0)}^2/2 <
	2m(t_0,w_0) 
	\qquad\text{for $j = 1$ and $2$}
	\label{6.7a}
	\end{align}
	Since $t \mapsto {\delta}(t,w_0)$ is monotone, we obtain
	~\eqref{6.11a} from \eqref{6.6a}, and ~\eqref{6.6} from~\eqref{6.7a} and ~\eqref{6.11a}.
	
	Next assume that $m(t_0,w_0) \le 
	{\varrho}\w{{\delta}(t_0,w_0)}$ for some $0 < {\varrho} \le 1$.
	Then we find from ~\eqref{6.6} that
	\begin{equation}\label{fsymbolny} 
	1 + |{\delta}(t_j,w_0)|  < 16{\varrho}H_0^{-1/2}\qquad\text{for $ j
		=0, \, 1, \, 2$}
	\end{equation}
	We may choose $g^\sh$ ~~orthonormal coordinates so that $w_0 = 0$. Since $ \delta $, $ H_1 $, $ H_0 $ and $ m $  are constant in $ x $,  the results will hold for any $ x $. 	
	I	
	If we choose $ x_j $ so that $ H_1(t_j ,0) =  H(t_j , x_j,0) $ 
	 then $\w{{\delta}(t_j,0)} < 16{\varrho}H^{-1/2}(t_j,x_j, 0)$  for $  j=1, 2 $ by~\eqref{fsymbolny} so we find from Proposition~\ref{g1prop} that
	\begin{equation}\label{dfest}
	h^{1/2} \le  | f'(t_j,x_j,0)| \cong | f'(t_j,x_j,w)|\quad \text{for }  |w| \le c H_0^{-1/2} \le c
		H^{-1/2}(t_j,x_j,0)
	\end{equation}
	when ${\varrho} \ll 1$ and $ j = 1,\, 2 $. Since $H_0^{-1/2} \le 3 h^{-1/2}$ we find from \eqref{fsymbolny} that
	$f(t_j,x_j,\wt w_j) = 0$ for some
	$|\wt w_j| < 16 {\varrho}H_0^{-1/2}$ by~\eqref{fsymbolny} when ${\varrho} < 1/48$ and $j = 1,\, 2$. Thus, when $16{\varrho} \le c$ we obtain from \eqref{dfest} for $ j = 1,\, 2 $ that 
	$$ 
	|f(t_j,x_j,w)| \le C |f'(t_j,x_j,0)| H_0^{-1/2}\qquad\text{when
		$|w| < c H_0^{-1/2}$}
	$$ 
	and then ~\eqref{dfsymbolest} gives $f(t_j,x_j,w) \in
	S(|f'(t_j,x_j,0)|H_0^{-1/2}, H_0g^\sh)$ since
	$H^{1/2}(t_j,x_j,0) \le H_0^{1/2}$, $ j =1,\,2$.  
	By Proposition~\ref{ffactprop} we have that $ f_j = \alpha \delta $ where $ \delta \in S(H_0^{-1/2}, H_0g^\sh) $ and $  \alpha \in S( |f'(t_j,x_j,0)|, H_0g^\sh)$ in a $ H_0g^\sh$ neighborhood of $ (t_j,0)$  such that $| \alpha | = |f'(t_j,x_j,0)| $ and $| \delta' | = 1$ at  $ (t_j,0)$. Now $ \partial_t \delta \ge 0$
	so if $ | \partial_t \delta |^{g^\sh} \ge \varepsilon > 0$ at $ (t_j,0) $ for $ j = 1 $ or $ 2 $ then $ \partial_t \delta \ge c \varepsilon h^{-1/2} $  in a small $ H_1g^\sh $ neighborhood. The interval $ \{(t, 0): \, |t - t_j | \le c_0 h^{1/2}H_1^{-1/2} \}$ is contained in this neighborhood for small enough $ c_0 > 0 $. Then we find 
	$$ |\delta(t_0,0) - \delta(t_j,0)| \ge c c_0 \varepsilon H_1^{-1/2}(t_j,0) \ge c c_0 \varepsilon  H_0^{-1/2} \ge c c_0 \varepsilon \w{\delta(t_j, 0)}/16\varrho
	$$
	by~\eqref{fsymbolny}, which by~\eqref{6.11a} contradicts~\eqref{6.6a} for small enough $ \varrho $. Thus, we may assume that $| \partial_w \delta | \ge 1/2 $ at $ (t_j,0) $ for $ j = 1,\, 2 $.
	
	Choose coordinates $z= H_0^{1/2} w$, we shall use  Proposition~\ref{betalemma} 
	with
	$$
	f_j(z) = H_0^{1/2} f(t_j, x_j, H_0^{-1/2}
	z)/| f'(t_j,0)|\in C^\infty\qquad\text{for
		$ j =1, \, 2$} 
	$$ 
	Let ${\delta}_j(z) = H_0^{1/2}{\delta}(t_j, H_0^{-1/2} z) \in C^\infty$ be the signed
	distance function to $f_j^{-1}(0)$ in $ z $ coordinates, then~\eqref{fsymbolny} gives that $ | \delta_j(0)| \le 16 \varrho  $ for $  = 0,\, 1,\,2 $. 
	Now  $| \partial_z \delta_j(t_j,0) | \ge 1/2 $, which for small enough $ \varrho $ gives
	$|  \partial_z f_j(0)|  \ge c_0$ for some $ c_0 > 0 $. 
	In fact, we have that $ f_j = a_j \delta_j $  where $ a_j \in  C^\infty$ is uniformly bounded and $ a_j(0) = 1 $. Then we obtain that $\partial_z f_j(0) =  a_j(0) \partial_z \delta_j(0)  +  \delta_j(0)  \partial_z a_j(0) \ge 1/2 - c \varrho $. 
	Because of condition $ \subr (\ol \Psi) $ given by~\eqref{pcond0} we find that $ f_1(z) > 0  \implies f_2(z) \ge 0$. 
	 Since  $|{\delta}_j(0)| < 16 {\varrho}$ we find that
	\begin{equation}
	\label{ddiff}
	|{\delta}_1(0) - {\delta}_2(0)| = {\varepsilon} <
	H_0^{1/2}m(t_0,0)\le
	H_0^{1/2}\w{{\delta}(t_0,0)}/2 < 8{\varrho}
	\end{equation}
	by~\eqref{6.6a}. Thus, for sufficiently small
	~${\varrho}$ we may use Proposition~\ref{betalemma} with this choice of $ f_j $ to obtain $g^\sh$
	orthogonal coordinates $(z_1,z')$ so that $w_0 = z_0= (y_1,0)$,
	$|y_1| = |{\delta}_1(0)|$ 
	and
	\begin{equation*}
	\left\{ 
	\begin{aligned}
	& \sgn(z_1)f_j(z) > 0 \quad\text{when $|z_1| \ge ({\varepsilon} +
		|z'|^2)/c_0 $}\\
	& |{\delta}_1(z) - {\delta}_2(z)| \le ({\varepsilon} + |z -
	z_0|^2)/c_0
	\end{aligned}\right.
	\end{equation*}
	when $|z| \le c_0$.
	Let $x_1 = H_0^{-1/2}y_1$ then $|x_1| < 2\w{{\delta}(t_0,0)} < 32 {\varrho}H_0^{-1/2}$
	by ~\eqref{6.11a} and~\eqref{fsymbolny}.
	We obtain~\eqref{signref}--\eqref{d0est}  by the condition~$\subr(\ol
	{\Psi})$, since $H_0^{-1/2}{\varepsilon} < m(t_0,0)$ by~\eqref{ddiff}. 
\end{proof}

\begin{prop}\label{h0slow}
	There exists $C>0$ such that
	\begin{equation}\label{wtemp}
	m(t_0,w) \le C m(t_0,w_0)(1 + |w-w_0|/\w{{\delta}(t_0,w_0)})^{3}
	\end{equation}
	thus $m$ is a weight for $g^\sh$.
\end{prop}

\begin{proof}
	Since $m \le \w{{\delta}}/2$ we only have
	to consider the case when
	\begin{equation} \label{h0rhoest}
	m(t_0,w_0)  \le {\varrho}\w{{\delta}(t_0,w_0)}
	\end{equation} 
	for some ${\varrho} > 0$.
	In fact, otherwise we have by \eqref{hhhest} that
	\begin{equation*}
	m(t_0,w) \le \w{{\delta}(t_0,w)}/2 <
	m(t_0,w_0)(1 + |w-w_0|/\w{{\delta}(t_0,w_0)})/2{\varrho}
	\end{equation*}
	since the Lipschitz continuity of $w \mapsto {\delta}(t_0,w)$ gives 
	\begin{equation}\label{6.11}
	\w{{\delta}(t,w)} \le \w{{\delta}(t,w_0)}(1
	+ |w-w_0|/\w{{\delta}(t,w_0)}) \qquad \forall\,t
	\end{equation}
	If ~\eqref{h0rhoest} holds for ${\varrho} \ll 1$, 
	then Proposition~\ref{weightprop} gives $t_1 \le t_0 \le t_2$ such that ~\eqref{6.11a}, \eqref{6.6} and~\eqref{d0est} 
	hold when $|w| \le c_0 H_0^{-1/2}$ with $H_0^{1/2} = \max(H_1^{1/2}(t_1,w_0),\,
	H_1^{1/2}(t_2,w_0)).$ 
	
	Now, for fixed $w_0$ it suffices to prove
	~\eqref{wtemp} when 
	\begin{equation}\label{suffref} 
	|w-w_0|\le {\varrho}H_0^{-1/2}
	\end{equation}
	for some ${\varrho} > 0$.
	In fact, when $|w-w_0| >
	{\varrho}H_0^{-1/2}$ we obtain from  \eqref{6.6}  that 
	\begin{multline*} 
	|w-w_0|^2/\w{{\delta}(t_0,w_0)}^2 >
	{\varrho}^2 H_0^{-1}/\w{{\delta}(t_0,w_0)}^2  > 
	{\varrho}^2\w{{\delta}(t_0,w_0)}^2/256 m^2(t_0,w_0) \\ \ge
	{\varrho}^2\w{{\delta}(t_0,w_0)}m(t_0,w)/64\w{{\delta}(t_0,w)}m(t_0,w_0)
	\end{multline*}
	since $\w{{\delta}} \ge 2m$. By ~\eqref{6.11} we obtain that ~\eqref{wtemp}
	is satisfied with $C= 64/{\varrho}^2$. Thus in the following we shall only consider
	$w$ such that ~\eqref{suffref} is satisfied
	for  $ {\varrho} \ll 1$. 
	We find by ~~\eqref{6.6}  and~\eqref{d0est} that
	\begin{multline}\label{dpmvar1}
	|{\delta}(t_1,w) - {\delta}(t_2,w)|
	\le ( m(t_0,w_0) + H_0^{1/2}|w-w_0|^2)/c_0 \\ 
	< 16 m (t_0,w_0) 
	(1 + |w-w_0|^2/\w{{\delta}(t_0,w_0)}^2)/c_0
	\end{multline}
	when $|w-w_0| \le c_0 H_0^{-1/2}$. Now $G$ is 
	slowly varying, thus we find for small enough
	${\varrho{}} >0$ that
	\begin{equation*}
	H_1^{1/2}(t_j,w) \le CH_1^{1/2}(t_j,w_0)
	\qquad\text{when  $|w-w_0| \le
		{\varrho{}}H_0^{-1/2}  \le 
		{\varrho{}}H_1^{-1/2}(t_j,w_0)$}
	\end{equation*}
	for  $j = 1$, $2$.
	By ~\eqref{6.11} and ~\eqref{6.11a} we obtain that 
	\begin{equation}\label{h1est}
	H_1^{1/2}(t_j,w) \w{{\delta}(t_j,w) }^2  \le
	4C H_1^{1/2}(t_j,w_0) \w{{\delta}(t_j,w_0) }^2 (1 +
	|w-w_0|/\w{{\delta}(t_0,w_0)})^2 
	\end{equation}
	when $j = 1, \ 2$, and $|w-w_0| \le c_0H_0^{-1/2}$. 
	Now $H_1^{1/2}(t_j,w_0) \w{{\delta}(t_j,w_0)}^2 < 16
	m(t_0,w_0)$ by ~\eqref{6.6} for $j=1$, 2. Thus, by using
	\eqref{dpmvar1}, \eqref{h1est} and taking the infimum we obtain
	$$ 
	m(t_0,w)  \le C_0 m(t_0,w_0)(1 + 
	|w-w_0|/\w{{\delta}(t_0,w_0)})^2$$ 
	when $|w-w_0|   \le {\varrho}H_0^{-1/2}$ 
	for ${\varrho} \ll 1$. 
\end{proof}

The following result will be important for
the proof of Proposition~\ref{mainprop} in Section~\ref{lower}.

\begin{prop}\label{mestprop}
	Let $M$ be given by Definition~\ref{Mdef} and $m$ by
	Definition~\ref{h0def}. Then there exists $C_0>0$ such that
	\begin{equation}\label{Mest0}
	MH^{3/2} \le C_0m/\w{{\delta}}^2
	\end{equation} 
\end{prop}

\begin{proof}
	In the proof, we shall include  the $t$ variable in the $ w $ variables. Observe that since $h^{1/2}
	\w{\delta}^2/6 \le m$ we find that \eqref{Mest0} is equivalent to
	\begin{equation}\label{h0eq0}
	|{f''}| H^{1/2} \le  C m/\w{{\delta}}^2
	\end{equation}
	by Proposition~\ref{mproplem}. First we note that 
	if $m \ge c\w{{\delta}} >0$, then $MH^{3/2}\w{{\delta}}^2 \le C\w{{\delta}}
	\le Cm/c$ since $\w{{\delta}} \le H^{-1/2}$ and $M \le CH^{-1}$ by
	Proposition~\ref{mproplem}. 
	
	Thus, we only have to consider the case $m \le {\varrho}\w{{\delta}}$ at $w_0$
	for some ${\varrho}>0$ to be chosen later. 
	Then we may use Proposition~\ref{weightprop} for ${\varrho}\ll 1$ to choose
	$g^\sh$ ~~orthonormal coordinates so that $|w_0| < 2\w{{\delta}(w_0)}
	< 32 {\varrho} H_0^{-1/2}$ with $ H_0 $ given by~\eqref{defH0}, and $f$ satisfies~\eqref{signref}
	with
	\begin{equation} \label{h0eq1}
	h^{1/2}/3 \le H_0^{1/2} < 
	16 m(w_0)/\w{{\delta}(w_0)}^2 \le
	8 H^{1/2}(w_0)
	\end{equation}
	by~\eqref{H1hest}, \eqref{hhhest} and~\eqref{6.6}.  Thus 
	it suffices to prove the estimate 
	\begin{equation}\label{fundest} 
	|{f''}| H^{1/2} \le CH_0^{1/2}
	\qquad\text{} 
	\end{equation}
	at $w_0$. Now it actually suffices to prove ~\eqref{fundest} at
	$w=0$. In fact, \eqref{tempest} gives 
	\begin{equation}\label{hhest}
	H(w_0) \le C_0H(0)(1 + H(w_0)|w_0|^2) \le 5C_0H(0)
	\end{equation} 
	since $|w_0| < 2 \w{{\delta}(w_0)}
	\le  2H^{-1/2}(w_0)$. Thus
	Taylor's formula gives
	\begin{equation} \label{h0eq2}
	|{f''(w_0)}| H^{1/2}(w_0) \le \left(|{f''(0)}| + C_3h^{1/2}|w_0|
	\right) H^{1/2}(w_0)
	\le C_1(|{f''(0)}|H^{1/2}(w_0) + h^{1/2})
	\end{equation}
	since $|{f^{(3)}}| \le C_3h^{1/2}$, which gives  \eqref{fundest}  at $w=0$ by  \eqref{h0eq1} and \eqref{hhest}.
	
	By Definition~\ref{g1def} we
	find that
	\begin{multline*} 
	H^{-1/2} \ge 1 +
	|f'|/(|{f''}| + h^{1/4}|f'|^{1/2} + h^{1/2})\\ \ge (|f''| + |{f'}|
	+ h^{1/2} )/(|{f''}| + h^{1/4}|f'|^{1/2} + h^{1/2})
	\end{multline*}
	thus~\eqref{fundest} follows if we prove
	\begin{equation}\label{fundest0}
	|{f''}|(|{f''}| + h^{1/4}|f'|^{1/2} + h^{1/2}) \le C\left(|f'| + |{f''}|
	+ h^{1/2} \right) H_0^{1/2} \qquad\text{at $0$}
	\end{equation}
	Since $h^{1/2}/3 \le H_0^{1/2}$ we obtain ~\eqref{fundest0} 
	by the Cauchy-Schwarz inequality if
	we prove that
	\begin{equation}\label{essest}
	|{f''(0)}| \le C (H_0^{1/4}|f'(0)|^{1/2} + h^{1/2}) 
	\end{equation}
	Let $F(z) = H_0 f(H_0^{-1/2}z)$, then
	~\eqref{signref} gives
	\begin{equation*}
	\sgn(z_1)F(z) \ge 0 \quad\text{when $|z_1| \ge {\varepsilon} +
		|z'|^2/r $ and $|z| \le r$}
	\end{equation*}
	where $r = c_0$ and 
	$${\varepsilon} =
	H_0^{1/2}m(w_0)/c_0 \le 16m^2(w_0)/c_0\w{{\delta}(w_0)}^2 \le
	16{\varrho}^2/c_0 \le c_0/5 
	$$ 
	by~\eqref{h0eq1} when ${\varrho} \le c_0/4\sqrt{5}$ which we shall
	assume. Proposition~\ref{newestprop} then gives that
	\begin{equation*}
	|{F''(0)}| \le C_1\left(|F'(0)|/{\varrho}_0 +
	H_0^{-1/2}h^{1/2}{\varrho}_0\right) \qquad \varepsilon \le
	{\varrho}_0 \le c_0/\sqrt{10}
	\end{equation*}
	since $\mn{F^{(3)}}_\infty \le C_3H_0^{-1/2} h^{1/2}$. 
	Observe that $|F'(0)| \le C_2$ since
	$H_0^{1/2} \le 8H^{1/2}(w_0)\le
	CH^{1/2}(0)$  by \eqref{h0eq1} and \eqref{hhest}, and $|f'(0)| \le CH^{-1/2}(0)$ by \eqref{dfest0}. 
	Choose
	$${\varrho}_0 = \varepsilon + {\lambda} |F'(0)|^{1/2} \le c_0/\sqrt{10}$$ 
	with ${\lambda} = c_0(\sqrt{10}-2)/10\sqrt{C_2}$, then we obtain that 
	\begin{equation}
	 |{F''(0)}| \le
	C_4(|F'(0)|^{1/2} + h^{1/2}m(w_0)) \label{F''est}	
	\end{equation}
	since $H_0^{-1/2} \le 3h^{-1/2}$ and ${\varepsilon} =
	H_0^{1/2}m(w_0)/c_0$. 
	If $h^{1/2}m(w_0) \le
	|F'(0)|^{1/2}$ then we obtain
	~\eqref{essest} 
	since $F' = H_0^{1/2}f'$ 
	and $F'' = f''$. If $|F'(0)|^{1/2} \le h^{1/2}m(w_0)$ then we
	find from~\eqref{F''est} that
	$$
	|{f''(0)}| \le 2C_4  h^{1/2}m(w_0) \le 4C_2 m(w_0)/\w{{\delta}(w_0)}
	$$
	Then \eqref{h0eq0} follows from \eqref{wtemp}, \eqref{6.11} and~\eqref{hhest}	since $ H^{1/2}(w_0)
	\le \w{{\delta}(w_0)}^{-1}$, which completes the
	proof of the proposition.
\end{proof}

Next, we shall prove a convexity property of $t \mapsto
m(t,w)$, which will be essential for the proof. 

\begin{prop}\label{qmaxpropo}
	Let $m$ be given by Definition~\ref{h0def}. Then
	\begin{equation}\label{qmaxwd}
	\sup_{t_1 \le t \le t_2}m(t,w) \le {\delta}(t_2,w) - {\delta}(t_1,w) + m(t_1,w) + m(t_2,w)\qquad \forall\,w 
	\end{equation} 
\end{prop}

\begin{proof}
	By definition we find that
	\begin{equation}\label{convest}
	\inf_{\pm(t-t_0) \ge 0} \left(|{\delta}(t,w) 
	-{\delta}(t_0,w)| + H_1^{1/2}(t,w)\w{{\delta}(t,w)}^2/2\right) \le
	m(t_0,w)
	\end{equation}
	Let $t \in [t_1,t_2]$, then by  taking the independent infima, we obtain that
	\begin{multline*} 
	m (t,w) \le \inf_{r \le t_1 < t_2\le s} {\delta}(s,w) - {\delta}(r,w)
 + H_1^{1/2}(s,w)\w{{\delta}(s,w)}^2/2  +
	H_1^{1/2}(r,w)\w{{\delta}(r,w)}^2/2  \\ \le
	{\delta}(t_2,w) - {\delta}(t_1,w)  + \inf_{t \ge t_2} \left(|{\delta}(t,w)
	-{\delta}(t_2,w)| + H_1^{1/2}(t,w)\w{{\delta}(t,w)}^2/2\right) \\
	+ \inf_{t \le t_1 } \left(|{\delta}(t,w)
	-{\delta}(t_1,w)| + H_1^{1/2}(t,w)\w{{\delta}(t,w)}^2/2\right)
	\end{multline*}
	By using ~\eqref{convest}  for $t_0 = t_1$,
	$t_2$, we obtain~\eqref{qmaxwd} after taking the supremum.
\end{proof}

Next, we shall construct the pseudo-sign $B = {\delta} + {\varrho}_0$,
which we shall use in Section~\ref{lower} to prove
Proposition~\ref{mainprop} with the multiplier $b^w = B^{Wick}$.

\begin{prop}\label{apsdef}
	Assume that ${\delta}$ is given by
	Definition~\ref{d0deforig} and  $m$ is given by
	Definition~\ref{h0def}. Then for $T>0$ there exists real 
	valued ${\varrho}_T(t,w)
	\in L^\infty$ with the property that $w \mapsto
	{\varrho}_T(t,w)$ is uniformly Lipschitz continuous, and
	\begin{align}\label{r0prop0}
	&|{\varrho}_T| \le m\\
	&T\partial_t({\delta} + {\varrho}_T) \ge m/2 \qquad \text{in
		$\Cal D'(\br)$}
	\label{r0prop1}
	\end{align}
	when $|t| <T$. 
\end{prop}

\begin{proof}
	(We owe this argument to Lars H\"ormander ~\cite{ho:NT}.)
	Let 
	\begin{equation}\label{r0def}
	{\varrho}_{T}(t,w) = \sup_{-T \le s
		\le t}\left({\delta}(s,w) - {\delta}(t,w) + \frac{1}{2T} \int_s^t
	m(r,w)\,dr - m(s,w)\right)
	\end{equation}
	for $|t| \le T$, then
	\begin{multline*}
	{\delta}(t,w) + {\varrho}_T(t,w) = \sup_{-T \le s
		\le t}\left({\delta}(s,w) - \frac{1}{2T}\int_0^s m(r,w)\,dr -
	m(s,w)\right) \\+ 
	\frac{1}{2T} \int_0^t m(r,w)\,dr
	\end{multline*}
	which immediately gives ~\eqref{r0prop1} since the
	supremum is nondecreasing. Since  $w \mapsto {\delta}(t,w)$ and $w \mapsto
	m(t,w)$ are uniformly   
	Lipschitz continuous by Proposition~\ref{weightprop}, we
	find that $w \mapsto {\varrho}_T(t,w)$
	is uniformly Lipschitz continuous  by taking the supremum.
	Since $ \delta(s,w ) \le \delta(t,w) $ when $ s \le t \le T $, we find from Proposition~\ref{qmaxpropo} that
	\begin{equation*}
	{\delta}(s,w) - {\delta}(t,w) + \frac{1}{2T}\int_s^t m(r,w)\,dr -
	m(s,w) \le m(t,w) \qquad -T \le s \le t \le T
	\end{equation*}
	By taking the supremum, we obtain that $ -m(t,w) \le
	{\varrho}_T(t,w) \le m(t,w)$ when $|t| \le T$, which proves the result.
\end{proof}

We shall also include a term in the multiplier to control  the error terms involving $ D_x u $.

\begin{lem}\label{multlem}
Assume that $ A $ satisfies the conditions in Proposition~\ref{prepprop} near $ w_0 \in \st $. Then there exists a matrix $ L $ and constant $ c _1 > 0$ such that $ \{A, \w{L(x-x_0), \xi}\}  \ge |\xi|^2 - c_1 $ microlocally near $ w_0 \in \st$, where $ x_0 $ is the value of $ x  $ at $ w_0 $. The constants only depend on the seminorms of $ A $.
\end{lem}

\begin{proof}
    Let $ w = (x,\xi, z) $, then we find from the conditions that 
	$$ 
	A(x,\xi,z) = \w{C_2(x,\xi,z)\xi,\xi} +\w{C_1(x,z), \xi}  +  C_0(x,z)
	$$ 
	where $ C_j \in S^0 $ is real valued, $ \forall\, j $, and $ C_2 $ is a  symmetric and  nondegenerate  matrix microlocally near $ w_0 $. By a translation 
	we may assume that $ x = 0 $ at $ w_0  $. If we take $ L =  C_2^{-1}(0,0,w_0) $ then we find that 
		\begin{equation}\label{Lest}
			 \{A, \w{Lx, \xi}\} =   \w{L \xi,\partial_\xi A}  -  \w{Lx, \partial_x A }   \ge c_2|\xi|^2 - c_0
		\end{equation}
	where $ c_2 = 2 $ at $w_0  $, since  $\partial_\xi A = 2C_2\xi +  \w{\partial_\xi C_2\xi,\xi}  + C_1$. 
	By continuity, we get the estimate in a neighborhood of $ w_0 $ where $ |\partial_\xi C_2 \xi | \ll  1$.
\end{proof}

Because of the cut-off in the estimate~\eqref{propest} we will only need the lower bound in a neighborhood of $ w_0$.

\begin{defn}\label{multdef}
	 Let the multiplier symbol  $B_T = {\delta}_0 + {\varrho}_T + \lambda_T$, where ${\delta}_0 = {\delta}$ is given by Definition~~\ref{d0deforig}, ${\varrho}_T$ is given by Definition~~\ref{apsdef} for $ T > 0 $ so it is
	real valued and Lipschitz continuous,
	satisfying $|{\varrho}_T| \le m$ when $ |t| \le T $, with ~$m \le \w{{\delta}_0}/2$ given by
	Definition~\ref{h0def}, and $ \lambda_T = \epsilon h^{1/2}\w{L(x-x_0), \xi}/T \in S(h^{-1/2}, g)$ uniformly when $ |x -x_0| \le T $ and $ | \xi | \ls h^{-1} $, where $ 0 < \epsilon \le 1 $ and $ L $ is given by Lemma~\ref{multlem}, so $ \lambda_T $ is  Lipschitz continuous.
\end{defn}

\section{The Wick quantization}\label{norm}

In order to define the multiplier we shall use the Wick quantization.
We shall start by recapitulating some results from Section~6 in \cite{de:X} about the Wick operators.
As before, we shall assume that $g^\sh =
(g^\sh)^{\sigma}$ and the coordinates chosen so that $g^\sh(w) =
|w|^2$.  For $a \in L^\infty(T^*\br^n)$ we define the Wick
quantization:
\begin{equation*}
a^{Wick}(x,D_x)u(x) = \int_{T^*\br^n}a(y,{\eta})
{\Sigma}^w_{y,{\eta}}(x,D_x)u(x)\,dyd{\eta}\qquad u \in  C_0^\infty
\end{equation*}
using the orthonormal projections ${\Sigma}^w_{y,{\eta}}(x,D_x)$ with Weyl symbol
$${\Sigma}_{y,{\eta}}(x,{\xi}) =
{\pi}^{-n}\exp(-g^\sharp(x-y,{\xi}-{\eta}))
$$
(see~\cite[Appendix~B]{de:suff} or~\cite[Section~4]{ln:coh}).
We find that 
$a^{Wick}$: $ \Cal S \mapsto \Cal S' $ so that 
\begin{equation} \label{poswick}
a \ge 0 
\implies
\sw{a^{Wick}(x,D_x)u,u} \ge
0 \qquad u \in C_0^\infty
\end{equation}
$(a^{Wick})^* = (\overline a)^{Wick}$ and
$ 
\mn{a^{Wick}(x,D_x)}_{\Cal L(L^2(\br^n))} \le
\mn{a}_{L^\infty(T^*\br^n)}, 
$ 
which is the main advantage with the Wick quantization
(see \cite[Proposition~4.2]{ln:coh}).  

We obtain from the definition that
$a^{Wick} = a_0^{w}$ where
\begin{equation}\label{gausreg}
a_0(w) = {\pi}^{-n}\int_{T^*\br^n} a(z)\exp(-|w-z|^2)\,dz 
\end{equation}
is the Gaussian regularization, thus Wick operators with real symbols
have real Weyl symbols.  This convolution also maps polynomials to polynomials.

\begin{rem}\label{wickrem}
	Observe that $ a^{Wick} = a^w$ if $ a(x,\xi) $ is affine in $ x$ for fixed $ \xi $ and affine in $ \xi$ for fixed $ x $, for example if $ a(x,\xi) = \w{Lx, \xi} $ with a constant matrix  $ L $ .
\end{rem}

In fact, $ a^{Wick} = a^w $ if and only if 
\[ 
{\pi}^{-n}\int_{T^*\br^n} (a(z)  - a(w)) \exp(-|w-z|^2)\,dz = 0
 \]
This vanishes if $ a $ is affine in $ x $ for fixed $ \xi $ and affine in $ \xi $ for fixed $ x $ since
$$  a(x,\xi)- a(y,\eta)= a(x,\xi)- a(y,\xi) + a(y,\xi)- a(y,\eta) $$ 
which are two odd integrands giving vanishing integrals.

In the following, we shall assume that
$G =Hg^\sh \le g^\sh$ is a slowly varying metric satisfying 
\begin{equation}\label{gentemp}
H(w) \le C_0 H(w_0) (1 + |w-w_0|)^{N_0} 
\end{equation}
and $m$ is a weight for $G$ satisfying~~\eqref{gentemp} with ~~$H$
replaced by ~~$m$, by Propositions ~\ref{g1prop}  and ~\ref{h0slow}. This means that $G$ and $m$ are strongly ${\sigma}$
~temperate in the sense of \cite[Definition~7.1]{bc:sob}. Recall the
symbol class $S^+(1,g^\sh)$ given by Definition~\ref{s+def}.

\begin{prop}\label{propwick} 
	Assume that $a \in L^\infty $ such that $|a| \le Cm$, where $ m $ is a weight for $ g^\sh $, then
	$ a^{Wick} = a_0^w$ where $a_0\in S(m,g^\sh)$ is given by
	~\eqref{gausreg}. If $a \ge m$ we obtain that $a_0 \ge c_0m$ for a
	fixed constant $c_0>0$. If $a \in S(M,G)$, where  $ M $ is a weight for $ G $, then $a_0 \cong a$ modulo 
	$S(mH,G)$. If $|a| \le Cm$ and $a = 0$ in a fixed
	$G$~ball with center ~$w$, then $a \in S(mH^N,G)$ at ~$w$ for any $N$.
	If $ a $ is polynomial in the variable $ \xi $ then $ a_0 $ is polynomial  in $ \xi $ with the same degree as $ a $, and if  $a$ is Lipschitz continuous then we have $a_0\in S^+(1, g^\sh)$.
\end{prop} 

By localization we find, for example, that if $|a| \le Cm$ and $a \in
S(m,G)$ in a $G$~neighborhood of ~$w_0$, then $a_0 \cong a$ modulo
$S(mH,G)$ in a smaller $G$~neighborhood of ~$w_0$. Observe that the
results are uniform in the metrics and weights.  The results are well
known, but for completeness we give a proof.

\begin{proof} 
	Since $a$ is measurable satisfying $|a| \le C m$, where $m(z) \le C_0
	m(w)(1 + |z-w|)^{N_0}$ by ~\eqref{gentemp}, we find that
	$a^{Wick} = a_0^{w}$ where $a_0 = \Cal O(m)$ is given by
	~\eqref{gausreg}. By differentiating on the exponential factor, we find
	$a_0 \in S(m,g^\sh)$, and similarly we find that $a_0 \ge m/C$ if $a
	\ge m$ since $ m(z) \gtrsim m(w)/(1 + |z - w|)^{N_0} $.
	
	If $a = 0$ in a $G$
	ball of radius ${\varepsilon}>0$ and center at $w$, then we can write
	\begin{equation*}
	{\pi}^{n}a_0(w) = 
	\int_{|z-w| \ge {\varepsilon}H^{-1/2}(w)} a(z)\exp(-|w-z|^2)\,dz =\Cal
	O(m(w) H^N(w)) 
	\end{equation*}
	for any $N$ even after repeated differentiation. 
	
	If $a \in S(m,G)$ then
	Taylor's formula gives
	\begin{equation*}
	a_0(w) = a(w) + {\pi}^{-n} \int_0^1\int_{T^*\br^n} (1-{\theta})^2\w{a''(w + {\theta}z)z,z} e^{-|z|^2}\,dzd{\theta}/2
	\end{equation*}
	where $a'' \in S(mH,G)$ since $G = Hg^\sh$.  Since $m(w + {\theta}z)
	\le C_0 m(w)(1 + |z|)^{N_0}$ and $H(w + {\theta}z) \le C_0H(w)(1 +
	|z|)^{N_0} $ when $|{\theta}| \le 1$, we find that $a_0(w) \cong a(w)$
	modulo $S(mH,G)$.  
	
	If $ a $ is polynomial in the variable $ \xi $ of degree $ k $ then $ \partial^\alpha_\xi a \equiv 0 $,$ \forall | \alpha | > k $, which gives $ \partial^\alpha_\xi a_0 \equiv 0 $. Thus $ a_0 $ is of degree $ \le k $.
	The  Lipschitz continuity of $a$ means that 
	$\partial a \in L^\infty(T^*\br^n)$. 
	Since we have $\partial a_0(w) = {\pi}^{-n}\int_{T^*\br^n}
	\partial a(z)\exp(-|w-z|^2)\,dz $, we obtain the proposition.
\end{proof}

We shall need the following result about the composition of Wick
operators.  

\begin{prop}\label{wickcomp0} Assume that $a$ and $b \in L^\infty $.
	If $|a| \le m_1$ and $|\partial b| \le m_2$, where $m_j$ are weights for
	$g^\sh$ satisfying ~\eqref{gentemp} with $ H = m_j $, then
	\begin{equation}\label{wickcomp1}
	a^{Wick}b^{Wick} = (ab )^{Wick} + r^w
	\end{equation}
	with $r \in S(m_1m_2,g^\sh)$. 
	If $a$ and $b$ are real such that  $|a| \le m_1$ and $| \partial^2 b|  \le m_2$, then
	\begin{equation}\label{wickcomp2}
	\re a^{Wick}b^{Wick} = \left(ab  - \frac{1}{2}\partial a\cdot \partial  b\right)^{Wick} + R^w
	\end{equation}
	with $R \in S(m_1m_2,g^\sh)$. By taking the adjoints, we get these results with $ a $ and $ b $ switched.
\end{prop}

Observe that since $ a \in L^\infty $ and $\partial b$ is Lipschitz continuous in \eqref{wickcomp2}, we find that $\partial a \cdot \partial b$ is a
well-defined distribution. In fact,  we can define it as $ \partial a \cdot\partial b(\varphi)  = -\int  a\partial (\varphi\partial b )dw$.
Proposition~\ref{wickcomp0} essentially follows from Proposition~3.4
in ~\cite{ln:wick} 
and Lemma~A.1.5 in ~\cite{ln:cutloss} but we shall for completeness
give a proof.

\begin{proof}
	By Proposition~\ref{propwick}
	we have $a^{Wick}b^{Wick} = a_0^wb_0^w$ in~\eqref{wickcomp1} where $a_0 \in S(m_1,g^\sh)$
	and $b_0 \in S^+(m_2,g^\sh)$. By  Lemma~\ref{calcrem} we find
	$a^{Wick}b^{Wick} \cong (a_0b_0)^w$ modulo $\op S(m_1m_2,g^\sh) $, where
	\begin{equation} \label{wickcomp}
	a_0(w)b_0(w) = \pi^{-2n} \iint a(w+z_1)b(w+z_2)e^{- |z_1|^2 - |z_2|^2}\,dz_1dz_2.
	\end{equation}
	By using the Taylor formula we find that $b(w + z_2) = b(w + z_1) + r_1(w,z_1,z_2)$
	where $|r_1(w,z_1,z_2)| \le Cm_2(w)(1 + |z_1|+ |z_2|)^N$
	by~\eqref{gentemp}. Integration in ~$z_2$ then gives ~\eqref{wickcomp1}.
	
	For the proof of~\eqref{wickcomp2} we use that $\re a_0^wb_0^w \cong
	(a_0b_0)^w$ modulo  $\op S(m_1m_2,g^\sh) $ by Lemma~\ref{calcrem},
	since $a_0$ and $b_0$ are real and $\partial^2 b_0 \in S(m_2,g^\sh)$. We use 
	the Taylor formula again:
	$$b(w + z_2) = b(w + z_1) + 
	\partial b(w + z_1)\cdot(z_2-z_1) + r_2(w,z_1,z_2)
	$$
	where $|r_2(w,z_1,z_2)| \le Cm_2(w)(1 + |z_1|+ |z_2|)^N$.
	The term with $z_2$ is odd and gives a vanishing contribution in ~\eqref{wickcomp}.
	Since $\partial_{z_1}e^{-|z_1|^2-|z_2|^2} = - 2z_1 e^{-|z_1|^2}$ we
	obtain ~\eqref{wickcomp2} after an integration by parts, since 
	$|a	\partial^2 b| \le m_1m_2$.
\end{proof}

\begin{exe}\label{wickcompexe} If $a \in S(H^{-1/2}, g^\sh) \bigcap S^+(1, g^\sh)$
	and $b \in S(M, G)$, then $\re a^{Wick}b^{Wick} \cong (ab )^{Wick}$ modulo
	$\op S(MH^{1/2}, g^\sh)$.
\end{exe}

We shall compute the Weyl symbol for the Wick operator
$B_T^{Wick} = ({\delta}_0 + {\varrho}_T + \lambda_T)^{Wick}$  given by Definition~\ref{multdef}.  In the
following, we shall suppress the  $ T $ parameter.

\begin{prop}\label{wickweyl}
	Let $B = {\delta}_0 + {\varrho}_0+ \lambda$ be given by
	Definition~~\ref{multdef}, 
	then we have $ B^{Wick} = b^w $ 
	where $b= {\delta}_1 + {\varrho}_1 + \lambda$  is real, ${\delta}_1 \in
	S(H^{-1/2},g^\sh)\bigcap 
	S^+(1,g^\sh)$, and ${\varrho}_1 \in S(m, g^\sh)\bigcap S^+(1, g^\sh)$ uniformly when $| t |\le T $.
	Also, there exists ${\kappa}_2 >0$ so that ${\delta}_1 \cong {\delta}_0$
	modulo $S(H^{1/2},G)$ when $\w{{\delta}_0} \le {\kappa}_2H^{-1/2}$.  
	For any $\varepsilon > 0$ we find that $|{\delta}_0| \ge
	\varepsilon H^{-1/2}$ and $H^{1/2} \le \varepsilon/3$ imply
	that $ 	| \delta_0 + \varrho_0  | \ge \varepsilon H^{-1/2}/3$. 
\end{prop}

\begin{proof}
	Let ${\delta}_0^{Wick} = {\delta}_1^w$ and ${\varrho}_0^{Wick} =
	{\varrho}_1^w$.  Since $|{\delta}_0|\le H^{-1/2}_1$, $|{\varrho}_0|
	\le m$ and the symbols are real valued, we obtain from
	Proposition~\ref{propwick} that ${\delta}_1 \in S(H^{-1/2},g^\sh)$
	and ${\varrho}_1 \in S(m, g^\sh)$ are real valued.  Since
	${\delta}_0$ and ${\varrho}_0$ are
	uniformly Lipschitz continuous, we find that ${\delta}_1$ and ${\varrho}_1 \in
	S^+(1,g^\sh)$ by Proposition~\ref{propwick}. By Remark~\ref{wickrem} we have $ \lambda^{Wick} = \lambda^w $.

	If $\w{{\delta}_0} \le {\kappa}H^{-1/2}$ at ~$w_0$ for sufficiently
	small ${\kappa} >0$, then we find by the Lipschitz continuity of~
	${\delta}_0$ and the slow variation of ~$G$ that $\w{{\delta}_0} \le
	C_0{\kappa}H^{-1/2}$ in a fixed $G$ neighborhood
	${\omega}_{\kappa}$ of ~$w_0$ (depending on ${\kappa}$). For ${\kappa}
	\ll 1$ we find that 
	${\delta}_0 \in S(H^{-1/2}, G)$ in ${\omega}_{\kappa}$ by
	Proposition~\ref{mproplem}, which implies that ${\delta}_1 \cong
	{\delta}_0$ modulo $S(H^{1/2},G)$ near ~$w_0$ by
	Proposition~\ref{propwick} after localization.
	
	When $|{\delta}_0| \ge \varepsilon H^{-1/2} \ge \varepsilon  > 0$ at
	$w_0$, then we find that 
	$$|{\varrho}_0| \le m \le \w{{\delta}_0}/2 \le
	(1 + H^{1/2}/\varepsilon)|{\delta}_0|/2.$$
	We obtain that $|{\varrho}_0| \le 2|{\delta}_0|/3$ and
	$| \delta_0 + \varrho_0  | \ge |{\delta}_0|/3 \ge
	\varepsilon H^{-1/2}/3 $ 
	when $H^{1/2}  \le \varepsilon/3$, which completes the proof.
\end{proof}

Let $m$ be given by Definition~\ref{h0def}, then
$m$ is a weight for $g^\sh$ according to
Proposition~\ref{h0slow}. We are going to use the symbol classes
$S(m^k, g^\sh)$, $k \in \br$. 
The following proposition shows that  the operator $m^{Wick}$ 
dominates all operators in $\op S(m,g^\sh)$.

\begin{prop}\label{wprop} 
	If $c \in S(m,g^\sh)$ then there exists a positive constant  $C_0$ such that
	\begin{equation}\label{west3}
	|\w{c^wu,u}| \le  C_0\sw{m^{Wick}u,u}
	\qquad u \in C_0^\infty.
	\end{equation}
	Here $C_0$ only depends on the seminorms of $c \in S(m,g^\sh)$.
\end{prop}

\begin{proof} 
	We shall use an argument by H\"ormander~\cite{ho:NT}.
	Let $0 < {\varrho} \le 1$
	\begin{equation}
	M_{\varrho}(w_0) = \sup_wm(w)/(1 + {\varrho}|w-w_0|)^3
	\end{equation}
	then $m \le M_{\varrho} \le Cm/{\varrho}^3$ and
	\begin{equation}\label{nytemp}
	M_{\varrho}(w) \le CM_{\varrho}(w_0)(1 + {\varrho}|w-w_0|)^3\qquad\text{uniformly in $0 < {\varrho} \le 1$}
	\end{equation}
	by the triangle inequality. 
	Thus, $M_{\varrho}$ is a weight for $g_{\varrho} = {\varrho}^2g^\sh$,
	uniformly in ~${\varrho}$.
	Take $0 \le {\chi} \in
	C_0^\infty $ such that $\int_{T^*\br^n} {\chi}(w) \,dw > 0$
	and let 
	$$
	m_{\varrho}(w) = {\varrho}^{-2n}\int {\chi}({\varrho}(w
	-z))M_{\varrho}(z)\,dz
	$$
	Then by ~\eqref{nytemp} we find $1/C_0 \le m_{\varrho}/M_{\varrho}
	\le C_0$, and 
	$
	|\partial^{\alpha}m_{\varrho}| \le
	C_{\alpha}{\varrho}^{|{\alpha}|} m_{\varrho}
	$ 
	thus $m_{\varrho} \in S(m_{\varrho},g_{\varrho})$ uniformly in $0 < {\varrho} \le 1$.
	Let $m_{\varrho}^{Wick} = {\mu}_{\varrho}^w$ then
	Proposition~\ref{propwick} and~\eqref{nytemp} give $m_{\varrho}/c \le
	{\mu}_{\varrho} \in  S(m_{\varrho},g_{\varrho})$ 
	uniformly in $0 < {\varrho} \le 1$.
	Since $m \cong m_{\varrho}$ (depending on $ \varrho $) we may replace $m^{Wick} $ with $m_{\varrho}^{Wick} = {\mu}_{\varrho}^w$ in~\eqref{west3} for any fixed ${\varrho}
	> 0$. 
	
	Let $a_{\varrho} = {\mu}_{{\varrho}}^{-1/2} \in S(m^{-1/2}_{\varrho},
	g^\sh_{\varrho})$ with $0 < {\varrho} \le 1$ to be chosen later. 
	Since $g_{\varrho}$ is uniformly ${\sigma}$
	~temperate, $g_{\varrho}/g_{\varrho}^{\sigma} = {\varrho}^4$,
	$m_{\varrho}$ is uniformly ${\sigma}$, $g_{\varrho}$ ~temperate, and
	${\mu}_{\varrho}^{\pm1/2} \in S(m_{\varrho}^{\pm1/2},g_{\varrho})$
	uniformly, the calculus gives that $ a_{\varrho}^w(a_{\varrho}^{-1})^w = 1 +
	r^w_{\varrho} $ where $r_{\varrho}/{\varrho}^2 \in S(1,g^\sh)$
	uniformly for $0 < {\varrho} \le 1$. Similarly, we find that 
	$a_{\varrho}^w {\mu}_{\varrho}^wa_{\varrho}^w = 1 +
	s^w_{\varrho} 
	$ 
	where $s_{\varrho}/{\varrho}^2 \in S(1,g^\sh)$ uniformly.
	We obtain that the $L^2$ operator
	norms 
	$$\mn{r^w_{\varrho}}_{\Cal L(L^2)} + \mn{s^w_{\varrho}}_{\Cal
		L(L^2)} \le C{\varrho}^2 \le 1/2
	$$ 
	for sufficiently small  ${\varrho}$.
	By fixing such a value of ${\varrho}$ we find that
	$1/2 \le a_{\varrho}^w {\mu}_{\varrho}^wa_{\varrho}^w \le 2$ and
	\begin{equation}\label{west1} 
	\frac{1}{2}\mn u \le  \mn {a_{\varrho}^w(a^{-1}_{\varrho})^w u}
	\le 2 \mn u
	\end{equation}
	thus $u \mapsto a_{\varrho}^w(a^{-1}_{\varrho})^w u$ is an homeomorphism on  $L^2$.
	The estimate~\eqref{west3} then follows from 
	$$ 
	|\w{c^wa_{\varrho}^w(a^{-1}_{\varrho})^w u,
		a_{\varrho}^w(a^{-1}_{\varrho})^w u }| \le  C\mn{(a^{-1}_{\varrho})^w u}^2 \le
	2C\w{\mu_\varrho^w
		a_{\varrho}^w(a^{-1}_{\varrho})^w
		u,a_{\varrho}^w(a^{-1}_{\varrho})^w u}
	$$
	which holds since $a_{\varrho}^wc^wa_{\varrho}^w \in \op S(1,g^\sh)$ 
	is bounded in $L^2$. Observe that the bounds only depend on the
	seminorms of $c$ in $S(m,g^\sh)$, since ${\varrho}$ and
	$a_{\varrho}$ are fixed.
\end{proof}

We shall also need a weight to handle the calculus errors in the $ \xi $ variables. Let $ \mu = h^{1/2}\w{\xi}^2$ which is a weight for $ g_0(dx,d\xi) = | dx|^2 + | d\xi |^2/\w{\xi}^2$ so that $ \lambda \in S(\mu, g_0) $ uniformly in $ h $. In order to handle the compositions with $ \lambda $  and $ f_0 $ we shall need the following metric:
\begin{equation}\label{G0def}
	G_0 = H|dw|^2 + H(dt^2 + |dx|^2)/h + Hh d\tau^2 + | d\xi|^2/\w{\xi}^2
\end{equation}
which is strongly $ \sigma $ temperate in the sense that 
\begin{equation}
	G_{0,z} \ls G_{0,z_0}(1 + G_{0,z_0}(z - z_0))
\end{equation}
but since $ G_0/G_0^\sigma \not \ls 1$ it is not $ \sigma $ temperate. But since $ \lambda $  and $ f_0 $ are linear in $ \xi $ we shall only use the calculus on $ \st $, i.e., in  $(t,x,w) $, and $ G_0 =  G$  on $ T\st$. 

To handle the compositions with $ b_1 = \delta_1 + \varrho_1 \in S(H^{-1/2}, g^\sh) $ we shall use the metric:
\begin{equation}\label{g0def}
g^\sh_0 = |dw|^2 + (dt^2 + |dx|^2)/h + h d\tau^2 + | d\xi|^2/\w{\xi}^2
\end{equation}
which is strongly $\sigma  $ temperate on $ \st $, since $ g^\sh_0 = g^\sh $ on $ T\st$.
Then $ \lambda f \in  S(Mh^{1/2} \w{\xi} , G_0) $ and $ b_1 f_0 \in S(MH^{1/2}h^{1/2} \w{\delta_0}\w{\xi} , g^\sh_0)$.

\begin{lem}\label{f0prop}
If $ | C| \ls \mu $ then we have
\begin{equation}\label{muest}
|\w{C^{Wick}u,u}| \ls 	\w{\mu^{Wick}u,u} \ls h^{1/2} \mn{\w{D_x} u}^2 =   h^{1/2}( \mn{D_x u}^2  + \mn{u}^2)
\end{equation}
where $ D_x u = (D_{x_1}u, D_{x_2}u, \dots ) $.	
\end{lem}

\begin{proof}
By taking the real and imaginary part it suffices to prove the estimate for real valued $ C $. We have $ | C | \ls \mu $ so $\pm \w{C^{Wick}u,u} \ls 	\w{\mu^{Wick}u,u}  $. Now  $ \mu^{Wick} = \nu^w  $ where $ \nu \in S(\mu, g^\sh_0) $ by Proposition~\ref{propwick}, so $\nu_0^w =  \w{D_x}^{-1}\nu^w \w{D_x}^{-1}\in \op S(h^{1/2}, g^\sh_0) $ which gives
$$
\w{ \mu^{Wick} u , u } = \w{ \nu_0^w \w{D_x} u, \w{D_x}u}\ls h^{1/2}\mn{\w{D_x} u}^2
$$
which proves the result.
\end{proof}

\section{The multiplier estimate}\label{lower}

In this section we shall obtain a proof of Proposition~\ref {mainprop}
which involves giving lower bounds on  $\re b_{T}^w f_1^w$, with the multipler $b_{T}^w = B_T^{Wick}$ having symbol $B_T = {\delta}_0 + {\varrho}_0+ \lambda $ given by Proposition~\ref{wickweyl}. Also,  $ f_1 \cong  f + f_0 $ modulo $ R $, where $ f \in S(M,G) $ and $ f_0 = \partial_\eta f \cdot  r \cdot \xi $  with $ r \in S(1,g) $ so that $ f_0 \in S(MH^{1/2}h^{1/2}\w{\xi}, G_0) $ by Remark~\ref{f0rem}.
Here $ G = Hg/h = H g^\sh$, with constant $ g \le h^2 g^{\sigma} $ and $ H $ given by Definition~\ref{g1def}. The weight $ M $ is given by Definition~\ref{Mdef}, the metric $ G_0 $ by~\eqref{G0def} and the weight $ m $ by Definition~\ref{h0def}.
 The results will only depend on the seminorms of $f$ in $ S(h^{-1},g)$, and we will assume the coordinates chosen so that $ t= 0 $ and $ x = 0 $ at $ w_0 \in \st $.
 We shall follow Section 7 in \cite{de:X} with some necessary changes because of the different conditions, metrics and normal forms.

\begin{prop}\label{lowersign}
	Let $B_T = {\delta}_0 + {\varrho}_0 + \lambda$ given by Definition~\ref{multdef}, so ${\delta}_0 = \delta$  is given by
	Definition~~\ref{d0deforig}, ${\varrho}_0 =  \varrho$ is 
	real valued and Lipschitz continuous,
	satisfying $|{\varrho}_0| \le m$ when $ |t| \le T $, with ~$m \le \w{{\delta}_0}/2$ given by
	Definition~\ref{h0def} and $ \lambda = \epsilon h^{1/2}\w{Lx,\xi}/T \in S(h^{1/2}\w{\xi},G_0)$ uniformly when $ | x| \le T $, where  $ 0 < \epsilon \le 1 $ and  $ L $ is given by Lemma~\ref{multlem}. Then  for small enough $ T $ we can find $C \in S(m, g^\sh) + S(\mu, g_0^\sh)$ so that
	\begin{equation}\label{lowerbound}
	\re\sw{B_T^{Wick}f_1^w u,u} \ge \sw{C^w u,u}  
	\end{equation}
	if $ u \in C_0^\infty$ has support where $ | t | \le T $ and $ | x | \le T $.
\end{prop}

\begin{proof} 
     In the proof we shall treat~$ \xi  $ as parameter and assume 
     the coordinates $ w = (t,x,y;\tau, \eta) $ chosen so that $g^\sh(w) = |w|^2$, i.e., $ g^\sh $ ON coordinates on $ \st $.
     We shall localize in $w $ with respect to the metric $G \ge g$ on $ \st $, i.e., when $ \xi = 0 $,	and then estimate the localized operators.
	We shall include $ t $ and $ x $  in $ w $ and use the neighborhoods 
	\begin{equation}\label{omegajdef} 
	{\omega}_{w_0}({\varepsilon}) = \set{w:\ |w-w_0| <
		{\varepsilon}H^{-1/2}(w_0)} 
	\end{equation}
	which gives that $\max (  | t  |  , | x  | )< {\varepsilon}H^{-1/2}(w_0) h^{1/2} \le 3\varepsilon$.
	Now, if
	${\varepsilon}_0$ is small enough then
	$ H(w)$ and $ M(w)$ will only vary with a fixed factor
	in~${\omega}_{w_0}(2{\varepsilon}_0)$.  
	By the uniform Lipschitz continuity of
	$w \mapsto {\delta}_0(w)$ we can find ${\kappa}_0 >0$
	with the following property: for $0 < {\kappa}\le {\kappa}_0$ there
	exist positive constants $c_{\kappa}$ and ${\varepsilon}_{\kappa} \le \varepsilon_0$ so
	that
	\begin{alignat}{2}
	&|{\delta}_0(w)| \le {\kappa}H^{-1/2}(w)&\qquad &w \in
	{\omega}_{w_0}(2{\varepsilon}_{\kappa})  \qquad\text{or}\label{d0case}\\
	&|{\delta}_0(w)| \ge c_{\kappa} H^{-1/2}(w)&\qquad &w \in
	{\omega}_{w_0}(2{\varepsilon}_{\kappa}) \label{d1case}
	\end{alignat}
	In fact, we have by the Lipschitz continuity that $|{\delta}_0(w) -
	{\delta}_0(w_0)| \le {\varepsilon} H^{-1/2}(w_0)$ when
	$w\in {\omega}_{w_0}({\varepsilon})$. Thus, if
	${\varepsilon}_{\kappa} \ll {\kappa}$ we obtain that ~\eqref{d0case}
	holds when $|{\delta}_0(w_0)| \ll {\kappa}H^{-1/2}(w_0)$ and
	~\eqref{d1case} holds when $|{\delta}_0(w_0)| \ge
	c{\kappa}H^{-1/2}(w_0)$.

	By shrinking ${\kappa}_0$ we may assume 
-	that $M \cong |f'|H^{-1/2}$ when $
	|{\delta}_0| \le {\kappa}_0H^{-1/2}$ and $H^{1/2} \le {\kappa}_0$
	according to Proposition~\ref{mproplem}.
	Let ${\kappa}_1$ be given by Proposition~\ref{ffactprop},
	${\kappa}_2$ by Proposition~\ref{wickweyl}, and
	let ${\varepsilon}_{\kappa}$ and $c_{\kappa}$ be given by
	\eqref{d0case}--\eqref{d1case} for ${\kappa} = \min
	({\kappa}_0,{\kappa}_1, {\kappa}_2)/2$. 
	Using Proposition~\ref{wickweyl} with $\varepsilon = c_{\kappa}$
	we find that 
	\begin{equation} \label{kappa3ref}
	\sgn( f)(\delta_0 + \varrho_0) \ge c_{\kappa}H^{-1/2}/3\qquad\text{in ~$
		{\omega}_{w_0}(2{\varepsilon}_{\kappa})$} 
	\end{equation}
	if $H^{1/2} \le c_{\kappa}/3$ and ~\eqref{d1case} holds in ~$
	{\omega}_{w_0}({\varepsilon}_{\kappa})$. 
	
	Choose real symbols
	$\set{{\psi}_j(w)}_j$ and $\set{{\phi}_j(w)}_j
	\in S(1,G)$ with values in ~~$\ell^2$, such that $ 0 \le \psi_j \le 1 $, 
	$\sum_{k} {\psi}_j^2 \equiv 1$, ${\psi}_j{\Psi}_j = {\psi}_j$ with
	$0 \le {\Psi}_j = {\phi}_j^2 \le 1$ which gives 
	$\set{{\Psi}_j(w)}_j \in S(1,G)$ with values in~$\ell^2$ so that  
	\begin{equation}\label{Odef}
		\supp {\phi}_j \subseteq
		{\omega}_j = {\omega}_{w_j}({\varepsilon}_{\kappa})
	\end{equation}
	We shall suppress $ T $, writing $B^{Wick} = b^w$ where $b= {\delta}_1 + {\varrho}_1 + \lambda$ is
	given by Proposition~\ref{wickweyl}. In the following, we shall for $ j  \in \bn $ denote $ A_{jk} = {\Psi}_kf_j b = f_{jk} b$ for $ j = 0,\, 1 $,
	and $A_{k} = {\Psi}_k f b = f_k b  $ where $ f_k $ should not be confused with $ f_0 $ and $ f_1 $.  
	
	\begin{lem}\label{bfconglem}
		When $|t| \le T$ and $|x| \le T$ we have  $A_{1j} \in S(MH^{-1/2},g^\sharp)\bigcap
		S^+(M,g^\sharp) + S(Mh^{1/2}\w{\xi}, g_0^\sharp)$ uniformly in $j$, 
		\begin{equation}\label{bfcong}
		\re b^w f_1^w \cong  \sum_{k}^{} {\psi}_k^w A_{1k}^w{\psi}_k^w  
		\qquad \text{modulo $\op S(m,g^\sharp) + \op S(\mu, g_0^\sharp) $ uniformly}
		\end{equation}
		 and $A_{1k}^w \cong \re b^w  f_{1k}^w $ modulo $\op S(m, g^\sh) + S( \mu ,g_0^\sh)$ uniformly in ~$j$.
	\end{lem}

	In the proof, we will use that  Proposition~\ref{mestprop}
    gives
    \begin{equation} \label{fprimest}
	MH^{3/2}\w{{\delta}_0}^2 \le Cm
    \end{equation} 
   so we may ignore terms with symbols in $\op
   S(MH^{3/2}\w{{\delta}_0}^2,g^\sharp) $. 

	\begin{proof}
		 Since $ f_1 = f + f_0 $  we have $ A_{1k} = A_k+ A_{0k}$ and we will start with~$ f_{0} $. First we note that  $ f_0 = r_0 \cdot \xi $ with  $r_0 \in S(MH^{1/2} h^{1/2},G) \bigcap S(1,g) $, which gives $ \lambda^wf_0^w  \in \op S(\mu, G_0)$ so we can skip this term. Let $b_1 = \delta_1 + \varrho_1 \in S(H^{-1/2}, g^\sh) \bigcap S^+(1, g^\sh) $. Since $ b_1 $ is constant in $ x $ we obtain $ \re b_1^w f_0^w =  (b_1f_0)^w + R^w $, where $ R = r\xi $ with $r \in S(h, g^\sh) $ giving  $R^w \in  \op S(\mu, g_0^\sh) $. In fact, $ r $ has an asymptotic expansion in $ S(h^{k/2}, g^\sh) $ for $ k \ge 0 $.
		
		We also find that $  \psi_k^wA_{0k}^w \psi_k^w  = ( A_{0k} \psi^2_k)^w + C^w$ with  $ C = c_0 + c_1\xi $ with $ c_0 \in S(MH^{3/2}, g^\sh)  $  and $c_1 \in S(MH^{3/2}h^{1/2}, g^\sh) $ which gives  $C^w \in  \op S(m, g^\sh) + S(\mu, g_0^\sh) $.  In fact, $ c_0  $ has an asymptotic expansion in $ S(MH^{k/2}, g^\sh)  $ for $ k \ge -1 $  and $c_1 \in S(MH^{k/2}h^{1/2}, g^\sh) $ for $ k \ge 0 $ and any $ x $ derivative of the symbols gives the factor $ H^{1/2}h^{-1/2} $.

		 Next, we will study $ f \in S(M,G) $. In that case $\lambda^w f^w  = (f\lambda - i h^{1/2} \w{Lx, \partial_x f}/2T)^w  $ since $ \partial_\xi f \equiv 0 $, which gives $\re \lambda^w  f^w = (f \lambda)^w $. We have 
		$A_k \in S(MH^{-1/2},g^\sharp)\bigcap
		S^+(M,g^\sharp) + \op S(M h^{1/2}\w{\xi} , G_0)$ uniformly in $j$.
	Since $b \in
		S(H^{-1/2},g^\sh) + S(h^{1/2}\w{\xi} ,G_0)$, $\set{{\psi}_k}_k \in S(1,G)$, $A_k \in S(MH^{-1/2},g^\sh)  + \op S(M h^{1/2}\w{\xi} ,G_0) $ uniformly  with  values in
		$\ell^2$, $h^{1/2}\w{\xi} \le \mu  $ and $ H^{-1/2} \lesssim h^{-1/2} $, we find by
		Lemma~\ref{calcrem} and Remark~\ref{vvcalc} that the symbols of
		$b^wf^w$, $b^wf_k^w$ and $\sum_k {\psi}_k^wA_k^w{\psi}_k^w$ have expansions in
		$S(MH^{j/2}, g^\sh) + S(M  H^{k/2}\mu , G_0)$. 
		Observe that in the domains ~${\omega}_k$
		where $H^{1/2} \ge c > 0$, we find that $ M \lesssim H^{-1} \lesssim 1 $ so the
		symbols of $\sum_{k} {\psi}_k^w A_k^w{\psi}_k^w$, $b^w f_k^w$ and
		$b^wf^w$ are in $ S(MH^{3/2}, g^\sharp) + S(\mu, g^\sharp_0)$ giving the result in this
		case. Thus we may assume $H^{1/2} \le {\kappa}_2/2$ in what follows.
		We shall consider the neighborhoods where~\eqref{d0case}
		or~\eqref{d1case} holds.
		
		If~\eqref{d1case} holds then we find  $\w{{\delta}_0} \cong
		H^{-1/2}$ so $ S(MH^{1/2},g^\sharp) \subseteq S(m, g^\sharp)$  	in ${\omega}_k$ by ~\eqref{fprimest} and $S(MH \mu ,G_0) \subseteq S(\mu ,G_0)$
	  since $M \ls H^{-1}$.  Since $b_1\in S^+(1,g^\sharp)$ 
		  we find that $A_k \in S^+(M,g^\sh) + S(M\mu ,G_0)$ and the symbols of both $b^w f^w $ and
		$\sum_{k}^{} {\psi}_k^w A_k^w {\psi}_k^w$ are equal to $\sum_k
		{\psi}_k^2 A_k \cong fb$ modulo $S(MH^{1/2}, g^\sharp) + S(MH\mu ,G_0)$ in
		~${\omega}_k$. Similarly, we find that the symbol of $ b^w f_k^w$ is equal to
		$A_k$ modulo $S(MH^{1/2},g^\sharp) + S(\mu ,G_0)$, which proves the result in this
		case.
		
		Next, we consider the case when ~\eqref{d0case} holds with ${\kappa} =
		\min ({\kappa}_0,{\kappa}_1, {\kappa}_2)/2$ and $H^{1/2} \le
		{\kappa}_2/2$ in~${\omega}_k$. Then $\w{{\delta}_0} \le {\kappa}_2
		H^{-1/2}$ so $b = {\delta}_1 + {\varrho}_1 + \lambda \in S(H^{-1/2},G)+
		S(m, g^\sh) + S(\mu ,G_0)$ in~${\omega}_k$ by Proposition ~\ref{wickweyl}.
		Since $ \re \lambda^w f^w = (\lambda f)^w $ we obtain from Lemma~\ref{calcrem} that the symbol
		of $\re b^w f^w -(fb)^w$ is in $S(MH^{3/2}, G) + S(MHm,
		g^\sh)  \subseteq S(m,g^\sh) $ in ~${\omega}_k$. Similarly, we find that $A_k^w \cong \re b^w f_k^w$ modulo
		$\op S(m,g^\sh)$.  Since $A_k \in S(MH^{-1/2},G)+
		S(Mm,g^\sharp) + S(M\mu ,G_0)$ uniformly in this case, we find that the symbol of $\sum_{k}
		{\psi}_k^w A_k^w {\psi}_k^w$ is equal to $bf$ modulo $ S(m,g^\sh) + S(\mu ,G_0)$
		in ~${\omega}_k$, which proves ~\eqref{bfcong} and
		Lemma~\ref{bfconglem}.
	\end{proof}
	
	In order to estimate the localized operator
	we shall use the following
	
	\begin{lem}\label{lowerlem}
		If\/ $A_{1j} = {\Psi}_jf_1b$ is given by Lemma~\ref{bfconglem} then there
		exists\/ $C_{j} \in S(m, g^\sharp) + S(\mu ,g^\sharp_0)$ uniformly when $|t| \le T$ and $|x| \le T$, such that
		\begin{equation}\label{lowerest}
		\sw{A_{1j}^w u,u} \ge \sw{C^w_{j} u,u}\qquad
		\end{equation} 
		when $ u \in C_0^\infty $.
	\end{lem}

	We obtain from~\eqref{bfcong} and~\eqref{lowerest} that 
	$$ 
	\re\sw{b^w f^wu,u} \ge \sum_{j} \sw{{\psi}_j^wC^w_{j}{\psi}_j^w u,u}
	+\sw{R^wu,u}\qquad\text{$u \in C_0^\infty$}
	$$ 
	where $\sum_{j}{\psi}_j^wC^w_{j}{\psi}_j^w$ and $R^w \in \op
	S(m,g^\sharp)+ \op S(\mu ,g^\sharp_0)$, which gives Proposition~\ref{lowersign}.
\end{proof}

\begin{proof}[Proof of Lemma~\ref{lowerlem}]
	In the following, we shall assume that $ \max(|t|,| x|)  \le T $.
	As before we are going to consider the
	cases when $H^{1/2} \cong 1$ or $H^{1/2} \ll 1$, and when
	~\eqref{d0case} or ~\eqref{d1case} holds in $
	{\omega}_{w_j}(2{\varepsilon}_{\kappa})$ for ${\kappa} = \min
	({\kappa}_0,{\kappa}_1, {\kappa}_2)/2$. When $H^{1/2} \ge
	c > 0$ we find that
	$A_{1j} \in S(MH^{3/2},g^\sharp) + S(MHh^{1/2}\w{\xi} ,g^\sharp_0)
	\subseteq S(m,g^\sharp) + S(\mu ,g^\sharp_0)$ uniformly by
	~\eqref{fprimest} which gives the lemma with $C_j = A_{1j}$ in this
	case. 
	For handling this case where $ H^{1/2} \ll 1$ we shall need the following result.
	
	\begin{lem}\label{f0prep}
	If  $ F \in S(M_0,G) $, where $ M_0 $ is a weight for $ G $, and $ \pm F \ge 0 $ in  ${\omega}_{w_j}(2{\varepsilon}_{\kappa})$ then we have that $| \partial_\eta F| \le  C_\kappa \sqrt{F}M_0^{1/2}H^{1/2}h^{1/2}$ in   ${\omega}_j$.
	\end{lem}

	\begin{cor}\label{f0cor}
	We obtain from Lemma~\ref{f0prep} that $ \partial_\eta f = \partial_\eta \alpha_0 \delta_0 + \alpha_0 \partial_\eta \delta  = \sqrt{\alpha_0} r_1$ where $ r_1 \in S(M^{1/2}H^{1/4}h^{1/2},G) $ in $ \omega_j $.
	\end{cor}

In fact, we find from Lemma~\ref{f0prep} that  $ \partial_{\eta} \alpha_0 \in S(\sqrt{\alpha_0} M^{1/2}H^{1/4}h^{1/2}, G)   $   since  $MH^{1/2} \ls \alpha_0 \in S(M H^{1/2},G) $ and we have $ \partial_{\eta} \delta_0 \in S(h^{1/2}, G) $  in $ \omega_j $.

\begin{proof}[Proof of Lemma~\ref{f0prep}]
	For any $ w \in \omega_j $  we may choose $ g^\sh $ orthogonal coordinates so that $ w = 0 $ and $ F \ge 0 $ in $ | \eta | <  \varepsilon_\kappa H^{-1/2}h^{-1/2} $. Then by Taylor's formula and the slow variation there exists $ C_\kappa > 0$ so that
	\begin{equation}
	| \eta\cdot\partial_\eta F(0) | \le F(0) + C_\kappa M_0H h| \eta |^2
 	\end{equation} 
 	Then by choosing $ | \eta | = \varepsilon_\kappa \sqrt{F(0)/M_0 Hh} \ls \varepsilon_\kappa	 H^{-1/2}h^{-1/2}$ we find that \begin{equation}
 			| \partial_\eta F(0) | \le   C'_\kappa \sqrt{F(0)M_0 Hh} 
 	\end{equation}
 	which proves the result.
\end{proof}
	
	In the following, we shall assume that 
	\begin{equation}\label{kappa4def}
	H^{1/2} \le {\kappa}_4 = \min({\kappa}_0, {\kappa}_1, {\kappa}_2,
	{\kappa}_3
	)/2 \qquad \text{in ${\omega}_j$} 
	\end{equation}
	with ${\kappa}_3  = 2c_{\kappa}/3$
	so that \eqref{kappa3ref} follows from ~\eqref{d1case}.  
	First, we consider the case when $H^{1/2} \le {\kappa}_4$ and
	~\eqref{d1case} holds in ${\omega}_j$.  
	Since $|{\delta}_0(w)| \ge c_{\kappa} H^{-1/2}(w)$, we find $ \pm f \ge 0 $ and
	$\w{{\delta}_0} \cong H^{-1/2}$ in~${\omega}_{w_j}(2{\varepsilon}_{\kappa})$. As before we may ignore
	terms in $ S(MH\mu ,g^\sh_0)
	\subseteq  S(\mu ,g^\sh_0) $ and $S(MH^{1/2},g^\sharp) \subseteq
	S(m, g^\sharp)$  in ${\omega}_j$ by
	~\eqref{fprimest}. 
	We shall estimate $ A_{1j} = A_j + A_{0j}$, starting with $ A_j $.
	If $ B_1 = \delta_0 + \varrho_0 $ we find from ~\eqref{kappa3ref} that 
	$f B_1\ge 0$ in ${\omega}_{w_j}(2{\varepsilon}_{\kappa})$. We shall only consider the case $ \sgn(f) = \sgn(B_1) = 1 $, for the other case we may replace the symbols by their absolute values. Since $ B_1 \gtrsim H^{-1/2} $ and $ B_1 \in S(H^{-1/2}, g^\sh) \bigcap S^+(1, g^\sh)$ we find $ B^{1/2} \in S(H^{-1/4}, g^\sh) \bigcap S^+(H^{1/4}, g^\sh) $ and $ B^{-1/2} \in S(H^{1/4}, g^\sh) \bigcap S^+(H^{3/4}, g^\sh) $.
	Since $f_j \in S(M, G)$, we find $f_j^w \cong f_j^{Wick}$ modulo
	$\op S(MH,G)$ by Proposition~\ref{propwick}. 
	As before, $ \re \lambda^w f_j^{w}= (\lambda f_j) ^{w}$ so 
	we find from Example~\ref{wickcompexe} that
	\begin{equation*}
	A_j^w \cong \re b^w f_j ^w  \cong \re B^{Wick} f_j^{Wick} \cong (f_jB)^{Wick} 
	\end{equation*}
	modulo $\op S(m,g^\sh) + \op S(\mu, G_0)$. In fact, we have  $ b^w = B^{Wick}= (B_1 + \lambda)^{Wick} $ and $ \lambda f_j = s_1\cdot \xi $ with  $s \in S(Mh^{1/2}, G) $ so $ (\lambda f_j) ^{w} \cong  (\lambda f_j) ^{Wick}  $ modulo $ \op S(\mu, G_0)$. 
	
	Similarly,  since $ f_{0j} \in S(MH^{1/2}h^{1/2}\w{\xi}, G_0 ) $   is linear in $ \xi $ we find that $ f_{0j} ^w \cong  f_{0j}^{Wick}$ modulo $ S(MH^{3/2}h^{1/2}\w{\xi}, G_0 ) $, thus
	\begin{equation*}
	A_{0j}^w \cong \re b^w f_{0j}^w \cong \re B^{Wick} f_{0j}^{Wick} \cong (f_{0j}B)^{Wick} 
	\end{equation*}
	modulo $\op S(MHh^{1/2}\w{\xi},g^\sh_0) + \op S(MH^{3/2}h\w{\xi}^2, G_0) \subset \op S(\mu, g^\sh_0)$. By Lemma~\ref{f0prep} we have $ | \partial_\eta f | \ls \sqrt{f}M^{1/2}H^{1/2}h^{1/2}$ in $ \omega_j $ which gives
	\begin{equation}
		| f_{0j} B_1 | \le \varepsilon f_j B_1 + C_\varepsilon \Psi_j MH^{1/2}h\w{\xi}^2 \quad \forall \, \varepsilon > 0
	\end{equation}
	where $ MH^{1/2}h\w{\xi}^2 \ls \mu $.
    If $ \varepsilon \le 1/2 $ we find modulo $ S(\mu, g^\sh_0) $ that
    $$ 
    f_{1j} B \gtrsim f_j(B_1/2 + \lambda) = f_j\left(\sqrt{B_1/2} + \lambda/\sqrt{2B_1}\right)^2 - f_j \lambda^2/2B_1
    $$
    so $(f_{1j}B)^{Wick}  \ge  -(  f_j \lambda^2/2 B_1)^{Wick}  \in\op S(MH^{1/2}\lambda^2, g_0^\sh) \subset \op S(\mu, g_0^\sh)$.

     Finally, we consider the case when ~\eqref{d0case} holds with ${\kappa} =
	\min ({\kappa}_0,{\kappa}_1, {\kappa}_2)/2$ and $H^{1/2} \le
	{\kappa}_4\le {\kappa}$ in ${\omega}_j$. Then $\w{{\delta}_0} \le
	2{\kappa}H^{-1/2}$ so we obtain from Proposition
	~\ref{mproplem} that $M \cong |f'|H^{-1/2}$ and 
	${\delta}_0 \in S(H^{-1/2}, G)$  in~${\omega}_j$.
	We shall estimate $ A_{1j} = A_j + A_{0j}$, starting with $ A_j $
    using an argument of Lerner ~\cite{ln:cutloss}. We have that $b^w =
	({\delta}_0 + {\varrho}_0 + \lambda)^{Wick} = B^{Wick}$, where 
	$ 	|{\varrho}_0|\le m \le H^{1/2}\w{{\delta}_0}^2/2 $ 
	by ~\eqref{hhhest}. Also, Lemma~\ref{bfconglem}
	gives  $A_{1j}^w \cong \re b^{w}  f_{1j}^w= \re B^{Wick} f_{1j}^w$
	modulo $\op S(m,g^\sharp)+ \op S(\mu ,g^\sharp_0)$. As before, $ f_{1j} = f_j + f_{0j} $ and we shall start with $\re b^{w}  f_{j}^w  $.
	Take ${\chi}(t) \in C^{\infty}(\br)$ such that
	$0 \le {\chi}(t) \le 1$, $|t| \ge 2$ in
	$\supp {\chi}(t)$ and
	${\chi}(t) = 1$ for $|t| \ge 3$.
	Let ${\chi}_0 = {\chi}({\delta}_0)$, then $ {\chi}_0 \in S(1, g^\sh)  $,
	$2 \le |{\delta}_0|$ and
	$\w{{\delta}_0}/|{\delta}_0| \le 3/2$
	in $\supp {\chi}_0$, thus 
	\begin{equation}\label{b0ref} 
	1 + {\chi}_0{\varrho}_0/{\delta}_0 \ge 1 -
	{\chi}_0\w{{\delta}_0}/2|{\delta}_0| \ge 1/4
	\end{equation}
	Since $|{\delta}_0| \le 3$ in $\supp (1-{\chi}_0)$
	we find by Proposition~\ref{wickweyl}
	that 
	$$
	B^{Wick} \cong ({\delta}_0 + {\chi}_0{\varrho}_0 + \lambda)^{Wick} 
	$$ 
	modulo $\op S(m/\w{{\delta}_0}, g^\sh) \subseteq \op
	S(H^{1/2}\w{{\delta}_0}, g^\sh)$ by ~\eqref{hhhest}.
	Since 
	$|{\chi}_0{\varrho}_0/{\delta}_0| \le 3H^{1/2}
	\w{{\delta}_0}/4$ and $ \delta_0 \in S^+(1,g^{\sh}) $ we find from~\eqref{wickcomp1} that 
	\begin{equation}\label{BWick}
	B_1^{Wick} \cong  {\delta}_0^{Wick}B_0^{Wick}  
	\qquad\text{modulo $\op
		S(H^{1/2}\w{{\delta}_0}, g^\sh)$}.
	\end{equation} 
	where $B_0 = 1 + {\chi}_0{\varrho}_0/{\delta}_0$.
	Proposition~\ref{propwick} gives
	$({\chi}_0{\varrho}_0/{\delta}_0)^{Wick} \in \op
	S(H^{1/2}\w{{\delta}_0}, g^\sh)$ and
	${\delta}_0^{Wick} = {\delta}_1^w $ where  ${\delta}_1 = {\delta}_0 + {\gamma}$
	with ${\gamma} \in  S(H^{1/2}, G)$ in ${\omega}_j$.
	Thus Lemma~\ref{calcrem} gives
	\begin{equation}\label{nycalcref0}
	B^{Wick}\cong {\delta}_0^{Wick} B_0^{Wick}  + \lambda^{Wick} \cong {\delta}_0^wB_0^{Wick}  + \lambda^{w} + c^w  \qquad\text{modulo $\op
		S(H^{1/2}\w{{\delta}_0}, g^\sh)$}
	\end{equation}
	where $c \in S(H^{-1/2}, g^\sh)$ such that $\supp c \bigcap {\omega}_j = \emptyset$.
	
	We find from Proposition~\ref{mproplem} that $ f =
	{\alpha}_0\delta_0 $, where $MH^{1/2} \ls  {\alpha}_0 \in
	S(MH^{1/2}, G)$, which gives ${\alpha}_0^{1/2} \in S(M^{1/2}H^{1/4},
	G)$. Let
    \begin{equation}\label{ajdef}
    	a_j = {\alpha}_0^{1/2}{\delta}_0{\phi}_j \in S(M^{1/2}H^{-1/4},G)
    \end{equation}
	then the calculus gives
	\begin{equation}\label{nycalcref1}
	\re a_j^w  ({\alpha}_0^{1/2}{\phi}_j)^w \cong f_j^w\qquad\text{modulo $\op S(MH,G)$}.
	\end{equation}
	since $f_j = {\Psi}_jf = {\phi}_j^2f\/$.
	Similarly, we find that $f_j^wc^w \in \op S(MH^{3/2},g^\sh)$ by the expansion and
	\begin{equation}\label{nycalcref2}
	\re  f_j^w{\delta}_0^{w} \cong a_j^w  a_j^w\qquad\text{modulo $\op S(MH^{3/2},G)$}
	\end{equation}
	with imaginary part in $\op S(MH^{1/2},G)$. We obtain from~
	\eqref{nycalcref0} and~\eqref{nycalcref1} that
	\begin{multline}\label{W1}
	f_j^wB^{Wick} \cong f_j^w({\delta}_0^wB_0^{Wick} + \lambda^w  +  c^w + r^w) \\ \cong
	f_j^w{\delta}_0^wB_0^{Wick} + f_j^w\lambda^w + a_j^w R_j^w \quad\text{modulo $\op S(m,g^\sh)$}
	\end{multline}
	where $r \in S(H^{1/2}\w{{\delta}_0}, g^\sh)$ which gives $R_j= ({\alpha}_0^{1/2}{\phi}_j)^wr^w \in
	S(M^{1/2}H^{3/4}\w{{\delta}_0}, g^\sh)$. Since 
	\begin{equation*}
	\re FB = \re (\re F)B + i [\im F, B]
	\end{equation*}
	when $B^* = B$, we find from ~\eqref{nycalcref2} by taking $ F =  f_j^w{\delta}_0^w$ and $ B = B_0^{Wick} $ that 
	\begin{equation}\label{W2}
	\re f_j^w{\delta}_0^wB_0^{Wick}  \cong \re a_j^w  a_j^wB_0^{Wick}
	\qquad\text{modulo $\op S(m,g^\sh)$.}
	\end{equation}
	In fact, since $B_0 = 1 + {\chi}_0{\varrho}_0/{\delta}_0$ and
	$({\chi}_0{\varrho}_0/{\delta}_0)^{Wick} \in \op
	S(H^{1/2}\w{{\delta}_0}, g^\sh)$ we find
	\begin{equation}\label{commrel}
		[a^w,B_0^{Wick}] = [a^w,
	({\chi}_0{\varrho}_0/{\delta}_0)^{Wick}]
	\in \op S(M H^{3/2}\w{{\delta}_0}, g^\sh)
	\end{equation}
	when $a \in S(MH^{1/2}, G)$.
	Similarly, since $a_j \in S(M^{1/2}H^{-1/4},G)$  we obtain that
	\begin{equation}\label{W3}
	a_j^w a_j^wB_0^{Wick} \cong
	a_j^w(B_0^{Wick}a_j^w + s_j^w) \qquad\text{modulo $\op S(m,g^\sh)$}
	\end{equation}
	where $s_j \in S(M^{1/2}H^{3/4}\w{{\delta}_0}, g^\sh)$.
	We also have 
	\begin{equation}\label{W4}
	 \re  \lambda^{Wick} f_j^w =  \re f_j^w \lambda^w \cong  (f_j \lambda)^w \cong \re a_j^w (\alpha_0^{1/2} \lambda \phi_j)^w 
	\end{equation}
	modulo $ \op S(\mu, G_0) $ where $ \alpha_0^{1/2} \lambda \phi_j \in S(M^{1/2}H^{1/4}h^{1/2}\w{\xi}, G_0) $.
		Since $B_0 \ge 1/4$ we find from
	\eqref{W1}--\eqref{W4} that 
	\begin{equation}\label{fjest}
	\re B^{Wick} f_j^w  \gtrsim \frac{1}{4}a_j^w  a_j^w + \re a_j^w  S_j^w
	\qquad\text{modulo $\op S(m,g^\sh) + \op S(\mu, G_0)$} 
	\end{equation}
	where $S_j \in 
	S(M^{1/2}H^{3/4}\w{{\delta}_0}, g^\sh) + S(M^{1/2}H^{1/4}h^{1/2}\w{\xi}, G_0) $. 
	
	We are  going to complete the  square in~\eqref{fjest}, but before that we must handle the term $\re B^{Wick} f_{0j}^w  =  \re b^w f_{0j}^w $. As before, we find that $\lambda^{Wick}  f_{0j}^w  \in \op S(\mu ,G_0) $ uniformly since we have $ M H^{1/2}h^{1/2}\w{\xi}\lambda \ls \mu $. For the term $\re  B_1^{Wick} f_{0j}^w $ with $ B_1 = \delta_0 + \varrho_0 $ we need the following result.

	\begin{lem}\label{f0lem}
	  For any $ \varepsilon > 0 $ there exists  $ R_\varepsilon \in  S(m,g^\sh) + S(\mu, g^\sh_0)$ so that
	  \begin{equation}\label{f0jest}
	  | \re\w{ B_1^{Wick}  f_{0j}^w u,u} | \le \varepsilon\w{a_j^w  a_j^wu,u} +  \w{R_\varepsilon ^w u,u}\qquad u \in C^\infty_0 
	  \end{equation}
	\end{lem}
	
	By using~\eqref{fjest} and~\eqref{f0jest} we obtain for $ \varepsilon \le 1/12 $ that
	\begin{equation}\label{fjest0}
	\re B^{Wick} f_j^w  \ge \frac{1}{6}a_j^w  a_j^w + \re a_j^w  S_j^w
	\qquad\text{modulo $\op S(m,g^\sh) + \op S(\mu, g^\sh_0)$} 
	\end{equation}
	where $S_j \in 
	S(M^{1/2}H^{3/4}\w{{\delta}_0}, g^\sh) + S(M^{1/2}H^{1/4}h^{1/2}\w{\xi}, G_0) $. 	
    By completing the square, we find
	\begin{equation*}
	A_j^w \cong \re f_j^wB^{Wick} \gtrsim \frac{1}{6}\left(a_j^w + 3S_j^w\right)^*
	\left(a_j^w + 3S_j^w\right)
	\ge 0
	\end{equation*}
	modulo $\op S(m,g^\sh) +  \op S(\mu, G_0)$. In fact,  $(S_j^w)^*S_j^w \in
	\op S(m, g^\sh) + \op S(\mu, G_0)$
	since we have $ MH^{3/2}\w{{ \delta}_0}^2 \lesssim m $ and $MH^{1/2}h\w{\xi}^2 \lesssim \mu  $. This gives
	~\eqref{lowerest} and the lemma in this case.  	
	This completes the proof of Lemma~\ref{lowerlem}.
\end{proof}

\begin{proof}[Proof of Lemma~\ref{f0lem}]
     First, we note that $B_1^{Wick} = (\delta_0+ \varrho_0)^{Wick} \cong (\delta_0 B_0 )^{Wick}$
     modulo $ S(H^{1/2}, g^\sh) $, where $ (\delta_0 B_0)^{Wick} \in \op S^+(1,g^\sh) $. Thus by  Propositions~\ref{propwick} and~\ref{wickcomp0} we find that $ \re B_1^{Wick} f_{0j}^w  \cong  \re  B_1^{Wick} f_{0j}^{Wick}  \cong  (f_{0j} \delta_0 B_0)^{Wick} $ modulo $ \op S(\mu, g^\sh_0) $. 
     
     Now since $ \Psi_j = \phi_j^2 $ and $ f_{0j} = \Psi_j \partial_\eta f \cdot r \cdot {{\xi}}$ we can factor the symbol $ f_{0j} \delta_0 B_0 = AB $, where 
     $
     A = \partial_\eta f \delta_0B_0 M^{-1/2}H^{-1/4}h^{-1/2}\phi_j 
     $ 
     and 
     $ B = M^{1/2}H^{1/4}h^{1/2}r\cdot \xi\phi_j $. Then  Corollary~\ref{f0cor} gives $| A| \ls \sqrt{\alpha_0}|\delta_0|\phi_j $
     and Proposition~\ref{propwick} gives
     $$ 
     A^{Wick} \in \op S(M^{1/2}H^{1/4}\w{\delta_0}, g^\sh) \bigcap \op S^+(M^{1/2}H^{1/4}, g^\sh) 
     $$ 
     and $B^{Wick} = \wt B^{Wick} D_x $  with  $\wt B \in S(M^{1/2}H^{1/4}h^{1/2},G) \subset S(h^{1/4}, G)$.
     
     By   Proposition~\ref{wickcomp0} we have $ (AB)^{Wick} \cong \re A^{Wick}  B^{Wick}  $
     modulo $\op S(\mu, g_0^\sh) $. Thus, it suffices to estimate
     \begin{equation}\label{ABest}
     	 \re \w{A^{Wick}  B^{Wick} u,u} \le \varepsilon \mn{A^{Wick} u}^2 + C_\varepsilon \mn{B^{Wick}u}^2 \qquad u \in C^\infty_0
     \end{equation}
     where we have $\mn{B^{Wick}u}^2 \ls  \mn{h^{1/4}D_xu}^2 \le \w{\mu^wu,u}  $. 
     We shall prove~\eqref{f0jest} by estimating  $ A^2 \ls \alpha_0  \delta_0^2\phi_j^2 = f \delta_0\Psi_j  =   a_j^2 $. Here 
     $ a_j  = \sqrt{\alpha_0}\delta_0 \phi_j \in 
     S(M^{1/2}H^{-1/4},G) $ 
     is given by~\eqref{ajdef}. We may then estimate $ A^{Wick} \ls a_j^{Wick}  $ where $a_j^{Wick} = a_j^{w} +  r_2^w$ with $   r_2^w \in  S(M^{1/2}H^{3/4},  G)$.
     Thus we obtain
     \begin{equation}
     (A^{Wick})^2 \le \left( a_j^w + r_2^w\right) \left(a_j^w + r_2^w  \right) \\ \ls  a_j^wa_j^w + r_2^wa_j^w+ a_j^w r_2^w  \ls 2a_j^wa_j^w
     \end{equation}
     modulo $(r_2^w)^2  \in \op S(m, g^\sh) $,
     which gives~\eqref{f0jest} and Lemma~\ref{f0lem}. 
     \end{proof}

We shall finish the paper by giving a proof of
Proposition~\ref{mainprop}.

\begin{proof}[Proof of Proposition~\ref{mainprop}]
	By the assumptions in Proposition~\ref{mainprop} we have
	\begin{equation}
		P^* \cong D_t + A^w + i f_1^w \qquad \text{modulo $ R $}
	\end{equation}
microlocally near $ w_0 \in \st $,  here 
	\begin{equation}
		A = \sum_{jk} a_{jk}\xi_j\xi_k + \sum_j a_j\xi_j + a_0
	\end{equation}
	where $ a_{jk} $ and $ a_j  \in S(1, g)$ are real and $ \{ a_{jk}\}_{jk} $ is symmetric and nondegenerate,
	$f_1 = f + f_0$  where  $f \in 	S(h^{-1}, g)$ is real valued satisfying condition $ \subr(\ol\Psi) $ in~\eqref{pcond0} and $ f_0 = \partial_\eta f \cdot r \cdot \xi $. 
 	Now $ P^* $ in~\eqref{propest} can be perturbed with operators with symbols in $ R $ since $ b_T \in S(h^{-1/2}, g^\sh)  $, so $ |\w{b_T^wr^w u,u}| \ls h^{1/2}\mn{\w{D_x}u}^2 $ because $ b_T^w s^w \in \op S(h^{1/2}, g^\sh)  $ when $ s \in S(h, g) $.

	Let $B_T = {\delta}_0 + {\varrho}_T + \lambda_T$ be given by Definition~\ref{multdef}, where $ \lambda_T =  \epsilon h^{1/2}\w{L(x-x_0), \xi}/T $  given by Lemma~\ref{multlem} with $ x_0 $  being the value of $ x $ at $ w_0 $ and   $ 0 < \epsilon \le 1 $, ${\delta}_0 +
	{\varrho}_T$ is the Lipschitz continuous pseudo-sign for~ $f$ given by
	Proposition~\ref{apsdef} for $0 < T \le 1$, so that $|{\varrho}_{T}|
	\le m \le \w{{\delta}_0}/2$ when $  |t| \le T $. 	Proposition~\ref{apsdef} also gives that
	\begin{equation}\label{dtbt} 
	{\partial_t } ({\delta}_0 + {\varrho}_{T}) \ge 
	m/2T \qquad \text{when $  |t| \le T $ }
	\end{equation}
	We have $B_T^{Wick} = b_T^w$ where $b_T(t,w) \in 
	S(H^{-1/2}, g^\sh) \linebreak[3] 
	\bigcap S^+(1, g^\sh) + S(\mu, g^\sh_0)
	$ uniformly by Proposition~\ref{wickweyl} when $\max ( |t| , |x|)  \le T $. In the following we shall assume that $u \in  C_0^\infty$ has support  where $  \max ( |t| , |x|)  \le T $.

	We are going to consider 
	\begin{equation}
		\im \sw{P^*u, B_T^{Wick}u } = i\sw{[D_t + A^w, B_T^{Wick}]u,u}/2 + \re \sw{f_1^w u, B_T^{Wick} u}
	\end{equation}
	We find by~\eqref{poswick}
	and ~\eqref{dtbt} that
	\begin{equation}\label{dbest}
  i\sw{[D_t, B_T^{Wick}]u,u}/2 =	\sw{{\partial_t}B_T^{Wick} u,u}/2 \ge
	\sw{m^{Wick}u,u}/4T 
	\end{equation}
	when   $ u \in  C_0^\infty $.
	By Proposition~\ref{lowersign}, we find that
	\begin{equation}
	\re \sw{B_T^{Wick} f^w u,u} \ge  
	\sw{C^wu,u} \quad u \in C_0^\infty
	\end{equation}
	with $C \in S(m, g^\sh) + S(\mu, g_0^\sh)$. Propositions~\ref{wprop} and~\ref{f0prop} gives $C_0 > 0$ so that 
	\begin{equation}\label{cest}
	|\sw{C^w u,u}| \le C_0 (\sw{m^{Wick}u,u} + \sw{\mu^{Wick} u,u}) \qquad u \in C_0^\infty
	\end{equation}
	By Lemma~\ref{multlem}  and~\eqref{hhhest}, there exist $ c_j > 0 $ so that when $ | \xi |  \ll h^{-1}$ we have
	\begin{equation}
	 i \sw{[A^w, B_T^{Wick}]u,u}/2  \ge (\epsilon - 2c_1T )\sw{\mu^w u , u} /2T - 3 c_0\epsilon \sw{m^{Wick} u , u}/2T
	\end{equation}
	for $ u \in C^\infty_0 $. In fact, $ h^{1/2} \le 6 m $ by~\eqref{hhhest}  and since $ b_1 $ is constant in $ x $ we have $ [A^w, b_1^{w} ]  = \w{s_2^wD_x, D_x} + s_1^w D_x + s_0^w$ with $s_j \in S(h^{1/2}, g_0^\sh) $.

	We find from ~\eqref{dbest}--\eqref{cest} that we can find  $ \Psi \in S^{2} $ such that $ \st \bigcap \supp \Psi = \emptyset $ and
	\begin{multline}\label{finest}
	 	\im \sw{P^*u, B_T^{Wick}u } \ge \left( \frac{1}{4} - 6c_0 \epsilon
	- C_0 T \right)\sw{m^{Wick}u,u}/T \\ + (\epsilon - 2C_0T- 2c_1T) )\sw{\mu^w u,u})/2T - \mn{\Psi^w u}^2/T
	\end{multline}
	for $u \in  C_0^\infty$. By taking first $ \epsilon $ and then $ T $ small enough we find 
	\begin{equation}\label{finest0}
	\im \sw{P^*u, B_T^{Wick}u }+ \mn{\Psi^w u}^2/T \gtrsim \sw{m^{Wick}u,u}/T + \sw{\mu^w u,u})/T  
	\end{equation}
	for $u \in  C_0^\infty$. 
		
	Since $|\delta_0 + \varrho_T | \le |\delta_0| + m  \le 3\w{\delta_0}/2 $,
	$ h^{1/2}\w{{\delta}_0}^2 \lesssim m $ by ~\eqref{hhhest} and $|\lambda | \lesssim h^{1/4}\mu^{1/2} $, we find 
	$|B_T | \lesssim h^{-1/4}{m}^{1/2} + h^{1/4}\mu^{1/2}$. Thus $h^{1/2}((B_T^{Wick})^2 +1)
	\in \op S(m, g^\sh) + \op S(h^{1/2}\mu , g_0^\sh)$ so Propositions~\ref{wprop} and~\ref{f0prop} give
	\begin{equation}\label{mlow}
	h^{1/2} (\mn{B_T^{Wick}u}^2 + \mn{u}^2 + \mn{D_xu}^2) \lesssim \sw{m^{Wick}u,u}  + \sw{\mu^wu,u} \qquad 
	u \in  C_0^\infty 
	\end{equation}
	Summing up, we obtain that
	\begin{multline}
	h^{1/2}(\mn{B_T^{Wick}u}^2  + \mn u^2 + \mn{D_xu}^2)\le  C_1  \left(\sw{m^{Wick}u,u}  + \sw{\mu^wu,u} \right) \\ \le
	C_2\left(T\im\sw{Pu,B_T^{Wick}u} + \mn{\Psi^w u}^2\right)
	\end{multline}
	if $u \in  C_0^\infty$ has support where $\max ( |t| , |x|) \le T $ which  completes the proof of 
	Proposition~\ref{mainprop}.
\end{proof}

\appendix

\section{Solvability of  quasilinear PDE}\label{app}

In this appendix, we are going to prove a result  that gives Proposition~\ref{approp} in Section~\ref{prep}. 
Let  $ f (x, w)\in C^\infty(\br^{n+m}) $ be real valued and $ (u_0, u_1) \in \br^{1 + n}$, then we shall solve
\begin{equation}\label{pequation}
P(\partial_x u,x, w,\partial_x) u = f, \quad u(x_0, w_0) = u_0 \in \br, \quad  \partial_{x}u(x_0, w_0) = u_1\in \br^n
\end{equation}
where $ u(x,w) \in C^\infty(\br^{n+m}) $ is real valued. Here $ P $ is a quasilinear second order PDO in the $ x $ variables with real $ C^\infty $ coefficients having $ w \in \br^m$ as parameter such that
\begin{equation}\label{psymbol}
P(v,x, w,\partial_x)  = p_2(v,x, w,\partial_x) + p_1(x, w)\partial_x + p_0(x,w) 
\end{equation}
with $ v(x, w) \in C^\infty(\br^{n+m}, \br^n) $, $ p_1(x, w) \in C^\infty(\br^{n+m}, \br^n) $  and $ p_0(x, w)  \in C^\infty(\br^{n+m}, \br) $. 
We assume that the principal symbol vanishes of second order, so that
\begin{equation}\label{pquant}
p_2(v,x,w,\partial_x) = 
\sum_{j,k= 1}^n L_{jk}(v,x,w)\partial_{x_j}\partial_{x_k} 
\end{equation}
where the real quadratic form 
\begin{equation}\label{pcond}
L(u_1,x_0,w_0) = \set{L_{jk}(u_1,x_0,w_0)}_{jk} \quad \text{has maximal rank $ n $ }
\end{equation}
  so that $ P $ is of real principal type, which then holds in a neighborhood of $ (u_1,x_0,w_0) $.

\begin{thm}\label{appthm}
	Let $ P $ be given by~\eqref{psymbol}  so that conditions~\eqref{pquant} and~\eqref{pcond} hold, then for any real valued
	$ f  \in C^\infty(\br^{m+n}) $, $ u_0 \in \br $  and $ u_1 \in \br^n $  there exists a neighborhood $ U $ of   $ (x_0, w_0) $ so that~\eqref{pequation} 
	has a real valued solution $u \in  C^\infty(\br^{m+n}) $ in $ U $. The  neighborhood~$ U $ will only depend on the bounds on $u_0, \ u_1 $,  $ f $ and the coefficients of~$ P $.
\end{thm}

Observe that the solution is not unique,  for uniqueness one needs  hyperbolicity of $ P $ and initial values at a noncharacteristic surface. Since the system~\eqref{changesys} is decoupled into equations on the form~\eqref{psymbol}--\eqref{pcond} we obtain Proposition~\ref{approp} 
from Theorem~\ref{appthm}.

We shall first reduce to the case with vanishing data by changing the dependent variable
\begin{equation}\label{reduc}
u(x,w) = v(x,w) + u_0 + u_1\cdot x 
\end{equation}
in~\eqref{pequation}, then we obtain the following equation for $ v $:
\begin{equation}\label{pmodequation}
P_{0}v = P(\partial v + u_1,x, w,\partial) v   
= f (x,w) -  p_1(x, w) u_1 - p_0(x, w)(u_0 + u_1\cdot x) = f_0(x,w) 
\end{equation}
with $ v(x_0, w_0) = 0 $ and $  \partial_{x}v(x_0, w_0) = 0 $.	
Now the right hand side of~\eqref{pmodequation} depends linearly on both $ f,\  u_0 $ and $u_1 $ and we have that~\eqref{pcond} holds when $ u_1 = 0 $. 

Renaming the operator, for the proof of  Theorem~\ref{appthm} we shall solve the linear equation
\begin{equation}\label{plineq}
P(v(x,w),x, w,\partial_x) u(x, w) = f_0(x, w) \qquad u(x_0, w_0) = 0\quad  \partial_{x}u(x_0, w_0) = 0
\end{equation}
which is a second order real linear PDE  with $ w \in \br^{m}$ and $ v(x, w) \in C^\infty(\br^{n+m}, \br^{n})$ as parameters such that $ v(x_0,w_0) = 0 $.  By using iteration and compactness we shall obtain a solution to~\eqref{pequation}, the proof of Theorem~\ref{appthm} will be at the end of the appendix.

To solve the linear equation~\eqref{plineq} we shall microlocalize using pseudodifferential equations.
In the following we will say that an pseudodifferential operator (or Fourier integral operator)  $ a(v,x,D) $ depends  $ C^\infty $ on a parameter $ v(x) \in   C^\infty  $ if any seminorm of the symbol (and phase function) is bounded by a finite number of  seminorms  of $ v $. 
For operators with symbols in $ S^{-\infty} $ this
means that the $ C^\infty $ kernel is a  $ C^\infty $ function of  $ v $.
Observe that compositions and adjoints of such operators also depend $ C^\infty $ on $ v $, see Lemma~\ref{contlem} and Remark~\ref{fiorem}.

Next, we shall microlocalize  in cones in the $ \xi $ variables. In the following, we will use the classical Kohn-Nirenberg quantization having classical symbol expansions. 

\begin{defn}\label{gammadef}
	For any $ \varepsilon > 0 $  and $ \xi_0 \in \br^n $ such $ | \xi_0 | = 1 $ we let
	\begin{equation}\label{gammaform}
	\Gamma_{\xi_0, \varepsilon} = \set{\xi : \left| \xi /|\xi| - \xi_0  \right| < \varepsilon  }
	\end{equation}
	which is a conical neighborhood of  $ \xi_0 $.
\end{defn}

Recall that a partition of unity is a set $ \set{\phi_j}_j $ such that $ 0 \le \phi_j  \in C^\infty$ and $ \sum_j \phi_j \equiv 1 $.

\begin{rem}\label{microrem}
	For any $ \varepsilon > 0 $ small enough we can find a partition of unity on $ S^*\br^n $ and extend it by homogeneity in $ \xi $ to get a partition of unity $ \set{\varphi_j(\xi)}_j $ on $ T^*\br^n \setminus 0 $ such that  $0 \le \varphi_j \in S^0 $ is homogeneous and supported in $ \Gamma_{\xi_j, \varepsilon} $ for some $ |\xi_j| = 1 $. We can also find  $ \set{\psi_j}_j $ such that $ 0 \le \psi_j \in S^0$ is homogeneous and supported in $ \Gamma_{\xi_j, \varepsilon} $ so that  $ \psi_j = 1 $ on $ \supp \phi_j $. 
	
	We shall also localize when $| \xi | \ge \varrho \ge 1 $ by $ \chi_\varrho(\xi) = \chi(| \xi|/\varrho ) \in C^\infty $, where  $\chi \in C^\infty(\br) $ such that $ 0 \le \chi \le 1 $,  $ \chi(t) = 0 $ when  $t \le 1 $ and  equal to $ 1 $ when $t \ge 2  $.  Let $\varphi_{j, \varrho} = \chi_\varrho \varphi_j $  and $\psi_{j, \varrho} = \chi_\varrho \psi_j $
	then $ \varphi_{j} - \varphi_{j, \varrho} $ and $ \psi_j - \psi_{j, \varrho} $ are in $ S^{-\infty} $, $ \forall\, j $ and $ \forall\, \varrho \ge 1 $.
	We also have that $ \varrho  \chi_{\varrho} \in S^{1}$ uniformly in $ \varrho \ge 1 $, $\forall\,  j $.
\end{rem}

In fact, since  $0 \le  \chi_{\varrho} \le 1$ and $ \varrho \le | \xi | $ in the support of this symbol, we find that  $| \varrho  \chi_{\varrho}( \xi)| \le  |\xi|$. Taking $ \xi $ derivatives  of  $ \varrho \chi_\varrho $  gives a factor $ \varrho^{-1} $ together with a symbol supported where $\varrho \le | \xi | \le 2\varrho $.

Next, we have to prepare the linear operator $ P $ microlocally with respect to this partition of unity. We will then use microlocal pseudodifferential operators which may give complex solutions. But since $ P $ is a real PDO, we may take the real part of the solution to the linear equation. 
In the following we shall use the notation $ \w{D} =  (1 + |D|^2) ^{1/2}$.

\begin{prop}\label{normalform}
	Let $ P $ be given by~\eqref{psymbol}--\eqref{pcond} with $ u_1 = 0 $ and real $ v(x,w) \in C^\infty$  such that $ v(x_0, w_0) = 0 $, and
	let  $ \Gamma=  \Gamma_{\xi_0, \varepsilon}$ be defined by~\eqref{gammaform} for  $ | \xi_0 | = 1 $,  $0 <  \varepsilon \le \varepsilon_0 $. Then for $ \varepsilon_0 $ small enough there exists  real valued $  0  \ne a(v,x, w,\xi) \in S^0 $, $ 0 < c \w{\xi}^{-1} \le b(\xi)  \in S^{-1} $ and orthonormal  variables $ (t,x) \in \br \times \br^{n-1}$ so that 
	\begin{equation}\label{prepform}
	P(v,t,x, w, D)b(D ) =  	a(v,t,x, w, D)  Q(v,t,x, w,D) + R(v,t,x, w,D)
	\end{equation}
	where
	\begin{equation}\label{normform}
	Q(v,t,x, w,D)  = D_{t} + \sum_{j = 1}^{n-1} A_j(v,t,x, w,D_x) D_{x_j} + A_0(v,t,x, w, D_x)   
	\end{equation}
	Here $ a $, $A_j $ and  $ R $ are operators that depend $ C^\infty $ on $ v(x,w)$, $A_j  \in  C^\infty(\br,  S^0) $ is real valued when $ j > 0 $ and $ R  = R_0 + R_1 \in \Psi^1 $ where $ R_0 \in \Psi^{-1}$ and $ \wf R_1 \bigcap \Gamma_0 = \emptyset $,  $ \Gamma_0 = (0, x_0, w_0) \times \Gamma_{\xi_0, \varepsilon} $. The seminorms of $ A_{j} $, $ R $ and the constant~$ \varepsilon_0 $  only depend on the seminorms of  $ v $ and the coefficients of  $ P $.
\end{prop}

Observe that since $ a \ne 0 $ we have by the calculus that $a^{-1} a \cong a a^{-1} \cong \id   $ modulo $ \Psi^{-1} $.

\begin{proof}
	For the proof, it is important that compositions of  operators that depend $ C^\infty $ on $ v $ also depend $ C^\infty $ on $ v $, see Lemma~\ref{contlem}.
	By taking an ON base of eigenvectors of $ L(0,x_0, w_0) $ and choosing ON variables $ x $, we may assume that $ p_2(0,x_0, w_0, \xi ) = \sum_j {c_j}\xi_j^2 $ for $ 0 \ne c_j \in \br$. Choose $ j  $ so that $ \xi_j \ne 0 $ at $ \xi_0 $, thus in a conical neighborhood of $ \Gamma_0 $ if $ \varepsilon  $ is small enough. By an ON change of variables we may take $ j = 1 $, then $ c_1 = - L_{11}(0,x_0, w_0) $.  Letting  $ b (\xi)  =  \xi_1^{-1}$  near  $ \Gamma_0 \cap \set{|\xi | \ge 1}$ we may extend $ b(\xi)  $ to a symbol in $ S^{-1}$ so that $ b(\xi) \gtrsim \w{\xi}^{-1} $. Then $ P b(D) $  has a symbol expansion with $  p_jb  \in S^{j-1}  $ and
	$$  
	p_2(v,x, w, \xi) b({\xi})  = -\sum_{jk} L_{jk}(v,x, w)\xi_j B_k(\xi)  
	$$ 
	where  $ B_k(\xi)  = \xi_k b({\xi})  \in S^0$ . Since $B_j(D)D_{k} =   B_k(D)D_{j}$, we  find from~\eqref{psymbol}--\eqref{pquant} that
	\[ 
	P b(D) = \sum_{j = 1}^{m} A_j(v,x, w,D) D_{j} + A_0(v,x, w,D)
	\]
	where $ A_0 = (p_1\partial + p_0) b(D)$, $ A_1(v,x, w, \xi) = -L_{11}(v,x, w) B_1(\xi)  \ne 0$  near $ \Gamma_0 $ and  $ A_j \in S^0$  is real valued when $ j > 0 $. Observe that $ B_1(\xi) = 1$ near $ \Gamma_0 \bigcap \set{| \xi | \ge 1}$ so it may be extended to be equal to 1 everywhere modulo terms having wave front set outside $ \ol \Gamma_0 $. This gives  that $ A_1 = -L_{11} \ne 0 $ near $(0,x_0, w_0)$ thus we can extend $a = A_1  $  so that $0 \ne  a(v,x, w) \in C^\infty$. 
	Replacing $ A_j $ with $ a^{-1}A_j $  we find that $ Pb(D) = a(v,x, w) Q  \in \Psi^{1}$ where the symbol of $ Q $ is equal to $ \xi_1 + \sum_{ j > 1}A_j(v,x,w,\xi)\xi_j  + A_0(v,x,w,\xi)$ modulo $ \Psi^{-1} $ and terms having wave front set outside $ \ol \Gamma_0 $.
	
	To obtain that $ A_ j$ is independent of $ \xi_1$ for $ j > 0 $, we shall use the Malgrange preparation theorem. If  $ \xi = (\xi_1, \xi') $ we find by homogeneity for small enough $ \varepsilon_0 > 0 $ that
	\begin{equation}
	\xi_1 + \sum_{ j > 1}A_j(v,x,w,\xi)\xi_j = q(v,x,w,\xi)\left(\xi_1 + r(v,x,w,\xi') \right)	
	\end{equation} 
	in a conical neighborhood of $ \Gamma_0 $, where $q > 0$ is  homogeneous  and $ r  $ is real, homogeneous of degree 1 and vanishes when $ \xi' = 0 $. Then we can extend $ q > 0 $ to a homogeneous symbol by a cut-off, observe that the symbols depend $ C^\infty $ on $ v $.
	This replaces $ a$ by $ 0 \ne aq \in S^0$, and by using Taylor's formula we find that $r(v,x,w,\xi') = \sum_{j >1}r_j(v,x,w,\xi') \xi_j $ with $ r_j $ homogeneous in $ \xi' $.  
	This gives that $ Pb(D) =  aq(v,x,w,D) Q$   where $ Q $ is equal to~\eqref{normform} modulo $ \Psi^0 $ and terms having wave front set outside $ \ol \Gamma_0 $.
	The composition $ aq(v,x,w,D) $ with $ Q $ also gives lower order terms in $ \Psi^0 $ which can be included in $ A_0 $.

	Now the term $ A_0 \in  \Psi^0$ of  $ Q $ can be replaced by $  aq(v,x,w,D)R_0 $ modulo $ \Psi^{-1} $ where the symbol $ R_0 = A_0/aq \in S^0$. To  make the term $ R_0 $ independent of $ \xi_1$ we may use Malgrange division theorem and homogeneity for small enough $ \varepsilon_0 > 0$ to obtain that
	\begin{equation}
	R_0(v, x,w,\xi) = q_0(v,x,w,\xi)\left(\xi_1 + r(v, x,w,\xi') \right)	 + r_0(v, x,w,\xi')
	\end{equation}
	in a conical neighborhood of $ \Gamma_0 $, where $ q_0  \in S^{-1}$ and $ r_0 \in S^0 $ by homogeneity. Cutting of $ q_0 $ we may replace $ R_0(v,x,w,D) $ with $ q_0(v,x,w,D) Q + r_0(v,x,w,D_{x'}) $ modulo~$\Psi^{-1}$ near $ \Gamma_0 $. Since $ R_0 \cong (1 + q_0)R_0 $ modulo $\Psi^{-1}$, we obtain~\eqref{prepform}--\eqref{normform} with $ a $ replaced by $aq( 1 + q_0)$.
	Cutting off $ q_0 $ where $ |\xi | \gg 1 $ only changes the operator with terms in $ \Psi^{-\infty} $, but gives that $  1 + q_0  > 0$ making $ aq(1 + q_0 )\ne 0 $.  The composition of $ aq(v,x,w,D) $ with $ q_0(v,x,w,D) Q$ will also give lower order terms in $ \Psi^{-1} $ which can be included in $ R $ together with any cut-off terms. This gives the proposition  after putting $  t= x_1 $ and $ x = x' $.
\end{proof}

Proposition~\ref{normalform} shows that the linerarized equation $P (v,x,w,D)u = f  $ may after ON changes of variables be microlocally be reduced to the system
$Q_j (v,x, w,D)u_j   \cong a_j^{-1}(v,x,w,D) f_j $
where $f_{j} = \varphi_{j}(D) f$ with $  \varphi_{j}$ given by Remark~\ref{microrem}.  Observe that $ u \cong \sum_j b_j(D)u_j $ where $ u_j $ also has to be microlocalized.

Now, the reduction and the calculus will give terms  $ S \in \Psi^{-\infty} $ which have smooth kernels. The errors $ Sf (x) = \int S(x,y) f(y ) \, dy$ can be made small if $ f $ has support in a sufficiently small neighborhood of  $ x_0 $ by cutting off the kernel $ S $.  Let $ \phi_\delta(x) = \phi((x-x_0)/\delta) $ where $ 0 < \delta \le 1 $ and $\phi \in C^\infty_0(\br^n) $ such that $ 0 \le \phi \le 1 $, has support where $ | x| < 2 $ and is equal to 1 when $ | x| \le 1$, so that $\phi(x/\delta) \in C_0^\infty(B_{x_0, 2\delta}) $  if  $ B_{x_0, \delta} = \set{x:  |x - x_0 | \le \delta} $.

\begin{lem}\label{estlem}
	Let $ S(x,y) \in C^\infty $ and 
	$S_\delta(x,y) = \phi_\delta(x) S(x,y) \phi_\delta(y)  \in C^\infty_0(B_{x_0, 2\delta}\times  B_{x_0, 2\delta})$.
	The mapping $S_\delta:   C^\infty \mapsto  C_0^\infty(B_{x_0, 2\delta}) $ is given by $ S_\delta f(x) = \int S_\delta(x,y) f(y ) \, dy$, and for $ f \in  C^\infty_0(B_{x_0, \delta}) $ we have $S_\delta f(x) = Sf(x)$ when $ |x| \le \delta $.
	For $ \delta  $ small enough, $ \id + S_\delta $  has the inverse  $ (\id + S_\delta)^{-1}  = \sum_{j= 0}^\infty (-S_{\delta})^j \cong \id $
	modulo operators with kernels in $C^\infty_0(B_{x_0, 2\delta}\times  B_{x_0, 2\delta}) $. 
\end{lem}

\begin{proof}
	We may assume that $x_0 = 0  $, clearly $S_\delta f(x) = Sf(x)$ if $  \phi_\delta f= f$ and $ \phi_\delta(x) = 1 $.
	If $ f \in C^\infty $  then $ L^\infty $ norm is  $ \mn{S_\delta f}_\infty \le c_n 2^n\delta^{n }\mn{S}_\infty \mn{f}_\infty $.
	By induction we get 
	\begin{equation}
	\mn{S_{\delta}^{j}f}_\infty \le c_n2^n\delta^{n }\mn{S_\delta}_\infty \mn{S^{j-1}_{\delta}f}_\infty \\
	\le c_n^{j}2^{jn}\delta^{jn }\mn{S}^j_\infty \mn{f}_\infty \qquad j > 1
	\end{equation}
	where the kernels of $ S_{\delta}^{j} $ are in $ C^\infty_0(B_{0, 2\delta}\times  B_{0, 2\delta}) $. Thus the series $ \sum_{j= 0}^\infty (-S_{\delta})^j  $ converges in $ L^\infty $ if $ c_n 2^n\delta^{n }\mn{S}_\infty < 1 $. 
	Derivation of  the terms in the series will only give factors $ O(\delta^{-1}) $ so the convergence is in $C^\infty_0(B_{0, 2\delta}\times  B_{0, 2\delta}) $. 
	Then the  inverse $(\id + S_\delta)^{-1} =  \id + \sum_{j= 1}^\infty (-S_\delta)^j \cong \id  $ modulo operators with  kernels  in $C^\infty_0(B_{x_0, 2\delta}\times  B_{x_0, 2\delta}) $.
\end{proof}

Next, we shall solve the microlocalized equations $ Q_ju_j = a_j^{-1} f_{j} = a_j^{-1}  \varphi_{j}f  $ with $ u_j = 0$ when $ t = 0 $. Here  $  \varphi_{j}$ is given by Remark~\ref{microrem}, $ Q_j $ is given by~\eqref{normform} near $ \Gamma_{\xi_j, \varepsilon} $ with  $0 \ne  a_j \in S^{0} $ given by Proposition~\ref{normalform}, but we shall treat terms $R \in \Psi^{-1} $ as perturbations.  
In the case when $ A_j \equiv 0 $, we would then find that  $Q_j u_j  \cong D_{t} u_j   = a_j^{-1} f_{j} $, which has the approximate solution $ u_j  \cong  \int_0^{t}a_j^{-1} f_{j} \, dt $. By using Fourier integral operators one can reduce to this case. 

We shall denote by $ I^k $ the classical Fourier integral operators  of order $ k $ with homogenous phase functions and classical symbol expansions depending $ C^\infty $ on~$ v $ and $ w $.
But we shall also use operators $F \in C^\infty(\br, I^{k}) $ which are  FIO $F(t) \in I^k $ in  $  x$ depending $ C^\infty $ on~$ t $, $ v $ and~$ w $ for $ (t,x) \in \br \times \br^{n-1}$.  Observe that $ \Psi$DO of order $ k$ in $ x $ depending $ C^\infty $ on~$ t $, $ v $ and~$ w $ are  also in $C^\infty(\br, I^{k})$ and that $C^\infty(\br, I^{k}) \subset I^k$. By multiplying  $ I^k $ by $ I^m $ we obtain operators in $ I^{k + m} $ by Remark~\ref{fiorem}.

As before, we shall use ON coordinates  $ (t,x) \in \br \times \br^{n-1}$ and suppress the dependence on  $ v $ and $ w $. But the operators will depend $ C^\infty $  on  $ v $ and $ w $ having symbols and phase functions that are uniformly bounded  if $ v \in C^\infty$  and $ w \in \br^m$ are bounded.

\begin{prop}\label{normeq}
	Assume that $ Q = D_{t} + a_1(t,x,D_{x}) + a_0(t,x,D_{x})  $ depend $ C^\infty $ on $ v \in C^\infty$and $ w \in \br^m $,  where $ a_1 \in  C^\infty(\br,  S^1) $ is real and homogeneous of degree 1 and $ a_0 \in  C^\infty(\br,  S^0) $. Then there exist elliptic Fourier integral operators $ F_0(t) $ and  $ F_1(t) \in C^\infty(\br, I^{0}) $ such that $ F_0(t)  F_1(t) \cong \id$ and $ Q F_0(t)  \cong F_0(t)  D_t $  modulo $C^\infty(\br, I^{-1})$.
	If $ f \in C^\infty_0 $ then we have that
	\begin{equation}\label{soleq}
	u(t,x) =i F_0(t) \int_0^t F_1(s) f(s,x)\, ds  = \mathbb{ F}f(t,x) 
	\end{equation}
	solves the initial value problem
	\begin{equation}\label{normequation}
	Q u  = (\id + S)f  \in C^{\infty} \qquad u(0,x)\equiv 0 
	\end{equation}
	where $ S \in C^\infty(\br, I^{-1})$ and $ \mathbb{ F} \in I^0 $. 
	Here $ F_0(t) $, $ F_1(t) $ and $ \mathbb{ F} $  depend $ C^\infty $ on $ v\in C^\infty $ and $ w $, and have wave front sets close to the diagonal when $ |t| \ll 1 $. In fact, the canonical transformations given by $ F_0(t) $ and  $ F_1(t) $ maps bicharacteristics of $D_t + a_1 $ to $t $ lines and vice versa. 
\end{prop}

\begin{cor}\label{locrem}
	If $ c_j \in \Psi^{k} $, $ j = 1,\ 2 $, and $ \supp c_1 \bigcap  \supp c_2  =  \emptyset$, then  \/ $ c_1F_0(t)F_{1}(s)c_2 \in I^{-\infty}$
	having a smooth kernel  for small enough $ s $ and $ t $.
\end{cor}

\begin{proof}
	It is a classical result that there exists elliptic Fourier integral operators $ F_0(t) $ and  $ F_1(t) \in C^\infty(\br, I^0)$  with the properties in the proposition. 
	The construction of the homogeneous phase function of the FIO involves solving the Hamilton-Jacobi equations, which depend on the derivatives of the principal symbol $ \tau + a_1 $ of $ Q $ and thus depend $ C^\infty $ on $ v $ and $ w $. Then the amplitude is given by the transport equations which also depend on the lower order term $ a_0 $ of  $ Q $ modulo terms in $C^\infty(\br,S^{-1}) $ and thus depend $ C^\infty $ on $ v\in C^\infty $ and $ w $. 
	For the approximate inverse, one takes the phase function for the inverse canonical relation and the inverse amplitude, which also depend $ C^\infty $ on $ v $ and $ w $.
	
	If $ f \in C^\infty_0 $ and $ u = i F_0(t) v$ with $ v = \int_0^t F_1(s) f(s,x)\, ds $  as in~\eqref{soleq}, then we find 
	\begin{multline}
	Qu = iQ F_0(t)v = (iF_{0}(t)D_t + S_0 )v =  F_0(t)F_{1}(t) f  +  S_0 v  \\
	= f  + S_1f  + S_0 v = f  + S_1f +  S_0\int_0^t F_1(s) f \, ds = f + S f
	\end{multline}
	where $ S_0 $, $ S_1 $ and $ S \in C^\infty(\br, I^{-1})$.
\end{proof}

The approximate solution  $ u $ in~\eqref{soleq} depends $ C^\infty $ on the data $ f $ and  $ v $, but we shall need stronger estimates. For that we shall use the $ L^2 $ Sobolev norms:
\begin{equation}\label{normdef2}
\mn{\varphi}^2_{(k)} = \mn{\w{D}^k\varphi}^2 \qquad \varphi \in C^\infty_0  \qquad \forall \, k\in \br 
\end{equation}
Then the continuity of  $ a(u,x,D) \in \Psi^m $ depend uniformly on $ u\in C^\infty $ in theses spaces according to the following result,  were we shall suppress the parameter $ w $.

\begin{lem}\label{contlem}
	If  $ a(u,x, D) \in \Psi^{m_1}$ and $ b(u,x, D) \in \Psi^{m_2}$  depend  $   C^\infty  $ on $ u(x) $ then we find that  $a(u,x,D) b(u,x,D) \in \Psi^{m_1 + m_2}$ also depends  $   C^\infty  $ on $ u(x) $. 
	There exists $ \ell \in \br $ so that for  any $ a(u,x, D)  \in \Psi^{0}$ depending  $   C^\infty  $ on $ u(x) \in C^\infty $ and any $ k \in \br$  there exists $ C_{k}(t) \in C^\infty(\br_+) $ so that
	\begin{equation}\label{est1}
	\mn{a(u,x,D)\varphi}_{(k)} \le C_{k}(\mn{u}_{(\ell)})\mn{\varphi}_{(k)} \qquad \forall \, \varphi \in C^\infty_0
	\end{equation}
	There exists $ \ell \in \br $ so that for any  $ a(u,x, D) \in \Psi^{m_1}$ and $ b(u,x, D) \in \Psi^{m_2}$  depending  $   C^\infty  $ on $ u(x) $ and $ k  \in \br$ there exists $ C_{k}(t) \in C^\infty(\br_+) $ so that
	\begin{equation}\label{est2}
	\mn{\, [a(u,x,D), b(u,x,D)]\varphi}_{(k)} \le C_{k}(\mn{u}_{(\ell)})\mn{\varphi}_{(k +m_1+m_2-1)}  \qquad \forall \, \varphi \in C^\infty_0
	\end{equation}
	There exists $ \ell \in \br $ so that for  any $ a (u,x,D)  \in \Psi^{m}$ depending  $   C^\infty  $ on $ u(x) $ having real valued symbol modulo $ S^{m-1} $ and any $ k  \in \br$ there exists $ C_{k}(t) \in C^\infty(\br_+) $ so that
	\begin{equation}\label{est3}
	\mn{\im a(u,x,D)\varphi}_{(k)} \le C_{k}(\mn{u}_{(\ell)})\mn{\varphi}_{(k+m-1)} \qquad \forall \, \varphi \in C^\infty_0
	\end{equation}
	where $ 2i \im  a(u,x,D) =  a(u,x,D) -  a^*(u,x,D)  $ depends $ C^\infty $ on $ u $. Here the functions $ C_k(t) $ only depend on the seminorms of the symbols.
\end{lem}

\begin{proof}
	First we note that by definition any seminorm of order $ k \in \bn$ of $ a(u,x, \xi) \in  S^m $  is bounded  by $ \mn{u}_{C^k} $ for some $k \in \br $. By the Sobolev embedding theorem, the $ C^k $ norm of $ u $ can be bounded by the norm $ \mn{u}_{(k + s)} $ with $ s > n/2 $.
	
	If $ a(u,x,\xi)  \in S^{m_1}$ and $ b(u,x,\xi) \in S^{m_2}$ then $ a(u,x,D) b(u,x,D) = c(u,x,D)$ is given by
	\begin{equation}\label{compform}
	c(u,x, \xi) = e^{i\w{D_{ \xi}, D_{y}}}a(u,x,\xi)b(u,y, \eta)\restr{\substack{y = x \\ \eta =  \xi}}
	\end{equation}
	The mapping $ a, b  \mapsto c $ is weakly continuous on the symbol classes $ S^{m}$ so that any semi\-norm of $ c $ only depends on some seminorms of $ a $ and $ b $, see~\cite[Th.\ 18.4.10$ ' $]{ho:yellow}. 
	(Weak continuity means that the restriction to a bounded set is continuous in the $ C^\infty $ topology, see~\cite[Def.\ 18.4.9]{ho:yellow}.) 
	Thus if $a(u,x,D)  $ and $ b(u,x,D) $ depend  $   C^\infty  $ on $ u \in C^\infty $ then $ c(u,x,D) $ also does.
	Observe that the number of seminorms that is needed does not depend on the symbol classes $ S^{m_j} $, it only depends on the dimension. 
	
	We may reduce the estimate~\eqref{est1} to the case $ k = 0 $ by replacing $ a(u,x,D) $ with  $A(x,D) = \w{D}^{k}a(u,x,D)\w{D}^{-k} \in \Psi^{0} $ and $ \varphi $ with $ \w{D}^{k}\varphi $. Here the symbol expansion of 
	$$ A(x, \xi) \cong a(u(x),x,\xi) + k D_x a(u(x),x,\xi) \xi \w{\xi}^{-2} + \dots  + A_j(x, \xi) + \dots 
	$$  
	where the term $ A_j(x, \xi) \in S^{-j} $ has the factor $ D^j_x a(u(x),x,\xi) $, $ \forall \, j $. The error term for the expansion $ \sum\limits_{j = 0}^{n} A_j(x,D)$ is $ r_n(x,D) \in \Psi^{-n-1} $, which has a $ C^0 $ kernel which is bounded by $ \mn{u}_{C^{n+1}} $ which gives a bound of $r_n(x,D)  $ on $ L^2 $. The $ L^2 $ norm of $ A_j(x,D) $ is bounded by a  fixed seminorm of $ A_j(x,\xi) $, $ \forall \, j $, see~\cite[Th.\ 18.6.3]{ho:yellow}.  This seminorm in turn depends on a fixed seminorm of $ a(u,x,\xi) $ which gives~\eqref{est1}   for some $ C_k(t) \in C^\infty(\br_+)$. 
	
	Any seminorm of the symbol of the commutator $[a(u,x,D), b(u,x,D)]  \in S^{m_1+m_2 -1}$ depends on the same seminorm of the symbols of the compositions $ a(u,x,D)b(u,x,D) $ and $ b(u,x,D)a(u,x,D) $. These seminorms in turn depend on some  seminorms of $  a(u,x,\xi) $ and $  b(u,x,\xi) $. Thus we obtain~\eqref{est2} from~\eqref{est1} for some $ C_k(t) \in C^\infty(\br_+)$.
	
	If $ a   \in S^{m}$ then the adjoint $ a^*(u, x,D)  $ is given by 
	\begin{equation}
	a^*(u, x,\xi)   = e^{i\w{D_\xi, D_x}}\ol a(u,x,\xi)
	\end{equation}
	which is weakly continuous in the symbol class $ S^{m} $ by~\cite[Th.\ 18.1.7]{ho:yellow}. If $ a  $ is real modulo $ S^{m-1} $ then $\im a(u, x,D) \in \Psi^{m-1}  $. Thus any seminorm of the symbol of $ \im a(u, x,D) $ is bounded by some seminorms of $ a(u, x,\xi) $, giving ~\eqref{est3}  for some $ C_k(t) \in C^\infty(\br_+)$. 
\end{proof}

\begin{rem}\label{fiorem}
	The results of Lemma~\ref{contlem} also holds for $a(u,x,D) \in \Psi^m $ depending $ C^\infty $ on~$ u $ composed by $F \in I^k $ depending $ C^\infty $ on $ u $, e.g., the FIO given by Proposition~\ref{normeq}. Operators in $ I^{-\infty} $ have smooth kernels which are  $ C^\infty $ functions of  $ u $.
\end{rem}

In fact, Theorem 9.1 in \cite{ho:weyl} shows that the conjugation of $ \Psi$DO  with FIO gives symbol expansions similar to~\eqref{compform} after change of variables, see for example (9.2)$ ''  $  
in \cite{ho:weyl}. This result is about Weyl operators, but by Theorem 4.5 in \cite{ho:weyl} it can be extended to operators having the Kohn-Nirenberg quantization. This gives a calculus with symbol expansions of classical homogeneous FIO with homogeneous phases and symbols, see pages~441--442 in~\cite{ho:weyl}. For example, if $ a \in \Psi^{m}  $ and $ F \in I^k$  then we have $ \mn{aF u}^2 = \w{F^*a^* a Fu,u} $ where $ F^*a^* a F = b\in \Psi^{2(m+k)} $, and similar result holds for $ \mn{Fa u}^2$.

For $ S \in  I^{-\infty} $ the $ C^\infty $  dependence means that for any $ k \in \br$ we have  $ S \in  I^{-k} $ depending $ C^\infty $   on $ u $. Since the kernel is obtained by taking the Fourier transform in $ \xi $ of the symbol, we find that the kernel of $ S $ is smooth and is a $ C^\infty $  function of  $ u $.

Next, we are going to prove estimates for the microlocalized operators. 
Then we will use ON coordinates  $ (t,x) \in \br \times \br^{n-1}$ and  for $ k \in \br $ and $  T > 0$ define the local norms
\begin{equation}\label{normdef1}
\mn{\varphi}^2_{k,T} = \int_{|t| \le T} \mn{\varphi}_k^2(t) \, dt\qquad \varphi \in C^\infty_0
\end{equation}
and  $ \mn{\varphi}_{k,j,T} = \mn{ \w{D_t}^j \varphi}_{k,T}$, $ \forall \, j \in \br $, with $ \mn{\varphi}_k^2(t) =   \int |\w{D_x}^k \varphi(t,x)|^2\ dx$.

\begin{prop}\label{linest}
	Let $ v,\, f \in C^\infty_0 $ and $ u \in C^\infty $ be a solution to
	\begin{multline}\label{lineq}
	Q(v, t,x,w,D)u  = \partial_t u + \sum_{j=1}^n A_j (v, t,x,w,D_x) \partial_{x_j} u \\ + A_0(v, t,x,w,D_x) u  =  f \qquad u(0, x, w) = 0
	\end{multline}
	where $ A_j \in C^\infty(\br , \Psi^0)$  depends $ C^{\infty} $ on $v $, $ \forall\, j $, and $ A_j $ is real valued modulo $ S^{-1} $ for $  j > 0$. Then there exists $ \ell \in \bn $ so that for any $ k \in \bn$ there exists $ C_k(r) \in C^\infty(\br_+) $ so that
	\begin{equation}
	\mn{\phi u}_{(k)}^2 \le C_{k}(\mn{v}_{\ell })
	\mn{f}_{(k)}^2    
	\end{equation}
	if  $ \phi \in C^\infty_0 $ has support where $ |t| \le 1$.
	The estimate only depends on the seminorms of the symbol of\/ $Q  $ and $ \phi $.
\end{prop}

Thus, for any $ k \in \bn$ we get uniform local bounds on $ \mn{\phi u}_{(k)} $ when $\mn{ v}_{(\ell)} $ is uniformly bounded. Now $ Q $ is a differential operator in $ t $ but a $ \Psi $DO in $ x $, so in the proof we shall use Lemma~\ref{contlem} in the $ x $ variables.

\begin{proof}
	Let $ A\partial_x = \sum_{j=1}^n A_j  \partial_{x_j} $  and $ \w{u,u}_{k}(t ) =  \mn{u}_{k}^2(t )$ be the sesquilinear form,
	then 
	\begin{multline}
	\partial_t \mn{u}_{k}^2(t ) = 2 \re  \w{	\partial_t u,u}_{k}(t )	=  2 \re  \w{ f, u}_{k}(t ) \\ - 2 \re  \w{A \partial_t u,u}_{k}(t ) - 2 \re  \w{A_0 u,u}_{k}(t )
	\end{multline}
	Conjugating with $ e^{-Ct} $ gives 
	\begin{multline}
	\partial_t ( e^{-Ct} \mn{u}_{k}^2(t )) 	=  e^{-Ct}  \big(2 \re  \w{f, u}_{k}(t ) \\ - 2 \re  \w{A \partial_x u,u}_{k}(t ) - 2 \re  \w{A_0 u,u}_{k}(t )  - C \mn{u}_{k}^2(t )\big)
	\end{multline}
	where 
	\begin{multline}
	2 \re  \w{A \partial_x u,u}_{k}(t ) = 2 \re \w{[\w{ D_x}^k, A]\partial_x \w{ D_x}^{-k} w, w}_{0}(t ) \\ +  \w{[\re A,\partial_x, ]  w, w}_{0}(t ) + 2 \re \w{i  \im A\partial_x  w, w}_{0}(t )
	= \w{R_k w, w}_{0}(t)
	\end{multline}
	where $ [\re A, \partial_x] $ and $ \im A\partial_x \in C^\infty(\br,  \Psi^{0}) $ and $ w = \w{ D_x}^k u $. The calculus gives that the operator $ [\w{ D_x}^k, A]\partial_x \w{ D_x}^{-k}  \in C^\infty(\br,  \Psi^{0})$, so that  $ R_k\in C^\infty(\br,  \Psi^{0}) $ depends $ C^\infty $ on $ v $. 
	
	Since $ \mn{w}_{0} =  \mn{u}_{k} $  we find by 
	using Lemma~\ref{contlem} that
	\begin{equation}
	| \w{R_k w,w}_{0}(t )| \le  C_k(\mn{v}_{\ell }(t))\mn{w}_{0}^2(t ) = C_k(\mn{v}_{\ell }(t))\mn{u}_{k}^2(t )
	\end{equation}
	for some   $ \ell \in\bn $  and $ C_k(r) \in C^\infty(\br) $, and clearly
	\begin{equation}
	| \w{ f,u}_{k}(t )|   \le  \mn{ f}_{k}^2 (t )+ \mn{u}_{k}^2(t )
	\end{equation}
	We also obtain from Lemma~\ref{contlem} that
	\begin{equation}
	|\w{A_0 u,u}_{k}(t )| \le C_k(\mn{v}_{\ell }(t))\mn{u}_{k}^2(t )
	\end{equation}
	where in the following we will take the maximum of different  $ \ell $ and $ C_k(r) $.
	Summing up, there exists $ C_k(r)  $ so that for any $ C > 0 $ we have
	\begin{equation}
	\partial_t ( e^{-Ct} \mn{u}_{k}^2(t )) 	\le  e^{-Ct}\left((C_k(\mn{v}_{\ell }(t))  - C)\mn{u}_{k}^2(t )
	+ 	C_k(\mn{v}_{\ell }(t))\mn{ f}_{k}^2(t)\right) 
	\end{equation}
	Now we may replace $ C_k(r)  $ by a nondecreasing function
	and put 
	$$ C_k =  \max_{|t | \le 1} C_k(\mn{v}_{\ell }(t)) \le C_k\left(\max_{|t | \le 1} \mn{v}_{\ell }(t)\right) 
	$$ 
	where $ \mn{v}_{\ell }(t) \le  C_0\mn{v}_{( \ell + 1)}$,  $\forall \, t$, by Sobolev's inequality. 
	Since  $  \mn{u}_{k}(0 ) = 0$ we find by integrating that
	\begin{equation}
	e^{-C_kt} \mn{u}_{k}^2(t ) 	\le e^{C_k}C_k(\mn{v}_{(\ell + 1)})\mn{ f}_{k,1}^2   \qquad  t \in [-1, 1]
	\end{equation}
	Integrating over $ [-1, 1] $ we obtain that
	\begin{equation}
	\mn{u}_{k,1}^2	\le 2e^{2C_k}	C_k(\mn{v}_{( \ell + 1)})\mn{ f}_{k,1}^2
	\end{equation}
	where $ \mn{ f}_{k,1} \le \mn{ f}_{(k)} $.
	Renaming  $2 e^{2C_k}C_k(\mn{v}_{( \ell + 1)})  $ as $C_k(\mn{v}_{( \ell + 1)})  $  and changing $ \ell $ 
	we obtain 
	\begin{equation}\label{kest}
	\mn{u}_{k,1}^2	\le C_k(\mn{v}_{( \ell )}) \mn{ f}^2_{(k)}  \qquad \forall\, k \in \bn 
	\end{equation}

	Next, we shall estimate 
	\begin{equation}\label{kjest}
	\mn{u}_{k,j,1}^2	\le C_{k,j}(\mn{v}_{( \ell )}) \mn{ f}^2_{(k+j)}  \qquad \forall\, k \in \bn  
	\end{equation}
	when $j  \in \bn$. For that it suffices to estimate $ 	\mn{D_t^{j} u}_{k,1} $ 
	which we shall do by induction. The case $ j = 0 $ is given by~\eqref{kest}, and we assume that~\eqref{kjest} holds for $ i \le j $ for some $ j \in \bn$. Since $ Qu = \partial_t u + Au = f $ with $ A\in C^\infty(\br, \Psi^1) $, we have
	\begin{equation}
	\mn{D_t^{j+1} u}_{k,1} \le	\mn{ D_t^{j}f}_{k,1} + 	\mn{D_t^{j}A u}_{k,1} 
	\end{equation}
	where $D_t^{j}A = A D_t^{j} + \sum_{0 \le i < j} B_{i}D_t^{i} $
	with $  B_{i}(t,x ,D_x) \in C^\infty(\br, \Psi^1)$ being a $ \Psi $DO in $ x $ depending $ C^\infty $ on $ t $ and $ v $. By using Lemma~\ref{contlem} and integrating  over $ [-1, 1] $ we find that 
	\begin{equation}
	\mn{D_t^{j}A u}_{k,1} \le	\sum_{0 \le i \le j} C_{k,i}(\mn{v}_{( \ell )})	\mn{ u}_{k,i,1} 
	\end{equation}
	so the induction hypothesis gives~\eqref{kjest} for any $ j \in \bn$.
	
	Finally, we shall show that
	\begin{equation}\label{finalest}
	\mn{ \phi u}^2_{ (k)}	\le C_k( \mn{v}_{(\ell)}) \mn{ f}^2_{(k)} \qquad \forall\, k \in \bn
	\end{equation}
	if $\phi \in  C^\infty $ is supported where $ |t| \le 1 $.
	To estimate $ \mn{\phi u}_{ (k)}^2 $ it suffices to estimate $ \mn{D_x^\alpha D_t^j  \phi u} $ for $ | \alpha | + j \le k $.
	We have that 
	$$
	[D_x^\alpha D_t^j , \phi  ] =  \sum_{|\beta | + i \le k}B_{\beta, i}D_x^\beta D_t^i 
	$$ 
	where $  B_{\beta, i}\in  C^{\infty} $ has support where $ |t | \le 1 $. 
	Thus, \eqref{kjest} gives that 
	\begin{equation}\label{finallyest}
	\mn{D_x^\beta D_t^j\phi u}	\le C\sum_{0 \le  i \le j} \mn{u}_{k-i,i, 1}	\le   \sum_{0 \le  i \le j} C_{k-i,i}(\mn{v}_{(\ell)}) \mn{ f}^2_{(k)}
	\end{equation}
	which completes the proof. 
\end{proof}

Next, we shall solve the  IVP for the linearized equation
\begin{equation}\label{plinequation}
P(v(x, w),x, w,\partial) u(x, w) = f(x, w) 
\end{equation}
where $ f $ and $ v \in C^\infty(\br^{m+n}, \br^n) $  with $  v(x_0, w_0) = 0 $ and $ P $ is on the form~\eqref{psymbol} satisfying~\eqref{pquant} and~\eqref{pcond} with $ u_1 = 0 $. 
In the following, we shall suppress the parameters $ v$ and $ w $, the preparation will only depend on the bounds on these parameters.

To solve equation~\eqref{plinequation}, we shall assume that $ x_0 = 0 $ and use the microlocal normal forms given by Proposition~\ref{normalform}. In fact, for any small enough $ \varepsilon > 0 $  we can by Remark~\ref{microrem} find  a partition of unity $\{ \varphi_{j}(\xi) \}_j $ with $ \varphi_{j}\in S^0 $ supported in cones $ \Gamma_{\xi_j, \varepsilon} $  and ON variables $ (x_1,x')$ so that $  P b_j = a_j Q_j + Rj$ satisfies the conditions in Proposition~\ref{normalform} with $ \Gamma_0 = (0,0,w_0) \times \Gamma_{\xi_j, \varepsilon} $ after the change of variables. 
Here  $ 0 \ne  a_j \in S^0$,  $\w{\xi}^{-1} \ls b_j(\xi) \in S^{-1}$  and 
$Q_j = D_{x_1} + A_jD_{x'} + A_{0,j} $
satisfies the conditions in Proposition~\ref{normeq}. The operator $ R_j \in \Psi^1 $ has symbol in $S^{-1}$  in a conical neighborhood of $\Gamma_0 $.
Ignoring the operator $ R_j $, which will be handled as a perturbation,  we obtain from Proposi\-tion~\ref{normeq}
that if  $ f \in C^\infty $ then $ u_j = \mathbb{ F}_ja_j^{-1}\varphi_{j}f $ solves
\begin{equation}\label{microeq0}
Q_j u_j = (\id + s_j) a_j^{-1}\varphi_{j}f = (a_j^{-1} + r_j)\varphi_{j}f  \qquad u_j \restr{t = 0} = 0
\end{equation}
where $ s_j $ and $ r_j \in I^{-1} $.

But $ u_j$ may not be localized near $\Gamma_0  $.
To handle the localization and the error term $ R_j $ we
shall microlocalize $ u_j  $ depending on parameters. We shall use  the cut-off $ \psi_{j,\varrho}(\xi)  = \psi_{j}(\xi)\chi_{\varrho}(\xi)$ given  by Remark~\ref{microrem} with $ \varrho \ge 1 $ such that $0 \le  \chi_{\varrho} \le 1$ has support where $ |\xi | \ge \varrho $,  $ \psi_j \varphi_{j} = \varphi_j $ and $ \supp \psi_j \in \Gamma_{\xi_j, \varepsilon} $. We shall make cut-offs with $\Phi(x) \in C^\infty_0(\br^n) $ such that $ 0 \le \Phi \le 1 $, $  \Phi $  has support where $ | x| \le 1 $ and is equal to 1 when $ | x| \le 1/2 $. 
We find that
\begin{equation}\label{microloc}
u_j =  \psi_{j,\varrho}\Phi u_j  + (1- \psi_{j,\varrho})\Phi  u_j  + (1- \Phi)u_j=  u_{j,\varrho}  + S_{j,\varrho} f
\end{equation}
when $ | x| < 1/2 $, where $u_{j,\varrho} =    \psi_{j,\varrho}\Phi u_j$ and $ S_{j,\varrho} = (1- \psi_{j,\varrho})\Phi  \mathbb{F}_ja_j^{-1} \varphi_{j}   \in I^0$  since $ \mathbb{F}_j a_j^{-1} \varphi_{j} \in I^0$. Thus we find that $ S_{j,\varrho}f(x) $ depends on the values of $ f(y) $ when $y_1 $ is in the interval between 0 and $ x_1 $.
Since  $ \psi_j \varphi_{j} = \varphi_j $ we find that $  (1- \psi_{j,\varrho}) \varphi_{j} = (1- \chi_{\varrho}) \varphi_{j}  $ which is supported where $ | \xi | \ls \varrho $.

By~\eqref{normequation} we also have 
\begin{multline*}
Q_jS_{j,\varrho} =   [\psi_{j,\varrho}, Q_j ]  \Phi\mathbb{F}_ja_j^{-1}\varphi_{j}+ (1- \psi_{j,\varrho})[Q_j, \Phi ] \mathbb{F}_ja_j^{-1}\varphi_{j} \\ +  (1- \psi_{j,\varrho})\Phi(\id + S_j)a_j^{-1} \varphi_{j}     \in I^{-\infty}	
\end{multline*}
and also that $Q_jS_{j,\varrho} \in I^0 $ uniformly with symbol supported where $| \xi | \ls \varrho $ modulo  $ S^{-\infty}$, and $ S_j \in I^{-1} $. 
Since $\Psi^0 \ni a_j \ne 0 $ we have $ a_j a_j^{-1}  = \id  + B_j$ with $ B_j \in \Psi^{-1} $. 
We find from~\eqref{prepform}, \eqref{microeq0} and~\eqref{microloc} that
\begin{equation}\label{linsol0}
Pb_j u_{j,\varrho} = a_jQ_j (u_{j}  - S_{j,\varrho} f)  + R_j u_{j,\varrho} \\
= \left((\id + B_j + a_jr_j)\varphi_j  - a_j Q_jS_{j,\varrho} +  R_{j,\varrho} \right) f  
\end{equation}
when $ | x| < 1/2 $, where $ a_j Q_jS_{j,\varrho} \in I^{-\infty}$, $ a_jr_j \in I^{-1} $ and $ R_{j,\varrho} =  R_j\psi_{j,\varrho} \mathbb{F}_ja_j^{-1} \varphi_{j}\in I^{-1}$ with symbol supported where $| \xi | \gtrsim \varrho $ modulo  $ S^{-\infty}$ when $|x | \ll 1 $ since  the symbol $ R_j \psi_{j,\varrho} \in S^{-1} $ for $| x| \ll 1 $ by Proposition~\ref{normalform}. 
Here $ E \in I^k $ in the open set $ \Omega \subset T^*\br^n $ means that $ E = E_0 + E_1 $ where $ E_1 \in I^k $ and $ \wf E_0 \bigcap \Omega= \emptyset $.
Since $ \varrho \psi_{j,\varrho} \in S^1 $ uniformly by Remark~\ref{microrem}, we find $ \varrho R_{j,\varrho} \in I^0$  uniformly when $| x| \ll 1$.

Now we define
\begin{equation}\label{linsol}
u_{\varrho}(x) = \sum_{j} b_j(D)\psi_{j,\varrho}(D) \Phi  u_j(x)  = \sum_{j} b_j(D)u_{j,\varrho}(x) 
\end{equation}
where $\varrho b_j\psi_{j,\varrho} \in S^0 $ uniformly when $ \varrho \ge 1 $.
We obtain from~\eqref{linsol0} and~\eqref{linsol} that
\begin{equation}\label{sumsys}
P u_{\varrho} =  
f +  \sum_j \left((B_j  + a_jr_j)\varphi_{j} + R_{j,\varrho} - a_jQ_jS_{j,\varrho}\right)f
= (\id+ R_{\varrho}) f
\end{equation}
where $ R_{\varrho} =  \sum_j (B_j + a_j r_j)\varphi_j  + R_ {j,\varrho}- a_jQ_jS_{j,\varrho} \in I^{-1}$ when $ |x|  \ll 1 $. We shall localize the first terms in $ \xi $ by writing $ B_j + a_j r_j =  (B_j + a_j r_j)(1 - \chi_{\varrho}) +  {(B_j + a_j r_j)\chi_\varrho} $ which gives 
\begin{equation}\label{RrhoTdef}
R_{\varrho } = 
R_{\varrho,0}  +  R_{\varrho, 1} 
\end{equation}
with $ R_{\varrho,0} =  \sum_j (B_j + a_jr_j )\varphi_j(1 - \chi_{\varrho})  - a_jQ_jS_{j,\varrho} \in  I^0  $ uniformly with symbol supported where $ |\xi | \ls \varrho $ modulo $ S^{-\infty}$ uniformly when $ \varrho\ge 1 $,  and $  R_{\varrho,1} =  \sum_j (B_j + a_jr_j )\varphi_{j,\varrho} +   R_ {j,\varrho}\in I^{-1}$ uniformly  and $ \varrho R_{\varrho, 1}  \in I^0$ uniformly when $ \varrho \ge 1 $  and $ |x| \ll 1 $ since $ \varrho \chi_{\varrho} \in S^1 $ uniformly.

Since we are only need local solutions, we may cut off near $ x= 0$.
Let  $ \Phi_\delta(x) = \Phi(x/\delta) $ with $ 0 < \delta \le 1 $, 
To solve the equation near $ x=0 $ it is enough that $ \Phi_\delta Pu_{\varrho} =  \Phi_\delta f $ for some $0 < \delta  \le 1$. 
If  $ f$ has support where $|x| \le \delta/2 $  then 
we obtain from~\eqref{linsol0} that $  \Phi_\delta Pu_{\varrho} =  \Phi_\delta(\id  + R_{\varrho})\Phi_{\delta} f = (\id +  R_{\delta, \varrho}) f$ where $R_{\delta, \varrho} =  \Phi_\delta R_{\varrho}\Phi_{\delta} $. 
By~\eqref{RrhoTdef} we have 
\begin{equation}\label{rdrt}
R_{\delta, \varrho} =   R_{\delta, \varrho, 0} + R_{\delta, \varrho, 1}
\end{equation}
with  $  R_{\delta, \varrho, j} = \Phi_\delta R_{\varrho,j}\Phi_{\delta}$. 
For fixed $0 < \delta \ll 1$ we have  $ \varrho R_{\delta, \varrho, 1} \in I^0$ uniformly when $ \varrho\ge 1 $ and  $R_{\delta, \varrho, 0} \in I^{0}$  has $ C^{\infty} $ kernel depending on  $ \varrho $  and  $ \delta $ with symbol supported where $ |\xi | \ls \varrho $ modulo $  S^{-\infty}$ uniformly when $ \varrho\ge 1 $.

It remains to invert the term $ \id +  R_{\delta, \varrho} $ in order solve equation~\eqref{plinequation}. 
This will be done in two steps, first making $R_{\delta, \varrho, 1}  $ small by taking large enough $ \varrho $. This may increase the seminorms of $ R_{\delta, \varrho, 0} $, but this term can then be made small by localizing in a sufficiently small neighborhood of  $ x= 0 $.

Since $ \varrho R_{\delta, \varrho, 1} \in I^0$ uniformly when $ \varrho\ge 1 $, we find by Remark~\ref{fiorem} that there exists $\varrho_{\delta}  \ge 1  $  so that if $ \varrho \ge \varrho_{\delta} $ we have  
$ \mn{R_{\delta, \varrho, 1} f}_{(0)} \le \mn{ f}_{(0)}/2 $  for $ f \in \Cal S $.
Then we find that 
$$ (\id+ R_{\delta, \varrho, 1})^{-1}  = \id + \sum_{ k > 0}(-R_{\delta, \varrho, 1})^k \in I^0 \  \text{ uniformly in $ \varrho $}
$$ 
Observe that $ (-R_{\delta, \varrho, 1})^k $ has kernel supported where $ |x | \le \delta $ and $ |y | \le \delta $.  If we then solve 
\begin{equation}\label{microeq1}
Q_j u_j = a_j^{-1} \varphi_{j}(\id+ R_{\delta, \varrho,1})^{-1}f  \qquad u_j \restr{x_1 = 0} = 0
\end{equation}
for $  f $ supported where $ |x| \le \delta/2 $, then the previous reduction and~\eqref{rdrt} gives
\begin{equation}\label{sumsys1}
\Phi_\delta P u_{\varrho} =  (\id+ R_{\delta, \varrho})(\id+ R_{\delta, \varrho,1})^{-1}f = (\id + R_{\delta,\varrho,2}) f 
\end{equation}
where $ R_{\delta,\varrho,2} = R_{\delta,\varrho, 0}(\id + R_{\delta,\varrho,1})^{-1} \in I^{-\infty}$ with $ C^\infty $ kernel supported where $ |x | \le \delta $ and $ |y | \le \delta $ and with symbol supported where $ |\xi | \ls \varrho $ modulo $ S^{-\infty}$ uniformly when $ \varrho\ge 1 $.  Observe  that we have uniform bounds for  fixed $ \delta $ when  $ \varrho \ge \varrho_{\delta} $ and these bounds depend on the bounds on the symbol of $ P $ and the parameters $ v \in C^\infty $ and~$ w $. We shall later put more restraints on the lower bound of $ \varrho $ because of conditions on the estimates, see~\eqref{indest}, and the values of $ u_{\varrho}(0)$ and $ \partial u_{\varrho}(0)$, see~\eqref{initrho}.

Now we have to shrink the support of $R_{\delta,\varrho,2}  $ to lower the norm of the kernel without changing $R_{\delta,\varrho,1}  $. With fixed $0 <  \delta \le 1$ and $ \varrho_{\delta} \ge 1$,  we assume $ \varrho \ge \varrho_{\delta} $ and multiply the equation~\eqref{sumsys1} with $ \Phi_{\delta_0}  $ with $ 0 < \delta_0 \le \delta/2 < 1/2$ so that $  \Phi_{\delta} = 1 $ on $ \supp  \Phi_{\delta_0} $. If  $ f  $ is  supported where $ |x| \le \delta_0/2 < 1/4$, then we obtain as before that
\begin{equation}\label{sumsys2}
\Phi_{\delta_0}  P u_{\varrho} =  \left(\id + R_{\delta_0,\varrho,2 }\right)f 
\end{equation}
where $R_{\delta_0,\varrho,2} =  \Phi_{\delta_0} R_{\delta,\varrho,2} \Phi_{\delta_0} $.

\begin{lem}\label{estlem2}
   There exists $C_0 > 0  $ so that if  $ \varrho \ge \varrho_{\delta} $ and $ 0 < \delta_0 \le \delta/2 < 1/2$ we have
   \begin{equation}
   	\mn{R_{\delta_0,\varrho,2}\varphi}_{(0) }\le C_0 \varrho^n \delta_0^n  \mn{\varphi}_{(0) } \qquad \forall \,  \varphi \in \Cal S
   \end{equation}
\end{lem}

\begin{proof}
	 Now $ (\id + R_{\delta,\varrho,1})^{-1}\in I^{0} $ uniformly in $ \varrho \ge \varrho_{\delta} $ so $ R_{\delta,\varrho,2} = R_{\delta,\varrho, 0}(\id + R_{\delta,\varrho,1})^{-1}\in I^{-\infty}$  having $ C^\infty $ kernel but also in $ I^0 $ uniformly in $ \varrho \ge 1 $ with bounded symbol supported where $ |\xi | \ls \varrho $ modulo $  S^{-\infty}$ uniformly when $ \varrho\ge 1 $.
	 
	 By integrating the symbol in $ \xi $, we find that $ R_{\delta,\varrho,2} $ has $ C^\infty $ kernel $ R_2 (x,y)$ so that $ \mn{R_2}_\infty \le C\varrho^n $.
	 Thus  $ R_{\delta_0,\varrho, 0}$  has kernel  $R_0 = \Phi_{\delta_0}(x) R_{2}(x,y) \Phi_{\delta_0}(y) \in I^{-\infty} $ and since 
	 $$ 
	 \mn{R_{\delta_0,\varrho, 2}f }^2_{(0)} \le \int \left| \int R_0(x,y) f(y)\, dy \right |^2 \, dx  \le \mn{R_0}^2_{(0)} \mn{f}^2_{(0)} 
	 $$
	 where  $ \mn{R_0}^2_{(0)} \ls \varrho^{2n}(\int \Phi_{\delta_0}(y)\, dy )^2  \ls \varrho^{2n}\delta_0 ^{2n}$, which proves the  lemma.
\end{proof}

By Lemma~\ref{estlem2} we find that if $ c_0 = (2C_0)^{-1/n}$ then for $0 < \delta_0 \le \min(c_0/\varrho, \delta/2) $  we have that $ \mn{R_{\delta_0,\varrho,2}(x,D)\varphi}_{(0) } \le \mn{\varphi}_{(0) }/2$ for $ \varphi \in \Cal S $. We obtain that
$$ 
(\id + R_{\delta_0,\varrho,2}(x,D) )^{-1} = \sum_{j \ge 0}(-R_{\delta_0,\varrho,2}(x,D) )^j \cong \id 
$$ 
modulo  an operator in $ I^{-\infty}$ with $ C^\infty_0 $ kernel supported where $ |x | \le \delta_0 $ and $ |y | \le \delta_0 $. 
By replacing $ f $ in~\eqref{microeq1} by $  (\id + R_{\delta_0,\varrho,2})^{-1} f $ we obtain the first part of the following result.

\begin{prop}\label{initlem}
	Let $ P $ be given by Theorem~\ref{appthm}, $ \varphi_{j} $ be given by Remark~\ref{microrem}, $P b_j  = a_jQ_j  + R_j $ by Proposition~\ref{normalform} with $ x_0 = 0 $ and let $R_{\delta, \varrho,1}  $ and $  R_{\delta_0,\varrho,2}$ be given by~\eqref{rdrt} and~\eqref{sumsys2}  depending $ C^\infty $ on $ v $ and~$ w $. Then there exist  $0 <  \delta \le 1 $, $  \varrho_{\delta} \ge 1$ and $ c_0 > 0 $ so that if  $ \varrho \ge \varrho_{\delta} $ and  $0 <  \delta_0  \le \min( c_0/\varrho, \delta/2) $, then $(\id+ R_{\delta, \varrho, 1})^{-1}(\id + R_{\delta_0,\varrho,2})^{-1} = \id + R_{\varrho,\delta_0}  \in I^0$ is uniformly bounded.
	If	$ f \in C^\infty $ and  $ u_j  \in C^\infty $ solves
	\begin{equation}\label{microeq2}
	Q_j u_j = a_j^{-1} \varphi_j (\id+ R_{\varrho,\delta_0}) \Phi_{\delta_0 }f \qquad u_j \restr{x_1 = 0} = 0 \quad \forall \ j
	\end{equation} 
	then after ON changes of $ x $ variables we find that $ u_{\varrho}(x) = \sum_{j} b_j\varphi_{j,\varrho} \Phi u_j(x) $ solves
	\begin{equation}
	P u_{\varrho} = f
	\end{equation}
	when $ |x| \le \delta_0/2 $ and $ |w - w_0| \ll 1 $. We also have that
	\begin{equation}
	u_{\varrho}(x) = c_{x, \varrho}(f) \quad \text{and} \quad	\partial u_{\varrho}(x) = d_{x, \varrho}(f)
	\end{equation}
	where $ \varrho c_{x, \varrho} $ and $ d_{x, \varrho} \in \Cal E' $  are $ o(1)$ uniformly for $ \varrho\to \infty $ when $ | x | \le \delta_0/2 $ and $ |w - w_0| \ll 1 $, i.e.,  $\varrho c_{x, \varrho}(f)  $ and $d_{x, \varrho}(f) = o(1) $  as $ \varrho\to \infty $. 
\end{prop}

\begin{proof}
	In only remains to prove the statement about the values at $ x $  when $ | x | \le \delta_0/2 $. Since $ u_{\varrho} = \sum_{j} b_ju_{j,\varrho}$ it suffices to consider the terms $ b_ju_{j,\varrho}  = E_{j,\varrho}\Phi u_j$, where  $E_{j,\varrho} = b_j\psi_{j,\varrho} \in I^{-1}$  uniformly and has symbol $ e_{j,\varrho}(x,\xi) \in S^{-1}  $ supported where $  |\xi | \ge \varrho  $. Then $ \varrho E_{j,\varrho} $ and $ \partial_x E_{j,\varrho}   \in I^0$ uniformly in $ \varrho \ge 1 $. By Proposition~\ref{normeq} we have that $ u_j(x) = F_{j,\varrho}f(x) $ depends linearly on $ f $, where $F_{j,\varrho} =  \mathbb{F}_j a_j^{-1} \varphi_{j} (\id+ R_{\varrho,\delta_0})  \Phi_{\delta_0 } \in I^0 $ uniformly when  $ \varrho \ge \varrho_{\delta} $. 
	We only have to prove the continuity, since the factor $ \Phi_{\delta_0 } $ in $ F_{j,\varrho} $ gives compact support. 
	 
    We have that
\[ 
\varrho E_{j,\varrho}   u(x) = (2\pi)^{-n} \int e^{i(\w{x,\xi} + \phi_j(x,\xi) )} \varrho e_{j,\varrho}(x,\xi)\wh u(\xi)\, d\xi \qquad \forall\  u \in  \Cal S  \quad \forall\, x    
\]
where $\phi_j(x,\xi)  $ is real and homogeneous of degree $ 1 $ in $ \xi $ and $\varrho e_{j,\varrho} \in S^0  $ uniformly. Thus, we find that
\begin{equation}\label{FL1est}
|E_{j,\varrho}   u(x)| \ls \varrho^{-1} \int_{| \xi | \ge \varrho}  |\wh{u}(\xi)| \, d\xi \le o( \varrho^{-1})\mn{\wh u}_{L^1}  \qquad \forall\  u \in \Cal S  \quad \forall\, x    
\end{equation}
by dominated convergence as $ \varrho \to \infty $, depending only on the seminorms of   $e_{j,\varrho}$. This estimate also gives $ \partial E_{j,\varrho}   u(x)  = o(1)$  since $ \partial_x E_{j,\varrho}   \in I^0$ uniformly in $ \varrho \ge 1 $ with symbol supported where $  |\xi | \ge \varrho  $.
Here $\mn{\wh u}_{L^1} =  \mn{u}_{FL^1} $ is the Fourier $ L^1 $ norm,  so there exists $ C_n > 0 $ so that $ \mn{u}_{FL^1}  \le C_n \mn{u}_{\left(\frac{n + 1}{2}\right)}$ for any $ u \in \Cal S   $. 
Thus, we find that $	| E_{j, \varrho}\Phi u_j(x) |  \ls o(\varrho^{-1}) \mn{ \Phi u_j}_{\left(\frac{n + 1}{2}\right)} $ and 
$| \partial E_{j, \varrho}\Phi u_j(x) | \le o(1)\mn{\Phi u_j}_{\left(\frac{n + 1}{2}\right)}$, $ \varrho \to \infty $,  
when $  | x | \le \delta_0/2  $. 
Now $ u_j(x) = F_{j,\varrho}f(x) $, where $ F_{j,\varrho} \in I^0  $ uniformly when $ \varrho \ge \varrho_{\delta} $, so 
by Lemma~\ref{contlem} we find that 
$$ 
\mn{\Phi u_j}_{\left(\frac{n + 1}{2}\right)} \le C_0 \mn{ u_j}_{\left(\frac{n + 1}{2}\right)} \le C_1  \mn{f}_{\left(\frac{n + 1}{2}\right)}\qquad \text{when  $ \varrho \ge \varrho_{\delta} $}
$$  
which finishes the proof of Proposition~\ref{initlem}.
    \end{proof}

\begin{proof}[Proof of Theorem~\ref{appthm}] 
	To solve~\eqref{pequation} we may first assume $ x_0 = 0 $ and make the reduction~\eqref{reduc} to the case with vanishing data. Then we find that $ f $ is replaced by $ f_0 = f - (p_1 + p_0\cdot x)u_1  - p_0 u_0 $,  $P_0 =  P $ is given by~\eqref{pmodequation} depending on  $ u_1 $ and~\eqref{pcond} holds with $ u_1 = 0 $. 
	Since we only need local solutions, we shall cut off $\partial v $ with   $ \Phi \in C^\infty_0$ that is supported where $ | x  | \le 1 $,  $0 \le  \Phi \le 1$ and  $ \Phi(x)= 1 $ when $ |x | \le 1/2 $.
	Starting  with  $ v^0 \equiv 0 $  we shall solve the linearized equation
	\begin{equation}\label{plinequation0}
	P_0(\Phi \partial v^j(x, w),x, w,\partial) v^{j+1}(x, w) = f_0(x, w)  \qquad \forall\, j \ge 0
	\end{equation}
	near $ x = 0 $ and $ w = w_0 $ with real valued solution $ v^{j+1} $ such that $ \varrho v^{j+1}(x, w_0) $ and  $\partial_{x}v^{j+1}(x, w_0)$  are  $o(1)$ uniformly when $ | x | \le \delta_0/2 $ and $ \varrho \to \infty $  depending linearly on $ f_0 $. Thus, for large enough $ \varrho $ we find that $ P_0(\Phi \partial v^{j+1}(x,w),x, w,\partial) $ is of real principal type  when $ | x | \le \delta_0/2 $ and $ | w - w_0 | \le c$ for some $ c > 0 $.

   	To  microlocalize, we use Propositions~\ref{normalform} and~\ref{initlem} to find  $  \varrho_{\delta} \ge 1$ and $ c_0 > 0 $ so so that if $ \varrho \ge  \varrho_{\delta} $, $0 < \delta_0 \le \min (c_0/\varrho, \delta/2)  $, we find that~\eqref{plinequation0} reduces  near $ x= 0 $ to the coupled system of equations given by~\eqref{microeq2} for $ j \ge 0 $: 
	\begin{equation}\label{plinsystem}
	Q_k( \Phi  \partial v^{j}(x), x,D) v_k^{j+1}  = a_k^{-1} \varphi_k (\id + R_{\varrho, \delta_0})\Phi_{\delta_0}f_0  \quad  \text{$ \forall \,k $}
	\end{equation}
	with $ v^0 \equiv 0 $. Here $ R_{\varrho, \delta_0} \in I^0 $ uniformly, $ v^{j}(x)= \re \sum_{k= 1}^N b_k\psi_{k, \varrho}(D) \Phi  v^j_k(x) \in  C^\infty$ for $ j > 0 $ with  $\varrho b_k\psi_{k, \varrho} \in S^{0} $ uniformly when $ \varrho \ge 1 $.
	Observe that $ a_k^{-1} $ and $R_{ \varrho,\delta_0}   \in I^0$ uniformly  depending $ C^\infty $ on $ \Phi  \partial  v^{j}( x)$ and $ w $.
	
	Now $ \Phi_{\delta_0 }(x) = 1$  when $ |x| < \delta_0/2 $, so by solving~\eqref{plinsystem} using Proposition~\ref{normeq} we will obtain a solution to~\eqref{plinequation0} when $ |x| < \delta_0/2 $  such that  $ \varrho v^{j+1}_k(x, w_0)  $ and $ \partial_{x} v^{j+1}_k (x, w_0) $ are  distributions  of $ f_0  $ which are $  o(1)$ when $ |x | \le \delta_0/2 $ and $ \varrho \to \infty $.
	We are going to prove that the solutions $ v^{j+1} $ to~\eqref{plinequation0} are uniformly bounded in $  C^\infty $ near $ x = 0 $, then we can use the Arzela-Ascoli theorem to get convergence of  a subsequence to a solution to the nonlinear equation~\eqref{pmodequation}.
	Observe that we may solve~\eqref{plinsystem} in any neighborhood of $ x= 0 $, but this system will only give a solution to~\eqref{plinequation0} when  $ | x | < \delta_0/2$ where $ P_0 $ is of principal type.
	
	First we obtain from Lemmas~\ref{contlem} and Remark~\ref{fiorem} that  there exists $ \ell \in \bn $ so that for any $ m \in \bn $ there exists $ C_m(t) \in C^\infty(\br_+)$ so that 
	\begin{equation}\label{auxeq}
	\mn{a_k^{-1} \varphi_k  (\id+ R_{\varrho,\delta_0})\Phi_{\delta_0}f_0 }^2_{(m)} \le C_m(\mn{ \Phi \partial  v^j}_{(\ell)})\mn{f_0}^2_{(m)}
	\end{equation}
	since the operators are in $ I^0 $ uniformly depending $ C^\infty $ on $  \Phi  \partial v^j (x)$.

	By~\eqref{plinsystem}, \eqref{auxeq} and Proposition~\ref{linest} we also find that
	there exists $ \ell \in \bn $ such that for any $ m \in \bn $ there exists $ C_m(t) \in C^\infty(\br_+)$ so that 
	\begin{equation}\label{recureq}
	\mn{\Phi  v_k^{j+1}}_{(m)}^2 \le C_m(\mn{ \Phi  \partial v^j}_{(\ell)})
	\mn{f_0}_{(m)}^2  \qquad \forall \, k
	\end{equation}
	since $ \Phi   \in C^\infty_0 $ has support when $  | x |  \le 1$.
	Here we may assume that $ C_m(t) $ is a nondecreasing function $\forall \, m $.
	Now  $ \varrho \partial b_k \psi_{k, \varrho} \in \Psi^1 $ uniformly when $ \varrho \ge 1$ and  $ \mn{\re u}_{(m)} \le  \mn{ u}_{(m)} $ for $ u \in \Cal S $, which gives
	\begin{equation}\label{stripest}
	\mn { \Phi \partial v^j}^2_{(m)} \le \varrho^{-2}\wt C_m   \sum_{ 1\le k \le N } \mn{  \Phi  v^j_k}^2_{(m+1)}  \qquad \forall \, m \in \bn \quad \forall \, \varrho \ge 1
	\end{equation}
	By using~\eqref{stripest} for $ m = \ell  $ and changing $ \ell $ we find for any $ k $ that
	\begin{equation}\label{indest}
	\mn{\Phi  v_k^{j+1}}_{(m)}^2\le  C_m\left(\varrho^{-1} \sqrt{\wt C_\ell \sum_{1 \le k \le N} \mn{\Phi  v^j_k}^2_{(\ell)}} \right) \mn{f_0}_{(m)}^2 \qquad \forall \, m\in \bn \quad \forall \, \varrho \ge 1
	\end{equation}

	Thus, for any $ m \in \bn$ we will obtain uniform bounds on $   \mn{\Phi  v_k^{j}}_{(m)}$  if we have uniform bounds when $m =  \ell $. Since $ v^{0} = 0,\  \forall\, k,$ we find by taking $ m = \ell $ in \eqref{recureq} that 
	\begin{equation}\label{est0}
	\mn{\Phi  v_k^{0}}_{(\ell)}^2\le C_\ell(0) \mn{f_0}_{(\ell)}^2 \qquad \forall \, k
	\end{equation}
	where $C_\ell(0) \le C_\ell(1)  $ since $ C_\ell(t) $ is nondecreasing. If we assume for some $ j \ge 0$ that
	\begin{equation}\label{inestj}
	\mn{\Phi  v_k^{j}}_{(\ell)}^2\le C_\ell(1) \mn{f_0}_{(\ell)}^2 \qquad \forall \, k
	\end{equation}
	then by choosing $  \varrho \ge \sqrt{N\wt C_\ell C_\ell(1)} \mn{f_0}_{(\ell)} = \varrho_0$  we find using~\eqref{indest} with $ m = \ell $   that \eqref{inestj} holds with $ j $ replaced by $ j +1 $. Since this is true for $ j = 0 $ we obtain by induction that \eqref{inestj} holds  for any $ j $. By~\eqref{indest} we obtain for any $ m $  uniform bounds on $  \mn{\Phi  v_k^{j}}_{(m)}$  for any $j,\, k $, which by~\eqref{stripest} gives
	that $   \mn{\Phi  v^{j}}_{(m)}$ is uniformly bounded for any $ j $. 
	This also holds for $  | w - w_0| < c  $ for some $ c > 0 $.
	
	By the Arzela-Ascoli theorem there exists a  subsequence $ \{ v^{j_k}\}_{j_k} $ that converges in $ C^\infty $  to a real valued limit $ v $ on $\Phi ^{-1}(1)  $, i.e., when $ |x| \le 1/2 $, as $ j_k \to \infty $.  By taking the limit of the equation~\eqref{plinequation0} we find by continuity for large enough $ \varrho $ that 
	\begin{equation}\label{limequation}
	P_0(\partial v(x, w),x, w,\partial) v(x, w) = f_0(x, w) 
	\end{equation}
	when $ | x| < \delta_0/2 $ and $ | w - w_0| < c $.
	We also obtain  that
	$  v(x, w) = c_{x, \varrho}(f_0)  \in \br $ and  $\partial_x  v(x, w) = d_{x, \varrho}(f_0)  \in \br^n $ where $ \varrho c_{x, \varrho}  $ and  $ d_{x, \varrho}  \in  \Cal E'$ are $ o(1)$ uniformly for  $ \varrho\to \infty $ when $ |x| \le \delta_0/2 $ and  $  | w - w_0| < c  $. Thus, for $ \varrho \ge \varrho_{1} $ large enough, we find that $ P_0(\Phi \partial v(x, w),x, w,\partial) $ is of principal type when $ | x | \le \delta_0/2 $  and $ | w - w_0| < c $.
	
	Thus, if  $ \varrho  \ge \max( \varrho_{0},\varrho_{1},\varrho_{\delta}) $ we have a solution to~\eqref{limequation} when $ | x| < \delta_0/2 $   and $ | w - w_0| < c $ with 
	$ f $ replaced by  $\Phi_{\delta_0}f_0 =  \Phi_{\delta_0}\left(f - (p_1 + p_0 x)\cdot u_1 - p_0 u_0 \right) $ by~\eqref{pmodequation} and 
	$ v $ replaced by $u = v + u_0 + u_1\cdot x $ by~\eqref{reduc} so that $ P_0(\Phi \partial u(x, w),x, w,\partial) $ is of principal type when $ | x | \le \delta_0/2 $  and $ | w - w_0| < c $. 
	By linearity we obtain that
	$$  
	\left\{
	\begin{aligned}
	&u(0, w_0)  = u_0 +  c_{ \varrho}(\Phi_{\delta_0}f) - c_{ \varrho}(\Phi_{\delta_0}p_0)u_0 -   c_{ \varrho}(\Phi_{\delta_0}(p_1 + p_0 x))\cdot u_1\\
	&\partial_x  u(0, w_0)  = u_1 +  d_{ \varrho}(\Phi_{\delta_0}f) - d_{ \varrho}(\Phi_{\delta_0}p_0)u_0 -   d_{ \varrho}(\Phi_{\delta_0}(p_1 + p_0x))\cdot u_1
	\end{aligned}
	\right.
	$$  
	If we replace $ u_j $ with indeterminate $ v_j $ for $ j = 1, \,2 $, then we obtain the linear system
	\begin{equation}\label{initrho}
	\left\{
	\begin{aligned}
	&u(0, w_0) = (1 +  o(\varrho^{-1})a_\varrho) v_0  +  o(\varrho^{-1})b_\varrho \cdot v_1 + o(\varrho^{-1})c_\varrho = u_0\\
	&\partial_x u(0, v_0) = o(1)d_\varrho v_0 +(\id_n  +  o(1)e_\varrho) v_1  + o(1)f_\varrho = u_1
	\end{aligned}
	\right.
	\end{equation}
	where the coefficients $ a_\varrho, \dots, f_\varrho $ are uniformly bounded when $ \varrho \ge 1 $. Observe that $ b_\varrho $ is a $ 1 \times n $, $ d_\varrho $  and $ f_\varrho $ are $ n \times 1 $ and $ e_\varrho $ is an $ n \times n $ matrix.
	This is a linear $ (n+1) \times( n+1) $ system in $ (v_0, v_1) $ which converges to the identity when $ \varrho \to \infty $. 
	Thus, there exists $ \varrho_{2} \ge 1$ so that the determinant of the system is nonvanishing when $ \varrho  \ge \varrho_{2} $ which gives a unique solution  $ (v_0, v_1) $   to~\eqref{initrho} so that $ |   v_j -   u_j | \le o(1)$, $j  = 1$, 2, as $ \varrho \to \infty $. 
	By solving~\eqref{pmodequation} for $ \varrho  \ge  \max(\varrho_{0},\varrho_{1},\varrho_{2},\varrho_{\delta}) $ large enough with $ u_j $ replaced by $ v_j $, $ j=1,\, 2 $,  we obtain  a solution to~\eqref{pequation}  when $ | x | \le \delta_0/2 $  and $ | w - w_0| < c $. In fact, if  $ \varrho $ is large enough then $ P_0(\Phi \partial u(x, w),x, w,\partial) $ is of principal type when $ | x | \le \delta_0/2 $ and $ | w - w_0| < c $. This finishes the proof of Theorem~\ref{appthm}.
\end{proof}

\section*{Acknowledgement} This project was partly funded by the Swedish Research Council grant 2018-04228.


\bibliographystyle{plain}

\end{document}